\documentclass[12pt,oneside]{amsart}

\usepackage{geometry}   
\usepackage{todonotes}
\usepackage{tikz-cd}
\usepackage{color}
\usepackage{mathrsfs}
\usepackage{enumerate}
\usepackage{bm}
\usepackage{bbm}
\usepackage{wasysym}
\usepackage[utf8]{inputenc}
\inputencoding{utf8}
\usepackage{comment}
\usepackage{graphicx}
\usepackage{amsthm,amssymb,amsmath}
\usepackage{subfigure}
\allowdisplaybreaks[4]

\newtheorem{theorem}{Theorem}[section]
\newtheorem{conjecture}[theorem]{Conjecture}
\newtheorem{lemma}[theorem]{Lemma}
\newtheorem{proposition}[theorem]{Proposition}
\newtheorem{corollary}[theorem]{Corollary}
\newtheorem{fact}[theorem]{Fact}
\theoremstyle{definition}

\numberwithin{equation}{section}

\def\d{\,\mathrm{d}}

\def\RR{\mathbb{R}}
\def\TT{\mathbb{T}}
\def\ZZ{\mathbb{Z}}

\def\BM{\mathrm{BM}}
\def\p{\!\cdot\!}
\def\id{\mathrm{id}}
\def\Vol{\mathrm{Vol}}
\def\exp{\mathrm{exp}}
\def\pr{\mathrm{p}}

\title{A nonabelian Brunn--Minkowski inequality}

\author{Yifan Jing}
\address{Mathematical Institute, University of Oxford, Oxford, UK}
\email{yifan.jing@maths.ox.ac.uk}

\author{Chieu-Minh Tran}
\address{Department of Mathematics, National University of Singapore, Singapore}
\email{trancm@nus.edu.sg}

\author{Ruixiang Zhang}
\address{Department of Mathematics, University of California - Berkeley, CA, USA}
\email{ruixiang@berkeley.edu}

\thanks{YJ was supported by  Arnold O. Beckman Research Award (Campus Research Board RB21011), by the Department Fellowship and the Trjitzinsky Fellowship from UIUC, and Ben Green’s Simons Investigator Grant, ID:376201}
\thanks{CMT was supported by Anand Pillay's NSF Grant-2054271.}
\thanks{RZ was supported by the NSF grant DMS-1856541, the Ky Fan and Yu-Fen Fan Endowment Fund at the Institute for Advanced Study and the NSF grant DMS-1926686}

\date{}
\begin{document}

\begin{abstract}
Henstock and Macbeath asked in 1953 whether the Brunn--Minkowski inequality can be generalized to nonabelian locally compact groups; questions along the same line were also asked by Hrushovski, McCrudden, and Tao. We obtain here such an inequality and prove that it is sharp for helix-free locally compact groups, which includes real linear algebraic groups, Nash groups, semisimple Lie groups with  finite center, solvable Lie groups, etc. The proof follows an induction on dimension strategy;  new ingredients include an understanding of the role played by maximal compact subgroups of Lie groups, a necessary modified form of the inequality which is also applicable to nonunimodular locally compact groups, and a  proportionated averaging trick. 
\end{abstract}

\maketitle

\tableofcontents

\section{Introduction}

\subsection{Background} Let  $\mu$  be the usual  Lebesgue measure  on $\RR^{d}$, let $X$ and $Y$ be nonempty and compact subsets of $\RR^{d}$, and set $X+Y := \{x+y: x \in X, y \in Y\}$. The Brunn--Minkowski inequality says that 
\begin{equation}\label{eq: BM for unimodular}
    \mu(X+Y)^{1/d} \geq  \mu(X)^{1/d} + \mu(Y)^{1/d}. 
\end{equation}
For fixed $\mu(X)$ and $\mu(Y)$, the inequality provides us with the minimum value of $\mu(X+Y)$ which is obtained, for example, when $X$, $Y$, and $X+Y$ are $d$-dimensional hypercubes with side length $\mu(X)^{1/d}$, $\mu(Y)^{1/d}$, and $\mu(X)^{1/d} + \mu(Y)^{1/d}$, respectively.

Under the further assumption that $X$ and $Y$ are convex, the inequality in an equivalent form
was proven by Brunn~\cite{Brunn} in 1887. In the celebrated Geometrie der Zahlen (Geometry of Numbers)~\cite{GeometryOfNumbers} published in 1896, Minkowski introduced the current form of the inequality and established that the equality occurs in \eqref{eq: BM for unimodular} if and only if $X$ and $Y$
are homothetic convex sets.
Lyusternik~\cite{Lyusternik} removed the convexity assumption in 1935. However, his proof that the same condition for equality still holds contained some errors, a situation eventually corrected by Henstock and Macbeath~\cite{HenstockMacbeath} in 1953. The Brunn--Minkowski inequality is widely considered a cornerstone of convex geometry. See~\cite{Gardner} for an excellent survey on its numerous generalizations and applications.

In this paper, we consider the problem of generalizing the Brunn--Minkowski inequality to a locally compact group $G$. Here, up to  positive constant factors, we have a unique left Haar measure $\mu$ generalizing the Lebesgue measure in $\RR^d$; see Appendix B for the precise definitions.

We temporarily further assume that $\mu$ is also invariant under the right translations. Such $G$ is called  {\it unimodular}.  This assumption holds when $G=\RR^d$ and in  many other situations (e.g., when $G$ is compact, discrete, a nilpotent Lie group, a semisimple Lie group, etc). Set $XY=\{xy: x \in X, y \in Y\}$ for nonempty compact $X, Y\subseteq G$. The translation invariance property of  $\mu$ implies that $$\mu(XY) \geq \max\{\mu(X), \mu(Y)\}$$ and should intuitively be even larger, hinting at a meaningful generalization of the Brunn--Minkowski inequality to this setting. This will be shown to  be the case.

 For an arbitrary locally compact group $G$, $\mu$ might no longer be right invariant. Hence, we still have $\mu(XY) \geq \mu(Y)$, but we might have $\mu(XY) < \mu(X)$.  By a result of Macbeath~\cite{Macbeath60} in 1960,  the trivial inequality $\mu(XY) \geq \mu(Y)$ for nonunimodular $G$ is already sharp in the sense that for any $\alpha, \beta, \varepsilon >0$, there are nonempty compact $X,Y \subseteq G$ with  
$$\mu(X)=\alpha, \mu(Y) = \beta, \ \text{and}\ 
\mu(XY)< \mu(Y) +\varepsilon.$$
We will later see in this paper  that there is still a meaningful generalization of the Brunn--Minkowski inequality involving both $\mu$ and a right Haar measure $\nu$. Surprisingly, it turns out that if one only cares about unimodular cases, the nonunimodular cases are still needed for our proof. We will keep the setting and notation of this paragraph throughout the rest of the paper.

The problem of generalizing the Brunn--Minkowski inequality was proposed  in 1953 by  Henstock and Macbeath~\cite{HenstockMacbeath}; different variations of this problem were also later suggested by Hrushovski~\cite{HrushovskiPC}, by McCrudden~\cite{Mccrudden69}, and by Tao~\cite{TaoBlog}. 
In the direction of the intuition described earlier, Kemperman~\cite{Kemperman} showed in 1964 that $\mu(XY) \geq \mu(X)+\mu(Y)$ when $G$ is connected, unimodular and noncompact. Even more important for us is the generalization~\eqref{eq: nonunimodular by Kemperman}, also in the same paper~\cite{Kemperman}, to all connected noncompact locally compact groups:
\begin{equation}\label{eq: nonunimodular by Kemperman}
\frac{\nu(X)}{\nu(XY)}+ \frac{\mu(Y)}{\mu(XY)}\leq1.
\end{equation}

While applicable to all locally compact groups, Kemperman's inequalities are not sharp: even for $\RR^2$, they give a weaker conclusion than the Brunn--Minkowski inequality. The most definite result toward the correct lower bound was obtained by McCrudden~\cite{Mccrudden69} in 1969.  In effect, he showed that when $G$ is a unimodular solvable Lie group of dimension $d$, and $m$ is the dimension of the maximal compact subgroup, we have 
\begin{equation}\label{eq: McCrudden solvable BM ineq}
 \mu(XY)^{1/(d-m)} \geq  \mu(X)^{1/(d-m)} + \mu(Y)^{1/(d-m)}. 
 \end{equation}
The above differs from McCrudden's original statement in that $m$ was defined using an inductive idea in~\cite{Mccrudden69}; the form in~\eqref{eq: McCrudden solvable BM ineq} is more suitable to get the later generalizations (Theorems~\ref{thm: main Lie} and~\ref{thm: main}) and to show that it is indeed sharp (Theorem~\ref{thm: mainsharp}). A number of special cases of this result were rediscovered by Gromov~\cite{Gromov}, by Hrushovski~\cite{HrushovskiLieModel}, by Leonardi and Mansou~\cite{LeoMasnou}, and by Tao~\cite{TaoBlog}. Sharpness for nilpotent groups was essentially proven by Monti~\cite{Monti}; see also Tao~\cite{TaoBlog}.

\subsection{Statement of main results} 

Suppose $G$ is a Lie group with connected component $G_0$. Following Levi decomposition (Fact~\ref{fact: Lie group decomp Levi}), we have an exact sequence of Lie groups
$$ 1 \to Q  \to G_0 \to S \to 1 $$
where $Q$ is solvable and $S$ is semisimple. It is known that the center $Z(S)$ is a finitely generated abelian group of rank $h$; see Lemma~\ref{lem: helixandnoncompact} and Facts~\ref{fact: centerless1}, \ref{fact: centerless2}. We call $h$ the {\bf helix dimension} of $G$.  As an example, $\text{SL}_2(\RR)$ has helix dimension $0$ while its universal cover has helix dimension $1$. If $h=0$, equivalently $S$ has finite center, we say that $G$ is {\bf helix-free}. Real linear algebraic groups and more generally, Nash groups (equivalently, semialgebraic Lie groups or groups definable in the field of real numbers) are helix free; see~\cite[Lemma 4.5]{BJO} and the subsequent discussion in the same paper.
Our first main result is a generalization of Brunn--Minkowski inequality to Lie groups whose exponent will be seen to be sharp for helix-free Lie groups:

\begin{theorem}\label{thm: main Lie}
Suppose $G$ is a Lie group, $\mu$ is a left Haar measure, $\nu$ is a right Haar measure, the dimension of $G$ is $d$, the maximal dimension of a compact subgroup of $G$ is $m$, the helix dimension of $G$ is $h$, and $X,Y$ are compact subsets of $G$ with positive measure.  Then
\begin{equation}\label{eq: BM for nonunimodular}
\frac{\nu(X)^{1/(d-m-h)}}{\nu(XY)^{1/(d-m-h)}}+ \frac{\mu(Y)^{1/(d-m-h)}}{\mu(XY)^{1/(d-m-h)}}\leq1;
\end{equation}
the left-hand-side is interpreted as $\max\{ \nu(X)/\nu(XY) , \mu(Y)/\mu(XY)\}$ if $d-m-h=0$.
In particular,  if $G$ is unimodular, then $   \mu(XY)^{\tfrac{1}{d-m-h}} \geq \mu(X)^{\tfrac{1}{d-m-h}} + \mu(Y)^{\tfrac{1}{d-m-h}}. $
\end{theorem}

 We require that  $X$ and $Y$ be compact in our results since there are measurable sets $X$ and $Y$ such that $XY$ is not measurable. 
  By using the regularity property of Haar measure, the conclusions in our main theorems still hold for measurable $X$ and $Y$ if we replace $\mu(XY)$ and $\nu(XY)$ with inner Haar measures.

Now consider an arbitrary locally compact group $G$. Using the Gleason--Yamabe Theorem (Fact~\ref{fact: Gleason}), one can choose an open subgroup $G'$ of $G$ and a normal compact subgroup $H$ of $G'$ such that $G'/H$ is a Lie group. It is shown in Proposition~\ref{prop: welldefinednomcompactdim} that 
\[
n = \dim(G'/H) - \max\{ \dim(K) : K \text{ is a compact subgroup of } G'/H\} 
\]
is independent of the choice of $G'$ and $H$ satisfying the above properties.
We call $n$ the {\bf noncompact Lie dimension} of $G$. Let $Q$ be the radical (i.e., the maximal connected closed solvable normal subgroup, see Fact \ref{fact: radical}
) of $G'/H$. Note that $(G'/H)_0/Q$ has discrete center $Z(G'/H)_0/Q$ by Facts~\ref{fact: centerless1} and \ref{fact: centerless2}. We call
\[
h = \mathrm{rank}(Z((G'/H)_0/Q))
\]
the {\bf helix dimension} of $G$. We will also show that the helix dimension $h$ of $G'/H$ is independent of the choice of $G'$ and $H$ in Proposition~\ref{prop: welldefinednomcompactdim}. Here is our second main result.
\begin{theorem}\label{thm: main}
Suppose $G$ is a locally compact group with noncompact Lie dimension $n$ and helix dimension $h$,  $\mu$ is a left Haar measure, $\nu$ is a right Haar measure, and $X,Y$ are compact subsets of $G$ with positive measure. Then 
\begin{equation}\label{eq: BM in theorem1.2}
\frac{\nu(X)^{1/(n-h)}}{\nu(XY)^{1/(n-h)}}+ \frac{\mu(Y)^{1/(n-h)}}{\mu(XY)^{1/(n-h)}}\leq1;
\end{equation}
the left-hand-side is interpreted as $\max\{ \nu(X)/\nu(XY) , \mu(Y)/\mu(XY)\}$ when $n-h=0$.
In particular, if $G$ is unimodular, then $ \mu(XY)^{\tfrac{1}{n-h}} \geq  \mu(X)^{\tfrac{1}{n-h}} + \mu(Y)^{\tfrac{1}{n-h}}. $
\end{theorem}
 When $G$ is as in Theorem~\ref{thm: main Lie}, the noncompact Lie dimension  $n$ is simply $d-m$, so Theorem~\ref{thm: main} is a generalization of Theorem~\ref{thm: main Lie}. 
 
Our last main result tells us that when $G$ is helix-free, the exponent $1/(n-h)= 1/n$ in Theorem~\ref{thm: main Lie} and Theorem~\ref{thm: main} are sharp even when we assume further that $X=Y$. As usual in the current setting, we write $X^k$ for the $k$-fold product of $X$.

\begin{theorem}\label{thm: mainsharp}
 Suppose $G$ is a locally compact group with noncompact Lie dimension $n$,  $\mu$ is a left Haar measure, and $\nu$ is a right Haar measure. Then
\begin{enumerate}
    \item  When $n=0$, there is a compact set $X$ with positive left and right measure in $G$ such that $\mu(X^2) = \mu (X)$ and $\nu(X^2) = \nu(X)$.
    \item When $n > 0$, for every $\varepsilon > 0$, there is a  compact set $X$ with positive left and right measure in $G$ such that
\begin{equation}\label{eq: BM in theorem1.3}
\frac{\nu(X)^{\frac{1}{n}-\varepsilon}}{\nu(X^2)^{\frac{1}{n}-\varepsilon}}+ \frac{\mu(X)^{\frac{1}{n}-\varepsilon}}{\mu(X^2)^{\frac{1}{n}-\varepsilon}}>1.
\end{equation}
\end{enumerate}

As a corollary, if $G$ is unimodular with $n>0$, for every $\varepsilon'>0$, there is a compact set $X$ in $G$ such that $\mu(X^2) < (2^n+\varepsilon') \mu(X)$.
\end{theorem}

The upper bound given in Theorem~\ref{thm: mainsharp} matches the lower bound given in Theorem~\ref{thm: main} when the group is helix-free, that is a group has helix dimension $0$, which essentially means the semisimple part of the group has  finite center.  Hence, for these groups,  our theorems resolve the problem of generalizing the Brunn--Minkowski inequality. 

We believe that the exponent in Theorem~\ref{thm: mainsharp} should be correct for all locally compact groups, which is made precise by the following conjecture:


\begin{conjecture}[Nonabelian Brunn--Minkowski Conjecture]\label{conj: main}
Suppose $G$ is a locally compact group with noncompact Lie dimension $n$,  $\mu$ is a left Haar measure, $\nu$ is a right Haar measure, and $X,Y$ are compact subsets of $G$ with positive measure. Then 
\begin{equation*}
\frac{\nu(X)^{1/n}}{\nu(XY)^{1/n}}+ \frac{\mu(Y)^{1/n}}{\mu(XY)^{1/n}}\leq1;
\end{equation*}
the left-hand-side is interpreted as $\max\{ \nu(X)/\nu(XY) , \mu(Y)/\mu(XY)\}$ when $n=0$.
\end{conjecture}

While we do not know whether the exponent in Theorem~\ref{thm: main} is sharp, it does have the correct order of magnitude. This is because the helix dimension
$h$ of $G$ is always at most $n/3$, where $n$ is the noncompact Lie dimension of $G$; see Corollary~\ref{cor: n/3}.  In particular, if there is a compact subset $X$ of a connected noncompact unimodular group $G$ with small $\mu(X^2)/\mu(X)$, Theorem~\ref{thm: main} still allows us to deduce that $G$ has a compact subgroup of small codimension, which
constrains the structure of $G$. In an upcoming paper~\cite{AJTZ} by An and the authors, we apply Theorem~\ref{thm: main} to resolve the last open case of the Kemperman Inverse problem, also answering the connected case of a question by Tao. Recently, Fanlo and Hrushovski similarly use Theorem~\ref{thm: main} to obtain a structural result for metric approximate groups~\cite{Fanlo, FanloHru}.

The next result shows that one can reduce Conjecture~\ref{conj: main} to simply connected simple Lie groups. As a consequence of this, the only remaining cases are what one might regard initially as the simplest cases.

\begin{theorem} \label{thm: mainreduction}
Suppose the nonabelian Brunn--Minkowski conjecture holds for all simply connected simple Lie groups, then it holds for all locally compact groups. 
\end{theorem}

\subsection{Further remarks}

The results of this paper continue a line of work by the first two authors \cite{JT20} on small measure expansions in locally compact groups. Through classifying groups $G$ and compact subsets $X$ and $Y$ of $G$ with nearly minimal expansion, it is shown there that when $G$ is a compact semisimple Lie group and $\mu(X)$ sufficiently small,
$$ \mu(X^2) >  (2+c)\mu(X)$$
for a positive constant $c$ independent of the choice of $G$. 
It is conjectured by Breuillard and Green~\cite[Problem 78]{BenBook} that $\mu(X^2)>3.99\mu(X)$ when $G=\mathrm{SO}_3(\RR)$ and $X$ sufficiently small. 
The following corollary of Theorem~\ref{thm: main Lie} 
resolves the counterpart of the aforementioned conjecture for all noncompact simple Lie group $G$. 
\begin{corollary}\label{cor: 1.6BG}
Let $G$ be a noncompact semisimple Lie group, and let $X\subseteq G$ be compact. Then $\mu(X^2)\geq 4\mu(X)$. 
\end{corollary}

Equality can occur in~\eqref{eq: BM for nonunimodular} and~\eqref{eq: BM in theorem1.2} of Theorems~\ref{thm: main Lie} and~\ref{thm: main}, for example, when $G=\RR^d$ for $d\geq1$ as described earlier. 
In fact, the classical Brunn--Minkowski inequality in $\RR^d$ comes along with a complete classification of when equality occurs. For connected locally compact abelian group $G$ with noncompact Lie dimension $1$ (automatically has helix dimension $0$ by Corollary~\ref{cor: n/3}), it follows from the work of Kneser~\cite{kneser} that equality occurs in~\eqref{eq: BM in theorem1.2} exactly when there is a continuous surjective group homomorphism $\chi:G\to\RR$ with the compact kernel, and $X$, $Y$ are preimages of compact intervals in $\RR$. The same conclusion is recently obtained by An and the authors~\cite{AJTZ} for general locally compact groups, using the results of the current paper as ingredients.

On the other hand, a result by McCrudden~\cite{MccruddenExample} shows that when $G$ is the Heisenberg group, equality cannot occur. It would be interesting to determine those groups with helix dimension $0$ and compact subsets inside them where equality can occur. This question can also be asked without the assumption of helix dimension $0$ once Conjecture~\ref{conj: main} is resolved. 
It would also be interesting to understand when equality nearly happens and develop a theory similar to that of Christ, Figalli, and Jerison~\cite{christ2012near,FigaliJerisonearlier, Figalli} for $\RR^d$.

Like the Brunn--Minkowski inequality for $\RR^d$, our results do not rely on the normalization of Haar measures. However, by fixing a Haar measure $\mu$ on a unimodular group $G$, it would be interesting to determine the value of
\[
\min \{ \mu(XY) : X,Y \subseteq G \text{ are compact}, \mu(X) = \alpha, \mu(Y)=\beta  \},
\]
for given $\alpha,\beta \in \RR^{>0}$, and to classify the situations where equality happens. We do not pursue this question here.

\subsection{Overview of the proof}
In this subsection, we discuss the idea of the proof of the main results and the organization of the paper. For expository purposes, we restrict our attention to  helix-free locally compact groups,  where we can fully prove Conjecture \ref{conj: main}. The proof of the full versions of Theorems~\ref{thm: main Lie} and \ref{thm: main} requires a more involved discussion of the helix dimension, which is developed in Section 2.

In each of our results, the exponent of the inequalities is controlled by the noncompact Lie dimension $n$ of $G$ instead of just its topological dimension $d$ as in the simpler versions for $\RR^d$.  Recall that, for a Lie group $G$, $n=d-m$  where $m$ is the maximum dimension of a compact subgroup of $G$.  The proof of Theorem~\ref{thm: mainsharp} explains the critical role of $m$: our construction is essentially a small neighborhood of a compact subgroup of $G$ having maximal dimension, see Figure~\ref{fig: example}. 
One may conjecture that the above construction is optimal. Theorems~\ref{thm: main Lie}  and \ref{thm: main} confirm this intuition for helix-free groups.

\begin{figure}[h]
\centering
    \includegraphics[width=2.6in]{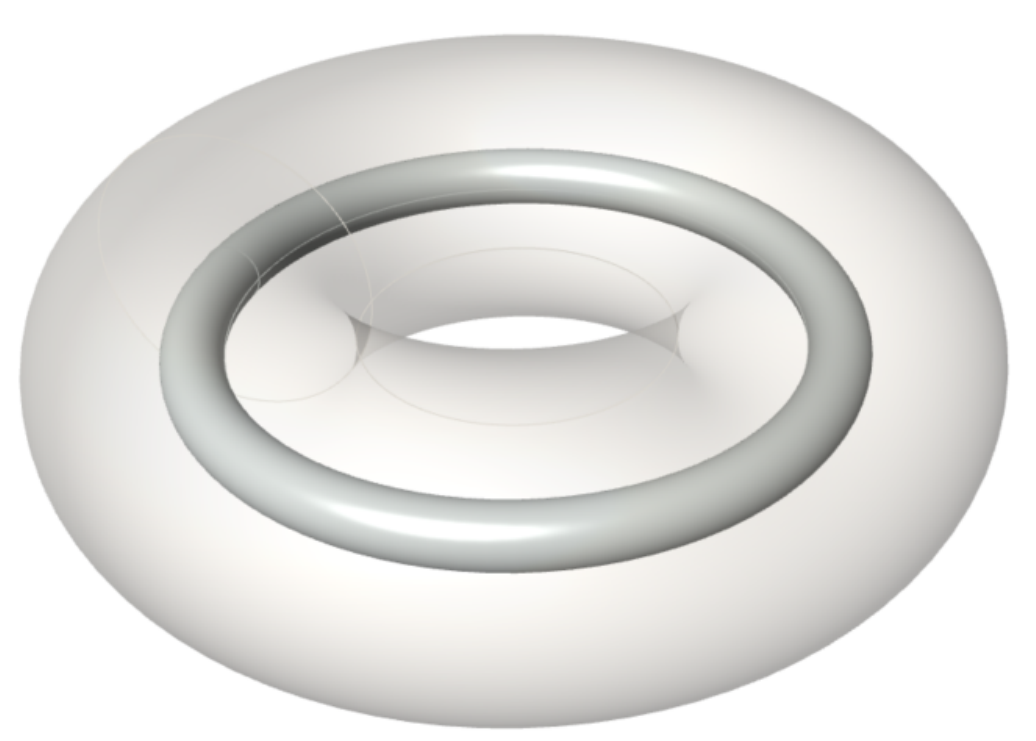}
    \caption{Let $G=\mathrm{SL}(2,\mathbb{R})$ (the open region bounded by the outer torus), and let $K=\mathrm{SO}(2,\RR)$ be the maximal compact subgroup of $G$. If we take $X$ to be a small closed neighborhood of $K$ (closed region bounded by the shaded inner torus), Theorem~\ref{thm: mainsharp} says when $X$ is sufficiently small, $\mu_G(X^2)$ will be very close to $4\mu_G(X)$ instead of $8\mu_G(X)$, although $G$ has topological dimension $3$.}
    \label{fig: example}
\end{figure}

To motivate our proofs of Theorems~\ref{thm: main Lie}  and \ref{thm: main}, we first recall some proofs of the known cases of the Brunn--Minkowski inequality.   Over $\RR^d$, the usual strategy is to induct on dimensions. This is generalized by McCrudden~\cite[Theorem 1.2]{Mccrudden69} to obtain the following ``unimodular exponent splitting'' theorem: Given an exact sequence of \emph{unimodular} locally compact groups
$$ 1 \to H \to G \to G/H \to 1, $$
if $H$ and $G/H$ satisfy  Brunn--Minkowski inequalities with exponents $1/n_1$ and $1/n_2$, respectively, then the group $G$ satisfies a Brunn--Minkowski inequality with exponent $1/(n_1+n_2)$.  McCrudden applied his unimodular exponent splitting theorem to obtain the Brunn--Minkowski inequality for unimodular solvable groups with sharp exponents. A simpler proof of his theorem is given in Section 4 for completeness. In the proof of our main theorems, one important ingredient will be an exponent splitting result (that is a generalization of his). 

We now sketch a proof of McCrudden's theorem and introduce a ``spillover'' argument which we will also use later on; some of the ideas presented here are also available in McCrudden's original approach.
For each $g$ in $G$, we call $X \cap gH$  a \emph{fiber} of $X$, and refer to the Haar measure of $g^{-1}X\cap H$ in $H$ as its \emph{length}. Let $\pi:G\to G/H$ be the quotient map. We now partition $X$ and $Y$ each into $N$ parts. Suppose $X=\bigcup_{i=1}^N X_i$ and $Y=\bigcup_{i=1}^N Y_i$, we require that the images under $\pi$ of the $X_i$'s are pairwise disjoint, the shortest fiber-length in each $X_i$ is at least the longest fiber-length in $X_{i-1}$, and likewise for the $Y_i$'s.

The induction hypotheses, i.e., the Brunn--Minkowski inequalities, 
in $H$ and $G/H$ give us a lower bound $l_N$ on fiber-lengths in $X_NY_N$ and a lower bound $w_N$ on the size of  $\pi(X_NY_N)$ in $G/H$.  Their product $l_Nw_N$ is a lower bound for $\mu(X_N Y_N)$.  Next we consider $(X_{N-1}\cup X_N)(Y_{N-1}\cup Y_N)$.  Again a lower bound $l_{N-1}$ on fiber lengths in this set and a lower bound $w_{N-1}$ on the size of its image under $\pi$ can be obtained from the induction hypotheses on $H$ and $G/H$. From our method, we have $l_{N-1} \leq l_N$ and $w_{N-1} \geq w_N$. Then $l_{N-1}w_{N-1}$ will be a weak lower bound for $\mu((X_{N-1}\cup X_N)(Y_{N-1}\cup Y_N))$ since the fibers in $X_NY_N$ are ``exceptionally long''. Taking all of these into account, a stronger lower bound is 
$$  l_Nw_N +l_{N-1}(w_{N-1} -w_N).   $$
Repeating the above process, we get that $\mu(XY)=\mu((\bigcup_{i=1}^NX_i)(\bigcup_{i=1}^N Y_i))$ is at least
\[
l_Nw_N+ \sum_{i=1}^{N-1}l_{N-i}(w_{N-i}-w_{N-i+1}).
\]
By choosing $X_i$ and $Y_i$ ``proportionately'' to their fiber lengths, applying suitable H\"older's inequalities, and taking the limit $N \rightarrow \infty$, we have the ``spillover'' argument which yields McCrudden's theorem.

Induction via the exact sequence method completely stops working 
when one is looking to prove Brunn--Minkowski for simple groups since there is no nontrivial closed normal subgroup to induct from. Next, we explain how we overcome this main difficulty. Our method  turns out to work also for semisimple groups in the same way and we will explain it in this more general setting. 

Let us assume $G$ is a connected semisimple Lie group with the finite center (hence helix-free and automatically unimodular) and think about how we can prove the special case of Theorem~\ref{thm: main Lie} for it. One can consider the Iwasawa decomposition $G=KAN$ where $K$ has a compact Lie algebra and $Q=AN$ is solvable and try to prove the Brunn--Minkowski inequality for $G$ from a suitable Brunn--Minkowski-type inequality for $Q$. However, $Q$ may not be unimodular in general, so the classical form of the Brunn--Minkowski inequality does not apply to $Q$. Let $\Delta_Q$ be the modular function on $G$. One can compromise by choosing $Q'=\ker(\Delta_Q)$, which is unimodular, and try to use the known Brunn--Minkowski inequality for $Q'$ to prove the Brunn--Minkowski inequality on $G$. This is indeed a good direction to go but along this direction one inevitably gives up on the sharp exponent $1/n$ and can at best prove a weaker inequality with the worse exponent $1/(n-1)$, where $n$ is the noncompact Lie dimension of $G$.

 Because of this, we need an analog of \eqref{eq: BM for unimodular} for nonunimodular groups. We propose the inequality \eqref{eq: BM for nonunimodular}, which seems to be new in the literature. To prove \eqref{eq: BM for nonunimodular} for $AN$, we need a nonunimodular exponent splitting result for the exact sequence $1\to \ker(\Delta_G)\to G\to G/\ker(\Delta_G)\to 1$ coming out from the modular function $\Delta_G: G\to\RR^{>0}$.  It turns out that the spillover method can also be used to reduce the problem to the case where the modular function is almost constant on $X$ and $Y$. We work this out in Lemma~\ref{lem: similar modular value}. In the next more involved step in the same section, we obtain an approximate version of McCrudden's theorem (Proposition~\ref{prop: reduce to unimodular}), which involves another use of the  spillover method, to finish off the proof.

In the next crucial step, we prove that the Brunn--Minkowski for a semisimple $G$ follows from \eqref{eq: BM for nonunimodular} for the solvable $AN$. Our method was motivated by a recent paper~\cite{JT20} by the first two authors, which characterizes nearly minimal expansion sets. Over there, a key idea is to choose a fiber $f$ uniformly at random in $Y$ and use $Xf$ to estimate $XY$. For our current proof, we also choose two fibers $f_X$ and $f_Y$ randomly from $X$ and $Y$, but with respect to two carefully chosen probability measures $\pr_X$ and $\pr_Y$ that are in general nonuniform. We show that by constructing $\pr_X$ and $\pr_Y$ based on the structure of $X$ and $Y$, $\mu(XY)$ can be estimated by the expected size of $f_Xf_Y$ in $AN$, and the latter is well controlled by the Brunn--Minkowski inequality~\eqref{eq: BM for nonunimodular} for $AN$. This part is done in Section 6. It is worth noting that in this case, our inequality matches the upper bound construction when the semisimple group has a finite center.

 With the above preparation, we now explain how we prove Brunn--Minkowski for a general helix-free Lie group $G$. By Proposition~\ref{prop: Reduceinequalitytoopensubgroups}, we can replace $G$ with its identity component and thus assume $G$ is connected. Next, we use Proposition~\ref{prop: reduce to unimodular} to replace $G$ with $\ker(\Delta_G)$ and reduce the problem to the case where $G$ is unimodular. 
  Such $G$  can be decomposed into a semi-direct product of a unimodular solvable group $Q$ and a semisimple group $S$ via the Levi decomposition. We already know how to handle $S$ from the discussion in Section 6. McCrudden's
 theorem  can then be used to deal with $Q$ and to deduce the desired inequality for $G$.

In many of our reductions, we have an exact sequence of groups $1 \to H \to G\to G/H \to 1$ and want to deduce the Brunn--Minkowski for $G$ from the Brunn--Minkowski for $H$ and $G/H$. One tricky issue is that this inductive method only gives sharp results if the sum of the noncompact Lie dimensions and helix dimensions of $H$ and $G/H$ is equal to the noncompact Lie dimension of $G$. Unfortunately, this is not always true (see the examples after the proof of Lemma~\ref{lem: IWasawaexact}). With this warning in mind, we must ensure the above property is always satisfied in the whole reduction. Our discussion in Section 2 guarantees this.

In the remaining part, we discuss some new challenges in the proof of Theorem~\ref{thm: main} for a helix-free locally compact group $G$.
The Gleason--Yamabe Theorem tells us that $G$ contains an open subgroup $G'$ that has a Lie quotient $G'/H$ with $H$ compact. 
For the start, we need to handle the nonuniqueness in the choice of $G'$ and $H$ in order to show that noncompact Lie dimension and helix dimension are well-defined for locally compact groups. This is also done in Section 2, making heavy use of Lie theory and the Gleason--Yamabe Theorem.

Unlike the Lie group case discussed earlier, we cannot directly replace $G$ by its identity component, as the latter may not be open. We must take the alternating path of deducing the Brunn--Minkowski inequality for $G$ from its open subgroup $G'$ given by the Gleason--Yamabe theorem and its Lie quotient $G'/H$. When $G$ is unimodular, this can be achieved using Proposition~\ref{prop: reducetoopensubgroup2} and Proposition~\ref{lem: quotient unimodular}. When $G$ is not unimodular, we encounter an extra technical difficulty as we do not find an easy way to generalize Proposition~\ref{prop: reducetoopensubgroup2}. Instead, we reduce the problem to the unimodular case using Proposition~\ref{prop: reduce to unimodular}, Proposition~\ref{prop: Reduceinequalitytoopensubgroups}, and Lemma~\ref{Lem: dichotomy2} to replace $G$ by $\ker(\Delta_G)$ which is unimodular.





\section{Noncompact Lie dimension and helix dimension}

In this section, we show that noncompact Lie dimensions and helix dimensions are well-defined in locally compact groups and that they behave well in many exact sequences. The latter is the underlying reason that the lower bound in Theorem~\ref{thm: main Lie} and Theorem~\ref{thm: main} matches the upper bound in Theorem~\ref{thm: mainsharp} for helix-free locally compact groups. 

 Throughout this section, all groups are locally compact, and we will use various definitions and facts from Appendices C, D, and E. The following lemma discusses the behavior of Iwasawa decomposition undertaking a quotient by a compact normal subgroup.

\begin{lemma} \label{IWasawaandquotientbycompact}
Suppose $G$ is a connected semisimple Lie group, $H$ is a (not necessarily connected) compact subgroup of $G$. Then we have the following.
\begin{enumerate}
    \item There is an Iwasawa decomposition $G= KAN$ such that $H \leq K$.
    \item Assume further that $H$ is a normal subgroup of $G$, $G= KAN$ is an Iwasawa decomposition such that $H \leq K$, $G'= G/H$, and $\pi: G\to G'$ is the quotient map. Then there is an Iwasawa decomposition $G' = K'A'N'$ such that $\pi(K)=K'$.
\end{enumerate}
\end{lemma}
 \begin{proof}
 
 We first  prove (1). Let $Z(G)$ be the center of $G$, $G'=  G/Z(G)$ and $\rho: G \to G'$ be the quotient map, and $H'=\rho(H)$. By Facts~\ref{fact: centerless1} and \ref{fact: centerless2}, $\rho$ is a covering map and $G'$ is centerless. Let $\mathfrak{g}$ be the common Lie algebra of $G$ and $G'$, and $\exp:\mathfrak{g} \to G$ and $\exp':\mathfrak{g} \to G'$  be the exponential maps.  Using Fact~\ref{fact: Lie group decomp Iwasawa}.2, it suffices to construct a Cartan involution $\tau$ of $\mathfrak{g}$ such that if $\mathfrak{k}$ is the subalgebra of $\mathfrak{g}$ fixed by $\tau$ and $\exp(\mathfrak{k}) = K$, then $H \leq K$. Take a maximal compact subgroup $K'$ of $G'$ that contains $H'$.  Let $\tau_0$ be an arbitrary Cartan involution of $\mathfrak{g}$ (this exists because of Fact~\ref{fact: ExistenceCartaninvolution}). Let $\mathfrak{k}_0$ be the  the subalgebra of $\mathfrak{g}$ fixed by $\tau_0$, and $K'_0= \exp(\mathfrak{k}_0)$ in $G'$. Then by Fact~\ref{fact: Lie group decomp Iwasawa}.2 about Iwasawa decomposition and the earlier observation that $G'$ is centerless, $K'_0$ is a maximal compact subgroup of $G'$. By Fact~\ref{fact: maximal compact}.1 and the assumption that $G$ is connected, there is an automorphism $\sigma'$ of $G'$ such that $\sigma'(K'_0)=K'$. Let $\alpha$ be the automorphism of $\mathfrak{g}$ obtained by taking the tangent map of $\sigma'$, and let
 $$\tau = \alpha \tau_0 \alpha^{-1} \text{ and } \mathfrak{k}= \alpha(\mathfrak{k}_0)$$
  As every Cartan--Killing form is invariant under automorphisms of $\mathfrak{g}$, we get that $\tau$ is a Cartan involution. It is also easy to check that $\mathfrak{k}$ is the subalgebra of $\mathfrak{g}$ fixed by $\tau$. Using the functoriality of the exponential function (Fact~\ref{fact:functoriality of exponential function}), we get $K' =\exp'(\mathfrak{k})$. Now set
 $K= \exp(\mathfrak{k}) $. By Fact~\ref{fact: Lie group decomp Iwasawa}, we get an Iwasawa decomposition $G= KAN$. Therefore, by the functoriality of the exponential function (Fact~\ref{fact:functoriality of exponential function}), $K' =\rho(K)$. Now as $H' \leq K'$, every element of $H$ is in $Z(G)K$. By Fact~\ref{fact: Lie group decomp Iwasawa}.2, we have $Z(G) \subseteq K$, so $H \leq K$ as desired.

 We now prove (2). Set $K'=\pi(K)$.  Let $\mathfrak{g}$, $\mathfrak{h}$, and $\mathfrak{k}$ be the Lie algebras of $G$, $H$, and $K$, and let $\kappa_{\mathfrak{g}}$, $\kappa_{\mathfrak{h}}$, $\kappa_{\mathfrak{k}}$ be the Cartan--Killing form of $\mathfrak{g}$, $\mathfrak{h}$, and $\mathfrak{k}$. Then, $\mathfrak{g}' = \mathfrak{g}/\mathfrak{h}$ is the Lie algebra of $G'$, and $\mathfrak{k}' = \mathfrak{k}/\mathfrak{h}$ is the Lie algebra of $K'$ by Fact~\ref{fact: CorrespondenceLiegroupandalgebra}. Let $\tau$ be a Cartan involution of $\mathfrak{g}$ that fixes $\mathfrak{k}$. We will construct from this a Cartan involution $\tau'$ of $\mathfrak{g}'$ which fixes $\mathfrak{k}'$. If we have done so, then using Fact~\ref{fact: Lie group decomp Iwasawa}, we obtain $A'$ and $N'$ such that $G'=K'A'N'$ is an Iwasawa decomposition, which completes the proof.
 
 Now we construct $\tau'$ as described earlier. As $\mathfrak{g}$ is semisimple, the Lie algebras $\mathfrak{h}$ and $\mathfrak{k}$ are also semisimple. With $\mathfrak{q}$ the orthogonal complement of $\mathfrak{k}$ in $\mathfrak{g}$ with respect to $\kappa_{\mathfrak{g}}$ and $\mathfrak{c}$ the orthogonal complement of $\mathfrak{h}$ in $\mathfrak{k}$ with respect to $\kappa_{\mathfrak{k}}$, we have  $\mathfrak{g} =\mathfrak{k}\oplus \mathfrak{q}$ and  $\mathfrak{k} =\mathfrak{h}\oplus \mathfrak{c}$ by Fact~\ref{fact: orthogonaldecomposition}. By the same fact, with $\kappa_{\mathfrak{q}}$ and $\kappa_{\mathfrak{c}}$ the Cartan--Killing forms of $\mathfrak{q}$  and $\mathfrak{c}$, we have   $\kappa_\mathfrak{g} = \kappa_\mathfrak{k}  \oplus \kappa_{\mathfrak{q}}$ and $\kappa_\mathfrak{k} = \kappa_\mathfrak{h}  \oplus \kappa_{\mathfrak{c}}$. It is then easy to see that every element of $\mathfrak{c} \oplus \mathfrak{q}$ is orthogonal to $\mathfrak{h}$ with respect to $\kappa_{\mathfrak{g}}$. A dimension comparison gives us $\mathfrak{c} \oplus \mathfrak{q} =  \mathfrak{d}$ with $\mathfrak{d}$ the orthogonal complement of $\mathfrak{h}$  in $\mathfrak{g}$. In summary, we have
  \begin{equation} \label{eq: Liedecomposition}
     \mathfrak{g} =  \mathfrak{k}  \oplus \mathfrak{q} =  \mathfrak{h}  \oplus \mathfrak{c}  \oplus \mathfrak{q} =  \mathfrak{h} \oplus \mathfrak{d} \quad  \text{and}  \quad \kappa_\mathfrak{g} =  \kappa_\mathfrak{k}  \oplus \kappa_\mathfrak{q} =  \kappa_\mathfrak{h}  \oplus \kappa_\mathfrak{c}  \oplus \kappa_\mathfrak{q} =  \kappa_\mathfrak{h} \oplus \kappa_\mathfrak{d}.
 \end{equation}

As a particular consequence, the quotient map from $\mathfrak{g}$ to $\mathfrak{g}'$ restricts to isomorphisms of Lie algebras from  $\mathfrak{d}$ to $\mathfrak{g}' = \mathfrak{g}/ \mathfrak{h}$ and from $\mathfrak{c}$ to $\mathfrak{k}' = \mathfrak{k}/ \mathfrak{h}$.  Since $\mathfrak{h}$ is a subalgebra of $\mathfrak{k}$, $\tau$ fixes $\mathfrak{h}$. As Cartan--Killing forms are invariant under automorphisms, $\tau$ restricts to an endomorphism of $\mathfrak{d}$, which is the orthogonal complement of $\mathfrak{h}$ in $\mathfrak{g}$ under $\kappa_\mathfrak{g}$. Therefore, $\tau|_\mathfrak{d}$ is an involution of $\mathfrak{d}$. The bilinear form 
$$\mathfrak{d} \times \mathfrak{d}:  (x, y) \mapsto -\kappa_\mathfrak{d}(x, \tau|_\mathfrak{d}(y))$$
is positive definite as it is simply the restriction to $\mathfrak{d}$ of the positive definite bilinear form $\mathfrak{g} \times \mathfrak{g}:  (x, y) \mapsto -\kappa_\mathfrak{d}(x, \tau(y))$. Hence,  $\tau|_\mathfrak{d}$ is 
a Cartan involution of $\mathfrak{d}$. Recall that the subalgebra of $\mathfrak{g}$ fixed by $\tau$ is $\mathfrak{k}$. By the first part of \eqref{eq: Liedecomposition}, the subalgebra of $\mathfrak{d}$ fixed by $\tau|_\mathfrak{d}$ is $\mathfrak{c}$. Finally, let $\tau'$ be the pushforward of $\tau|_\mathfrak{d}$ under the quotient map from $\mathfrak{g}$ to $\mathfrak{g}'$. It is easy to see that $\tau'$ satisfies the desired requirement.
  \end{proof}

The following lemma allows us to compute the noncompact Lie dimension of the universal cover of a compact Lie group. Recall that a finitely generated abelian group $G$ is isomorphic to a direct sum $\mathbb{Z}^h\oplus G^{\mathrm{tor}}$, with $h\geq 0$ and $G^{\mathrm{tor}}$ finite; the integer $h$ is called the \emph{rank} of $G$.

\begin{lemma} \label{lem: helixandnoncompact}
 Let $K$ be a covering group of a compact Lie group $K'$ with covering map $\rho: K \to K'$, and assume that $K$ and $K'$ are connected. 
Suppose $\ker(\rho)$ is a finitely generated abelian group of rank $h$, and $m$ is the maximum dimension of a compact subgroup of $K$. Then $h= \dim(K)-m$.
\end{lemma}

\begin{proof}
We first consider the case when $K$ is a solvable group. Then $K'\cong \TT^{k}$ where $k$ is the dimension of $K$ by Fact~\ref{fact: classification}.2. Recall that $K$ is a quotient of the universal cover of $K'$, which is $\RR^k$. Hence, $K \cong \RR^h \times \TT^{k-h}$. It is easy to see that the maximum dimension of a compact subgroup of $K$ is $k-h$, which gives us the desired conclusion in this case.

We now prove the statement of the Lemma when $K$ is not solvable. Let $Q_K$ be the radical of $K$, $Q_{K'}$ the radical of $K'$,  $S_K= K/Q_K$, and $S_{K'}= K'/Q_{K'}$. Note that $K$ and $K'$ have the same Lie algebra $\mathfrak{k}$. By Fact~\ref{fact: radical}, $Q_K$ and $Q_{K'}$ have the same Lie algebra $\mathfrak{q}$, which is the radical of $\mathfrak{k}$. Moreover, by the functoriality of the exponential function (Fact~\ref{fact:functoriality of exponential function}), $\rho$ restricts to a covering map from $Q_K$ to $Q_{K'}$ with kernel $\ker\rho \cap Q_K$. By Fact~\ref{fact: CorrespondenceLiegroupandalgebra}, the Lie algebras of $S_K$ and $S_{K'}$ are both isomorphic to $\mathfrak{k}/\mathfrak{q}$. Hence, $S_K$ is a connected semisimple Lie group with compact Lie algebra. Using Fact~\ref{fact: semisimple compact}, we get $S_K$ is compact with finite center $Z(S_K)$.  Let $\pi: K \to S_K$ be the quotient map. Note that $\ker \rho$ is a subgroup of the center of $K$ by Fact~\ref{fact: covering group}. Hence,  the image of $\pi|_{\ker \rho}$ is a subset of $Z(S_K)$, which is finite.  As a consequence,  $\ker{\rho} \cap Q_K$, which is the kernel of  $\pi|_{\ker \rho}$, has the same rank $h$ as $\ker \rho$. Let $m_1$ and $m_2$ be the maximum dimensions of a compact subgroup of $Q_K$ and of $S_K$ respectively. Then $m= m_1+m_2$ by Fact~\ref{fact: maximal compact}.2. By the special case for the solvable group $K$ proven earlier, $h+m_1= \dim Q_K$. As $S_K$ is compact, $m_2 = \dim S_K$. Thus, $h+m= h+m_1+m_2 = \dim(Q_K)+ \dim(S_K)$ as desired.
\end{proof}

The following proposition links the noncompact Lie dimension and the helix dimension.

\begin{proposition} \label{prop: helixandnoncompact}
Suppose $G$ is a connected semisimple Lie group of dimension $d$, $m$ is the maximal dimension of a compact subgroup of $G$, $h$ is the helix dimension of $G$, and $G=KAN$ is an Iwasawa decomposition of $G$. Then $h= \dim K-m$, or equivalently, $d-m-h =\dim (AN)$. 
\end{proposition}

\begin{proof}

Let $Z(G)$ be the center of $G$. Then $Z(G)$ has rank $h$ by definition. By Fact~\ref{fact: Lie group decomp Iwasawa}.2, we have $Z(G)$ is a subset of $K$. Let $G'= G/ Z(G)$, and $K'= K/Z(G)$. Using Lemma~\ref{IWasawaandquotientbycompact}.2, we obtain $A'$ and $N'$ such that 
 $G'=K'A'N'$ is an Iwasawa decomposition.  Let $\rho: G \to G'$ be the quotient map. The group $Z(G)$ is discrete by Fact~\ref{fact: centerless1}, so $\rho$ and $\rho|_K$ are covering maps.

Now, the maximum dimension of a compact subgroup of $G$ is the same as that of $K$ by Lemma~\ref{IWasawaandquotientbycompact}.1.  Applying Lemma~\ref{lem: helixandnoncompact} to $K$, we have that $h=\dim K -m$. Note that $d = \dim(K)+\dim(AN)$ by Fact~\ref{fact: Lie group decomp Iwasawa}, so we also get  $d-m-h =\dim (AN)$.
\end{proof}

The next lemma discusses the noncompact Lie dimensions and the helix dimensions of a Lie group and its open subgroups.

\begin{lemma} \label{dimensionsforopensubgroup}
Suppose $G$ is a Lie group, and $G'$ is an open subgroup of $G$. Then $G$ and $G'$ have the same dimension, the same maximum dimension of a compact subgroup, and the same helix dimension.
\end{lemma}
\begin{proof}

It is clear that $G$ and $G'$ have the same dimension. Any compact subgroup of $G'$ is a compact subgroup of $G$. If $K$ is a compact subgroup of $G$, then $K \cap G'$ is an open subgroup of $K$, hence $K \cap G'$ has the same dimension as $K$. Therefore the maximum dimension of a compact subgroup of $G$ is the same as that of $G'$. Finally, note that $G$ and $G'$ have the same identity component $G_0$, and the helix dimension is defined using $G_0$. Thus, $G$ and $G'$ have the same helix dimension.
\end{proof}

The following lemma tells us the behavior of the radical under quotient by a compact normal subgroup.

\begin{lemma} \label{lem: Leviandquotientbycompact}
Suppose $G$ is a Lie group, $H$ is a compact normal subgroup of $G$, $G' = G/H$, and $\pi: G \to G'$ is the quotient map. Let $Q$ be the radical of $G$, and let $S= G/Q$. Then we have the following:
\begin{enumerate}
    \item  with $Q'=\pi(Q)$, and $S'= G'/ Q'$,  we have $HQ$ is closed in $G$, $Q'= HQ/H$, and $S'=G'/(HQ/H) =(G/H)/(HQ/H)$ is canonically isomorphic as a topological group to  both $G/HQ$ and $(G/Q)/(HQ/Q) = S/(HQ/Q)$;
    \item  $Q'$ is the radical of $G'$.
   \end{enumerate}
\end{lemma}
\begin{proof}
We prove (1). As $H$ is compact, we get that $HQ$ is closed in $G$ by Lemma~\ref{lemma:productofclosedandcompact}. Then $Q' = HQ/H$, and $S'=G'/(HQ/H) =(G/H)/(HQ/H)$. The remaining part of (1) is a consequence of the third isomorphism theorem (Fact~\ref{fact: homomorphism facts}.3).

We next prove (2). As $Q'$ is a quotient of the solvable group $Q$, it is solvable. Moreover, $Q'$ is a connected  closed normal subgroup of $G'$ as $Q$ is a connected closed normal subgroup of $G$. By (1), $G'/Q'$ is a quotient of the semisimple group $S$. Hence, $G'/Q'$ is semisimple. Therefore, $Q'$ is the maximal connected solvable closed normal subgroup of $G'$. In other words,  $Q'$ is the radical of $G'$.
\end{proof}

The next lemma says in a Lie group, taking a quotient by a normal compact subgroup does not change the helix dimension. Doing so also does not change the difference between the dimension and the dimension of a maximal compact subgroup.

\begin{lemma} \label{dimensionunderquotientbycompactgroup}
Suppose $G$ is a Lie group, $H$ is a compact normal subgroup of $G$, and $G' = G/H$. Let $d$, $m$, and $h$ be the dimension, the maximal dimension of a compact subgroup, and the helix dimension of $G$,  respectively. Define $d'$, $m'$, and $h'$ likewise for $G'$. Then:
\begin{enumerate}
    \item $d= d'+\dim(H)$ and $m= m'+\dim(H)$;
    \item $h= h'$.
\end{enumerate}
\end{lemma}

\begin{proof}

We prove (1). Clearly, $d=d'+\dim(H)$. If $K$ is a compact subgroup of $G$ and $K'=\pi(K)$, then $K'$ is a compact subgroup of $G'$, then $\dim(K') +\dim(H) = \dim(K)$. Conversely, if $K'$ is a compact subgroup of $G'$, then $K=\pi^{-1}(K')$ is a compact subgroup of $G$ by Lemma~\ref{lem: inverse image of compact sets}, and Lemma~\ref{lem: helixandnoncompact} implies that $\dim(K)=\dim(K')+\dim(H)$. Therefore, $m=m'+\dim(H)$.

We now prove (2). First further assume that both $G$ and $G'$ are semisimple. Let $\pi: G \to G'$ be the quotient map. 
Using Lemma~\ref{IWasawaandquotientbycompact}.1, we obtain an Iwasawa decomposition $G=KAN$ of $G$ such that $H \subseteq K$.  By Lemma~\ref{IWasawaandquotientbycompact}.2, we obtain an Iwasawa decomposition $G'=K'A'N'$ with $K' = \pi(K)$, $A' = \pi(A)$, and $N' = \pi(N)$. Let $m_K$ be the maximum dimension of a compact subgroup of $K$, and $m'_K$ be the maximum dimension of a compact subgroup of $K'$. By Proposition~\ref{prop: helixandnoncompact}, $m_K+h =\dim(K)$, and $m_{K'}+h'= \dim(K')$. Now, by (1) applied to $K$, we have $m_K = {m_K'}+\dim(H)$. Therefore, we get $h=h'$.

Next,  consider the case where $G$ is connected. Let $Q$ be the radical of $G$, $S=G/Q$, $Q'=\pi(Q)$, and $S'= G'/Q'$. Then by Lemma~\ref{lem: Leviandquotientbycompact}.2, $Q'$ is the radical of $G'$. Hence, it suffices to show that $S$ and $S'$ have the same helix dimension. By Lemma~\ref{lem: Leviandquotientbycompact}.1, $S'$ is isomorphic as a topological group to $S/(HQ/Q)$. Note that $HQ/Q$ is isomorphic as a topological group to $H/(H\cap Q)$ by the second isomorphism theorem for Lie groups (Fact~\ref{fact: iso theorems Lie}.2). In particular, $HQ/Q$  is compact, and $S'$ is the quotient of $S$ by a compact group. Applying the known case for semisimple groups, we get the desired conclusion.

Finally, we address the general case. Let $G_0$ be the identity component of $G$. Then $G_0$ is open by Fact~\ref{factLieid}, and $G_0H/H$ is an open subgroup of $G'=G/H$. Hence, by Lemma~\ref{dimensionsforopensubgroup}, $G$ has the same helix dimension as $G_0$, and $G'$ has the same helix dimension as $G_0H/H$. By the second isomorphism theorem (Fact~\ref{fact: homomorphism facts}.2), $G_0H/H$ is isomorphic as a topological group to $G_0/(G_0\cap H)$, which is a quotient of $G_0$ by a compact subgroup. Thus, we get the desired conclusion for the general case from the known case discussed above for connected groups.
\end{proof}

\begin{lemma}\label{lem: Common Lie refinement} 
Suppose $G$ is an almost-Lie group, $H_1$ and $H_2$ are closed normal subgroups of $G$ such that $G/H_1$ and $G/H_2$ are Lie groups, and $H= H_1 \cap H_2$. Then $G/H$ is a Lie group. 
\end{lemma}

\begin{proof}
By Fact~\ref{lem: almostliesubgroupandquotient}, $G/H$ is an almost-Lie group. In light of Fact~\ref{fact: Gleason}.2, we want to construct an open neighborhood  $U $ of the identity in $G/H$ that contains no nontrivial compact subgroup. Let $\pi: G \to G/H$, $\pi_1: G \to G/H_1$, and $\pi_2: G \to G/H_2$ be the quotient maps. Using Fact~\ref{fact: homomorphism facts}.3, we get continuous surjective group homomorphisms $p_1: G/H \to G/H_1$ and $p_2: G/H \to G/H_2$ such that 
\[
\pi_1 = p_1 \circ \pi \quad \text{and} \quad \pi_2= p_2 \circ \pi.
\]
As $G/H_1$ is a Lie group, we can use  Fact~\ref{fact: Gleason}.2 to choose an open neighborhood $U_1$ of the identity in $G/H_1$ such that $U_1$ contains no nontrivial compact subgroup of $G/H_1$. Choose an open neighborhood $U_2$ of the identity in $G/H_2$ likewise, and set 
\[
U = p_1^{-1}(U_1) \cap p_2^{-1}(U_2).
\]
If $K \subseteq U$ is a compact subgroup of $G/H$, then $p_1(K)$ is a compact subgroup of $U_1$. By our choice of $U_1$,  $p_1(K) =\{ \id_{G/H_1}\}$, which implies that $\pi_1^{-1}(p_1(K)) = \pi^{-1}(K)$ is a subgroup of $H_1$. A similar argument yields that $\pi_2^{-1}(p_2(K)) = \pi^{-1}(K)$ is a subgroup of $H_2$. Hence, $\pi^{-1}(K)$ must be a subgroup of $H = H_1 \cap H_2$. It follows that $K =
\{ \id_{G/H}\}$, which is the desired conclusion.
\end{proof}

Proposition~\ref{prop: welldefinednomcompactdim} below ensures us the notion of noncompact Lie dimension and helix dimension of a locally compact group as described in the introduction are well defined.

\begin{proposition} \label{prop: welldefinednomcompactdim}
Suppose $G'$ is an open subgroup of a locally compact group  $G$, and $H \vartriangleleft G'$ is compact such that $G'/H$ is a Lie group with dimension $d$, with  $m$ the maximum dimension of a compact subgroup, and helix dimension $h$. Then $d-m$ and $h$ are independent of the choice of $G'$ and $H$.
\end{proposition}

\begin{proof}

We first prove a simpler statement: if $G'$ is an almost-Lie subgroup of $G$, $H$ is a compact subgroup of $G'$, and we define $d$, $m$, and $h$ as in the statement of the proposition, then $d-m$ and $h$ are independent of the choice of $H$.  Let $H_1$ and $H_2$ be compact normal subgroups of $G$ such that both $G/H_1$ and $G/H_2$ are Lie groups. Then by Lemma~\ref{lem: Common Lie refinement}, $G/(H_1\cap H_2)$ is also a Lie group. Note that $G/H_1$ and $G/H_2$ are quotients of $G/(H_1 \cap H_2)$ by compact subgroups by the third isomorphism theorem (Fact~\ref{fact: homomorphism facts}.3). Hence, it follows from Lemma~\ref{dimensionunderquotientbycompactgroup} that $G/H_1$ and $G/(H_1 \cap H_2)$ have the same difference between the dimension and the maximum dimension of a compact subgroup, and the same helix dimension. A similar statement holds for $G/H_2$ and $G/(H_1 \cap H_2)$. This completes the proof of the simpler statement.

Now we prove the proposition in general.
Let $G'_1$ and $G'_2$ be open subgroups of $G$, and let $H_1$ and $H_2$ be compact normal subgroups of $G'_1$ and $G'_2$ respectively such that $G'_1/H_1$ and $G'_2/H_2$ are Lie groups.  Using the Gleason--Yamabe Theorem (Fact~\ref{fact: Gleason}), we get an open subgroup $G'$ of $G_1' \cap G_2'$ which is an almost-Lie group. Then $G'$ is an open subgroup of $G$. Note that $G'\cap H_1$ and $G'\cap H_2$ are compact subgroups of $G'$. Then $G'/(G' \cap H_1)$ is an open subgroup of $G'_1/H_1$.  It follows from Lemma~\ref{dimensionunderquotientbycompactgroup} that $G'/H_1$ and $G'/(G' \cap H_1)$ have the same difference between the dimension and the maximum dimension of a compact subgroup and the same helix dimension. A similar statement holds for $G'/H_2$ and $G'/(G' \cap H_2)$. Thus, from the simpler statement we proved in the preceding paragraph, $G_1'/H_1$ and $G_2'/H_2$ have the same noncompact dimension and the same helix dimension.
\end{proof}

We have the following two corollaries. 

\begin{corollary} \label{cor: prop 2 (1)}
If $H$ is an open subgroup of $G$, then $H$ has the same noncompact Lie dimension and helix dimension as $G$. 
\end{corollary}
\begin{proof}
Proposition~\ref{prop: welldefinednomcompactdim} implies that the noncompact Lie dimension and helix dimension of a locally compact group is the same as those of its open almost-Lie subgroups, if those exist. Hence, it suffices to show that there is a common  almost-Lie open subgroup of $G$ and $H$. This is an immediate consequence of the Gleason--Yamabe Theorem (Fact~\ref{fact: Gleason}.1).
\end{proof}

\begin{corollary} \label{cor: prop 2 (2)}
If $H$ is  a compact normal subgroup of $G$, then $G/H$ has the same noncompact Lie dimension and helix dimension as $G$. 
\end{corollary}

\begin{proof}
Let $\pi$ be the projection from $G$ to $G/H$. If $G/H$ is a Lie group, then from the definitions, $G$ has the same noncompact Lie dimension and helix dimension as $G/H$. Hence, the conclusion holds in this special case. 

Suppose there is a compact $K \vartriangleleft G/H$ such that $(G/H)/K$ is a Lie group, then $(G/H)/K$ is isomorphic as a topological group to $G/ \pi^{-1}(K)$ by the third isomorphism theorem (Fact~\ref{fact: homomorphism facts}.3). By Lemma~\ref{lem: inverse image of compact sets}, $\pi^{-1}(K)$ is compact. Hence $(G/H)/K$ is a quotient of $G$ by a compact normal subgroup, and we can use the previous case to get the desired conclusion. 
 
 Now we treat the general situation. By the Gleason--Yamabe Theorem, we get an almost-Lie open subgroup $G'$ of $G$. Then $G'H$ is an open subgroup of $G$ and hence has the same noncompact Lie dimension and helix dimension as $G$ by Corollary~\ref{cor: prop 2 (1)}. By the second isomorphism theorem (Fact~\ref{fact: homomorphism facts}.2), we get that $G' / (G' \cap H)$ is isomorphic to  $G'H/H$ which is an open subgroup of $G/H$. In particular, $G' / (G' \cap H)$ has the same noncompact Lie dimension and helix dimension as $G/H$ by Corollary~\ref{cor: prop 2 (1)}. Note that $G' / (G' \cap H)$ is an almost-Lie group by Fact~\ref{lem: almostliesubgroupandquotient}. Hence, we can find a compact subgroup $K$ such that $(G'/(G' \cap H))/K$ is a Lie group. We are back to the earlier known situation in the second paragraph.
\end{proof}

We have the following lemma about Iwasawa decompositions. 

\begin{lemma} \label{lem: IWasawaexact}
Suppose  $1 \to H \to G \overset{\pi\ }{\to} G/H \to 1$ is an exact sequence of connected semisimple Lie groups. Then there are Iwasawa decompositions $G=KAN$, $H=K_1A_1N_1$, and $G/H= K_2A_2N_2$ such that $K_1 = (K \cap H)_0$, and $K_2= \pi(K)$.
\end{lemma}

\begin{proof}
Let $\mathfrak{g}$ and $\mathfrak{h}$ be the Lie algebras of $G$ and $H$, and let $\kappa_\mathfrak{g}$ and $\kappa_\mathfrak{h}$ be the Cartan--Killing form of $\mathfrak{g}$ and $\mathfrak{h}$. Then $\mathfrak{g}/\mathfrak{h}$ is the Lie algebra of $G/H$, and $\mathfrak{g}$ ,$\mathfrak{h}$, and $\mathfrak{g}/\mathfrak{h}$ are semisimple. By Fact~\ref{fact: orthogonaldecomposition}, $$\mathfrak{g} = \mathfrak{h} \oplus \mathfrak{c}  \text{ and } \kappa_\mathfrak{g} = \kappa_\mathfrak{h} \oplus \kappa_\mathfrak{c}  $$ where $\kappa_\mathfrak{c}$ is the orthogonal complement of $\kappa_\mathfrak{h}$ with respect to $\kappa_\mathfrak{g}$, and  $\kappa_\mathfrak{c}$ is the Cartan--Killing form of $\kappa_\mathfrak{c}$. Therefore, the quotient map from $\mathfrak{g}$ to $\mathfrak{g}/\mathfrak{h}$ induces an isomorphism from $\mathfrak{c}$ to $\mathfrak{g}/\mathfrak{h}$, so we can identify $\mathfrak{g}/\mathfrak{h}$ with $\mathfrak{c}$.
Let $\tau_1$ and $\tau_2$  be Cartan involutions of $\mathfrak{h}$ and $\mathfrak{c}$. Then $\tau = \tau_1 \oplus \tau_2$ is an involution of $\mathfrak{g}$. As $\tau_1$ and $\tau_2$ are Cartan involutions, the bilinear forms $\mathfrak{h} \times \mathfrak{h}: (x_1, y_1) \mapsto -\kappa_\mathfrak{h}(x_1, \tau_1(y_1))$ and $\mathfrak{c} \times \mathfrak{c}: (x_2, y_2) \mapsto -\kappa_\mathfrak{c}(x_2, \tau_2(y_2))$ are positive definite. Hence, the bilinear from  $\mathfrak{g} \times \mathfrak{g}: (x, y) \mapsto -\kappa_\mathfrak{g}(x, \tau(y))$ is also positive definite. Therefore, $\tau$ is a Cartan involution of $\mathfrak{g}$.  Let $\mathfrak{k}$, $\mathfrak{k}_1$, and $\mathfrak{k}_2$ be the Lie subalgebras of $\mathfrak{g}$,  $\mathfrak{h}$, and $\mathfrak{c}$ fixed by $\tau$, $\tau_1$, and $\tau_2$ respectively. It is easy to see that $\mathfrak{k}= \mathfrak{k}_1 \oplus \mathfrak{k}_2$. Let $\exp: \mathfrak{g} \to G$, $\exp_1: \mathfrak{h} \to H$, and $\exp_2: \mathfrak{c} \to G/H $ be the exponential maps, and set  $$K= \exp(\mathfrak{k}), K_1=\exp_1(\mathfrak{k}_1) \text{ and } K_2= \exp(\mathfrak{k}_2).$$  From Fact~\ref{fact: Lie group decomp Iwasawa}, we obtain Iwasawa decompositions $G=KAN$, $H=K_1A_1N_1$, and $G/H= K_2A_2N_2$. By the functoriality of the exponential function (Fact~\ref{fact:functoriality of exponential function}), we get $K_1 \leq K \cap H$ and $K_2 = \pi(K)$. Since $K_1$ is connected, by a dimension calculation we have $K_1=(K\cap H)_0$. 
\end{proof}

In a short exact sequence of locally compact groups, one may hope that the noncompact Lie dimension and the helix dimension of the middle term is the sum of those of the outer terms. This is not true in general. For instance, in the exact sequence 
\[
1 \to \ZZ \to \RR \to \RR/\ZZ \to 1,
\]
the noncompact Lie dimension of $\RR$ is $1$, while both $\ZZ$ and $\TT = \RR/ \ZZ$ has noncompact Lie dimension $0$. Another example is the following. Let $H$ be the universal cover of $\mathrm{SL}(2,\RR)$. We identify the kernel of the covering map $\rho: H\to  \mathrm{SL}(2,\RR)$ with the additive group $\ZZ$ of integers. 
Let $G=(H\times \RR)/\{(n,n): n\in\ZZ\}$. Then we have the exact sequence
\[
1 \to H \to G \to \TT \to 1,
\]
the helix dimension of $H$ is $1$, but the helix dimensions of $G$ and $\TT$ are $0$.

Nevertheless, we have the summability of noncompact Lie dimensions and helix dimensions in many short exact sequences of interest: 
\begin{proposition}\label{prop: additivitydim1}
Suppose $1 \to H \to G \overset{\pi\ }{\to} G/H \to 1$ is an exact sequence of  connected Lie groups. Then we have the following:
\begin{enumerate}
    \item If $n$, $n_1$, and $n_2$ are the noncompact Lie dimensions of $G$, $H$, and $G/H$ respectively, then $n=n_1+n_2$;
    \item If $G$ is moreover semisimple, and $h$, $h_1$, and $h_2$ are the helix dimensions of $G$, $H$, and $G/H$ respectively, then $h=h_1+h_2$.
\end{enumerate}
\end{proposition}
\begin{proof}
 We first prove (1). Let $m$ be the maximum dimension of a compact subgroup in $G$.  As $G$ is connected, $m$ is also the dimension of an arbitrary maximal compact subgroup of $G$ by Fact~\ref{fact: maximal compact}.1. Defining $m_1$ and $m_2$ likewise for $H$ and $G/H$, we get similar conclusions for them from the connectedness of $H$ and $G/H$.  Let $K$ be a maximal compact subgroup of $G$.
By Fact~\ref{fact: maximal compact}.2, $K \cap H$ is a maximal compact subgroup in $H$, and $\pi(K)$ is a maximal compact subgroup in $G/H$. The kernel of $\pi|_K$ is $K \cap H$, and the image is $\pi (K)$. Hence,  $m = m_1+ m_2$. This gives us (1) recalling that $m+n =\dim(G)$, $m_1+n_1=\dim(H)$, $m_2+n_2 =\dim(G/H)$, and $\dim(G)= \dim(H)+\dim(G/H)$.

We now prove (2). Since $Z(G)\cap H\leq Z(H)$, and $\pi(Z(G))\leq Z(G/H)$, we have $h\leq h_1+h_2$. It remains to show $h\geq h_1+h_2$. As $G$ is semisimple, $H$ and $G/H$ are semisimple by Fact~\ref{fact: preservationsemisimple}. Take Iwasawa decompositions $G=KAN$, $H=K_1A_1N_1$, and $G/H= K_2A_2N_2$ as in Lemma~\ref{lem: IWasawaexact}. By the first isomorphism theorem for Lie groups (Fact~\ref{fact: iso theorems Lie}.1),  $1\to K\cap H \to K \to K_2 \to 1$ is an exact sequence of Lie groups. With $K_2'=K/K_1$, we have an exact sequence 
\begin{equation}\label{eq: exact seq}
   1\to K_1 \to K \to K_2' \to 1.
\end{equation} 
As $K_1=(K\cap H)_0$, by the third isomorphism theorem, we have $K_2=K/(K\cap H)=(K/K_1)/((K\cap H)/K_1)=K_2'/((K\cap H)/K_1)$. Since $(K\cap H)/K_1$ is discrete, $K_2'$ is a covering group of $K_2$. Let $\phi: K_2'\to K_2$ be the covering map. Note that $\phi$ has a discrete kernel, and $K_2$, $K_2'$ have the same dimension. Suppose $S$ is a compact subgroup of $K_2'$ with the maximum dimension. Then $\phi(S)$ is a compact subgroup of $K_2$, and $S$ and $\phi(S)$ have the same dimension. This shows that the noncompact Lie dimension of $K_2'$ is at least the noncompact Lie dimension of $K_2$. By \eqref{eq: exact seq} and Statement (1), the noncompact Lie dimension of $K$ is the sum of noncompact Lie dimensions of $K_1$ and $K_2'$, hence it is at least  the sum of noncompact Lie dimensions of $K_1$ and $K_2$. It then follows from  Proposition~\ref{prop: helixandnoncompact} that $h\geq h_1+h_2$.
\end{proof}

\begin{lemma} \label{lem: connectedkernelofmodular}
Suppose $1 \to H \to G \overset{\pi\ }{\to} (\RR^{>0}, \times) \to 1$ is an exact sequence of Lie groups, then $H_0=H\cap G_0$. In particular, if $G$ is connected, then $H$ is connected.
\end{lemma}

\begin{proof}
As Lie groups are locally path connected, $G_0$ is open in $G$. 
As $G_0$ is an open connected subgroup of $G$, the map $\pi|_{G_0}$ is continuous and open.  Hence, its image $\pi(G_0)$ is an open connected subgroup of  $(\RR^{>0}, \times)$. Therefore,  $\pi(G_0) = (\RR^{>0}, \times)$, and  $\pi|_{G_0}$ is a quotient map by the first isomorphism theorem (Fact~\ref{fact: homomorphism facts}.1). The kernel of $\pi|_{G_0}$ is $H \cap G_0$, so we get the exact sequence of Lie groups
 $$ 1 \to H \cap G_0 \to G_0 \overset{\pi|_{G_0}}{\to} (\RR^{>0},\times)\to 1. $$
We now prove that  $H_0 = H \cap G_0$. 
 The forward inclusion is immediate by definition. By the third isomorphism theorem (Fact~\ref{fact: homomorphism facts}.3), we get the exact sequence of Lie groups
$$  1 \to (H \cap G_0)/ H_0 \to  G_0/H_0 \to (\RR^{>0},\times) \to 1.  $$
 The group $(H \cap G_0)/ H_0$ is discrete. Hence, $G_0/H_0$ is a Lie group with dimension $1$.  As $G_0$ is connected, the Lie group $G_0/H_0$ is also connected. By Fact~\ref{fact: classification}, $G_0/H_0$ is either isomorphic to $\RR$ or $\TT$. But since $G_0/H_0$ has  $(\RR^{>0},\times)$ as a quotient, it cannot be compact, and therefore must be isomorphic to $\RR$. This implies that $(H \cap G_0)/ H_0$ is trivial, and hence $H_0= H \cap G_0$.
\end{proof}

The next proposition gives us a summability result of noncompact Lie dimensions along a short exact sequence of locally compact groups when the quotient group is $(\RR^{>0}, \times)$.

\begin{proposition}\label{prop: additivityofdimension2}
Suppose $1 \to H \to G \overset{\pi\ }{\to} (\RR^{>0}, \times) \to 1$ is an exact sequence of locally compact groups. Then we have the following:
\begin{enumerate}
    \item If $n$, $n_1$, and $n_2$ are the noncompact Lie dimensions of $G$, $H$, and $(\RR^{>0}, \times)$ respectively, then $n= n_1+n_2 = n_1+1$.
    \item $G$ and $H$ have the same helix dimension.
\end{enumerate}
\end{proposition}

\begin{proof}
First, we consider the case when $G$ is a connected Lie group. Then by Lemma~\ref{lem: connectedkernelofmodular}, $H$ is also connected. Hence, (1) for this case is a consequence of Proposition~\ref{prop: additivitydim1}.1.  

We prove (2) for this special case. Let $Q$ be the radical of $G$. We claim that $QH=G$, or equivalently, that $\pi(Q)= (\RR^{>0}, \times)$. Suppose this is not true. Then $\pi(Q)$ is a connected subgroup of $(\RR^{>0}, \times)$, so it must be $\{1\}$. Hence, $Q \subseteq H$. Then $(\RR^{>0}, \times) =G/H$ which is isomorphic as a topological group to $(G/Q)/(H/Q)$ by the third isomorphism theorem (Fact~\ref{fact: homomorphism facts}.3).  This is a contradiction, because $(G/Q)/(H/Q)$ is semisimple as a quotient of the semisimple group $G/Q$, while $(\RR^{>0}, \times)$ is solvable.

We next show that $Q \cap H$ is the radical of $H$.  The radical of $H$ is a characteristic closed subgroup of $H$ (by Fact~\ref{fact: radical}), hence a connected solvable closed normal subgroup of $G$. Thus, the radical of $H$ is a subgroup of $Q\cap H$. It is straightforward that $Q\cap H$ is solvable. We also have that $Q\cap H$ is second countable as both $Q$ and $H$ are second countable. From the preceding paragraph, $\pi(Q)=(\RR^{>0}, \times)$.
Using the first isomorphism theorem for Lie groups (Fact~\ref{fact: iso theorems Lie}.1), we have the exact sequence
$$ 1 \to Q\cap H \to Q \to (\RR^{>0}, \times) \to 1. $$ Applying Lemma~\ref{lem: connectedkernelofmodular}, we learn that $Q\cap H$ is connected. This completes the proof that  $Q \cap H$ is the radical of $H$.

Note that $QH=G$ and $Q$ is a closed subgroup of $G$. Hence, by the second isomorphism theorem for Lie groups (Fact~\ref{fact: iso theorems Lie}.2), $H/(Q\cap H)$ is isomorphic as a topological group to $HQ/Q= G/Q$. Therefore $G$ and $H$ have the same helix dimension.

Next, we address the slightly more general case where $G$ is a Lie group but not necessarily connected. Let $G_0$ be the connected component of $G$. Then $\pi(G_0)$ is an open subgroup of $(\RR^{>0}, \times)$, so $\pi(G_0) = (\RR^{>0}, \times)$. By the first isomorphism theorem for Lie groups (Fact~\ref{fact: iso theorems Lie}.1), we have the exact sequence $$1 \to G_0\cap H \to G_0 \to (\RR^{>0}, \times)\to 1.$$ Applying  Lemma~\ref{prop: Reduceinequalitytoopensubgroups} and the known case of the current lemma where the middle term of the exact sequence is a connected Lie group, we obtain both (1) and (2) for this more general case.

Using the Gleason--Yamabe theorem and a similar argument as in the preceding paragraph, we can reduce (1) and (2) for general locally compact groups to the case where we assume that $G$ is an almost-Lie group. Hence, there is a compact normal subgroup $K$ of $G$ such that $G/K$ is a Lie group. As $K$  is compact, $\pi(K)$ is a compact subgroup of $(\RR^{>0}, \times)$, so $\pi(K)=\{ 1\}$.
Hence $K\vartriangleleft H$. By the third isomorphism theorem (Fact~\ref{fact: homomorphism facts}.3), we have the exact sequence 
$1 \to H/K \to G/K \to  (\RR^{>0}, \times)\to 1$.  Applying  Lemma~\ref{dimensionunderquotientbycompactgroup} and the known case of the current lemma where the middle term of the exact sequence is a Lie group, we obtain both (1) and (2) for this remaining case.
\end{proof}

We discuss the relationship between the noncompact Lie dimension and helix dimension of a locally compact group $G$.

\begin{corollary}\label{cor: n/3}
Suppose $G$ has noncompact Lie dimension $n$ and helix dimension $h$. Then we have $h\leq n/3$.
\end{corollary}

\begin{proof}
We first check the result for simple Lie groups. If $h=0$, then the statement holds vacuously. Hence, using Fact~\ref{fact: centersimpleLie}, it suffices to consider the case where $h=1$.  Let $G =KAN$ be an Iwasawa decomposition. Then by Proposition~\ref{prop: helixandnoncompact}, we have $n-1 =\dim(AN) \geq 0$. Hence, $n>0$ and $G$ is not compact. From Fact~\ref{fact: classification2}, we have $\dim(AN) \geq 2$. Therefore $n \geq 3$. Hence, we get the desired conclusion for simple Lie groups.

Now consider the case where $G$ is a connected semisimple Lie group which is not simple. Using induction on dimension, we can assume we have proven the statement for all connected semisimple Lie groups of smaller dimensions. Using Fact~\ref{fact: simple}, we get an exact sequence of semisimple Lie groups $1 \to H \to G \to G/H \to 1$ with $0<\dim(H)< \dim(G)$. Replacing $H$ with its connected component if necessary, we can arrange that $H$ is connected. The desired conclusion then follows from  Proposition~\ref{prop: additivitydim1}.2.

For a general locally compact group $G$, from Proposition~\ref{prop: welldefinednomcompactdim}, we may assume $G$ is a Lie group. Corollary~\ref{cor: prop 2 (1)} and Fact~\ref{factLieid} allow us to reduce the problem to connected Lie groups. By Lemma~\ref{lem: Leviandquotientbycompact} and Lemma~\ref{dimensionsforopensubgroup}, the radical of $G$ only contributes to the noncompact Lie dimension of $G$. Using Fact~\ref{fact: Lie group decomp Levi} and Proposition~\ref{prop: additivitydim1}.1, we reduce the problem to connected semisimple Lie groups. 
\end{proof}

\section{Proof of Theorem~\ref{thm: mainsharp}}

Theorem~\ref{thm: mainsharpopen} is a version of Theorem~\ref{thm: mainsharp}.2 with compact sets replaced by open  sets. We will deduce Theorem~\ref{thm: mainsharp} from Theorem~\ref{thm: mainsharpopen} at the end of this section using the inner regularity of the Haar measure. It is convenient to prove  Theorem~\ref{thm: mainsharpopen} first because, intuitively, the constructed set $X$ is an open neighborhood of the maximal compact subgroup of $G$. On the other hand, in Theorem~\ref{thm: mainsharp}, we want a compact set as traditionally stated in the Brunn--Minkowski inequalities.

\begin{theorem}\label{thm: mainsharpopen}
 Suppose $G$ is a locally compact group with noncompact Lie dimension $n>0$,  $\mu_G$ is a left Haar measure, and $\nu_G$ is a right Haar measure. Then, for every $\varepsilon > 0$, there is an open set $X$ with positive left and right measure in $G$ such that
\[
\frac{\nu_G(X)^{\frac{1}{n}-\varepsilon}}{\nu_G(X^2)^{\frac{1}{n}-\varepsilon}}+ \frac{\mu_G(X)^{\frac{1}{n}-\varepsilon}}{\mu_G(X^2)^{\frac{1}{n}-\varepsilon}}>1.
\]
As a corollary, if $G$ is unimodular with $n>0$, for every $\varepsilon'>0$, there is an open set $X$ in $G$ such that $\mu_G(X^2) < (2^n+\varepsilon') \mu_G(X)$.
\end{theorem}

We first prove the theorem when $G$ is a unimodular Lie group.

\begin{proof}[Proof of Theorem \ref{thm: mainsharpopen}, unimodular Lie group case]

Since $G$ is unimodular, without loss of generality we assume that $\mu_G = \nu_G$. Let $d$ be the dimension of $G$. Let $K$ be a maximal compact subgroup of $G$ and let $m = \dim K$. Hence $n=d-m$ is the noncompact Lie dimension of $G$.

Since $K$ is closed, $G/K$ is a homogeneous (and smooth) manifold.  For $g \in G$, we let $[g]$ denote its quotient image in $G/K$.  Fix an arbitrary $G$-invariant (smooth) Riemannian metric on $G/K$ (such a metric exists by first finding a $K$-invariant Riemannian metric at $[\id]$ and then extend it onto the whole $G/K$ by the action of $G$). This metric induces a volume measure $\Vol$ on $G/K$.

Let $\pi$ be the projection from $G$ to $G/K$. For any Borel subset $U$ of $G/K$, $\pi^{-1} (U)$ is also Borel and hence $\mu_G$-measurable. For any $r>0$, we use $B_{r}$ to denote the (open) $r$-ball around $[\id]$ on $G/K$ under the chosen metric and use $D_r$ to denote $\pi^{-1} (B_{r})$. We claim that:\smallskip
\begin{enumerate}[(i)]
    \item  There exists a constant $b>0$ only depending on the metric on $G/K$ such that as Borel measures $\pi_* (\mu_G) = b\cdot \Vol$, and
    \item For any $r>0$, $D_r\!\cdot\! D_r \subseteq D_{2r}$.
\end{enumerate}
\smallskip

We postpone the proofs of claims (i) and (ii) to the end of this proof and first show how they lead to Theorem \ref{thm: mainsharp}. We can take $X$ to be $D_{\delta}$ for a sufficiently small $\delta>0$ (depending on $\varepsilon$) to be determined. Then by (i), $$\mu_G(X)=\pi_*(\mu_G (B_{\delta})) = b\cdot \Vol(B_{\delta}).$$ And by (ii), $X^2 \subseteq D_{2\delta}$ and hence as before, we get $\mu_G(X^2) \leq \mu_G(D_{2\delta}) = b\cdot \Vol(B_{2\delta})$. Note that the invariant metric on $G/K$ is smooth and thus $$\lim_{\delta \rightarrow 0} \frac{\Vol(B_{2\delta})}{\Vol(B_{\delta})} = 2^n.$$  Hence a sufficiently small $\delta$ can guarantee $\frac{\mu_G(X)^{\frac{1}{n}-\varepsilon}}{\mu_G(X^2)^{\frac{1}{n}-\varepsilon}} > \frac{1}{2}$ and we have proved Theorem \ref{thm: mainsharp} in this special case.

It remains to prove claims (i) and (ii). To see claim (i), note that $\Vol$ is $G$-invariant. We also see that $\pi_* (\mu_G)$ is $G$-invariant because $\mu_G (\pi^{-1} (U)) = \mu_G (g\pi^{-1} (U))= \mu_G (\pi^{-1} (gU))$ for any $g \in G$ and any Borel subset $U \subseteq G/K$. Since the $G$-invariant Borel measure on $G/K$ is unique up to a scalar (see Theorem 8.36 in \cite{knapp2013lie}), $\Vol$ has to be a scalar multiple of $\pi_* (\mu_G)$.

Finally we verify claim (ii). Taking arbitrary $g_1, g_2 \in D_r$ and it suffices to show $g_1 g_2 \in D_{2r}$. By definition, there is a piecewise smooth curve $\gamma_j$ connecting $[\id]$ and $[g_j]$ such that the length of $\gamma_j$ is strictly smaller than $r$ (for $j=1, 2$).  Note that by the invariance of the metric, $[g_1] \gamma_2$ must have the same length as $\gamma_2$. Let $\gamma$ be the curve formed by $[g_1] \gamma_2$ after $\gamma_1$. It is a curve connecting $[\id]$ and $[g_1 g_2]$ and by the reasoning above has two pieces and each of them has length strictly smaller than $r$. Hence $\gamma$ has length shorter than $2r$ and thus by definition $g_1 g_2 \in D_{2r}$. We have verified (ii).
\end{proof}

The following will be proved similarly to the above special case of Theorem~\ref{thm: mainsharpopen}.

\begin{proposition}\label{stableversionof13}
Given any unimodular Lie group $G$, let $n$ be its noncompact Lie dimension. Let $\tilde{\varepsilon} >0$ be fixed. Then there exists precompact open subsets $X$ and $X_1$ with $\mu_G(X)>0$ such that the closure $\overline{X} \subseteq X_1$ and $\mu_G(X_1\!\cdot\! X) < (2+\tilde{\varepsilon})^n \mu_G(X)$.
\end{proposition}

\begin{proof}
 We continue to use the notation from the proof. Let $X = D_{\delta}$ and $X_1 = D_{\delta_1}$, where $0<\delta<\delta_1$ are both to be determined.

We see that $X$ and $X_1$ are  open, being the preimage of open sets under $\pi$. Since
$B_{\delta}$  and $B_{\delta_1}$ are both precompact, Lemma \ref{lem: inverse image of compact sets} implies $X$ and $X_1$ are also precompact.
Moreover we have $\overline{X} \subseteq \pi^{-1} (\overline{B_{\delta}}) \subseteq \pi^{-1} (B_{\delta_1}) = X_1$. It remains to choose $\delta$ and $\delta_1$ such that $\mu_G(X_1\!\cdot\! X) < (2+\tilde{\varepsilon})^n \mu_G(X)$.

By the same reasoning as in the previous proof of Theorem \ref{thm: mainsharp} (unimodular Lie case), we see that $X_1\p X \subseteq D_{\delta_1 + \delta}$. Now,
\[\lim_{\delta \rightarrow 0} \frac{\Vol(B_{2\delta})}{\Vol(B_{\delta})} = 2^n,\quad \text{and}\quad \lim_{\delta_1 \rightarrow \delta} \frac{\Vol(B_{2\delta})}{\Vol(B_{\delta_1 + \delta})} = 1.
\]
Hence we can take $\delta$ sufficiently small, and then $\delta_1$ sufficiently close to $\delta$, such that
\[
\frac{\mu_G(X_1\!\cdot\!  X)}{\mu_G(X)} \leq \frac{\mu_G(D_{\delta_1+\delta})}{\mu_G(D_{\delta})} = \frac{\Vol(B_{\delta_1+\delta})}{\Vol(B_{\delta})} < (2+\tilde{\varepsilon})^n,
\]
which proves the proposition. 
\end{proof}

Next we use Proposition \ref{stableversionof13} to prove Theorem \ref{thm: mainsharpopen} for general Lie groups.

\begin{proof}[Proof of Theorem \ref{thm: mainsharpopen}, Lie group case]
We have already proved the theorem when $G$ is unimodular. In the rest of this proof, we assume $G$ is nonunimodular. Let $G_0$ be the connected component of $G$. Since $\mu_G|_{G_0}$ is a left Haar measure on $G_0$, and $\nu_G|_{G_0}$ is a right Haar measure on $G_0$, we may assume without loss of generality that $G=G_0$. As the only connected subgroups of $(\RR^{>0}, \times)$ are itself and $\{1\}$, and $G$ is not unimodular, the modular function $\Delta_G$ must be surjective. Hence, $\Delta_G$ is a quotient map by the first isomorphism for Lie groups Fact~\ref{fact: iso theorems Lie}.1.

Let $H$ be the kernel of the modular function on $G$. By Proposition \ref{prop: additivityofdimension2}.1, the noncompact Lie dimension of $H$ is $n-1$ where $n$ is the noncompact Lie dimension of $G$.
By Fact \ref{fact: modular function}.1, $H$ is unimodular. To avoid confusion, we will always use $\mu_G$ and $\nu_G$ for $\mu$ and $\nu$ below and use $\mu_H = \nu_H$ to denote a fixed Haar measure on $H$.

In light of Fact \ref{fact: Quotient Integral Formula}, we can fix a Haar measure $\d r$ on the multiplicative group $(\RR^{>0},\times) = G/H$ such that for any Borel function $f$ on $g$,
\begin{equation}\label{quotientintegralGandH}
    \int_G f(x)\d \mu_G(x)=\int_{G/H}\int_H f(rh) \d \mu_H(h)\d \mu_{G/H}(rH).
\end{equation}

Let $\mathfrak{g}$ and $\mathfrak{h}$ be the Lie algebras of $G$ and $H$, respectively. 
We fix an element $Z \in \mathfrak{g}$ such that $Z \notin \mathfrak{h}$. Note that $t \mapsto \Delta(\exp(tZ))$  is  a nontrivial continuous group homomorphism  from $(\RR,+)$ to $(\RR^{>0},\times)$.  As the only connected subsets of $(\RR^{>0},\times)$ are points and intervals, this map must be surjective, and hence an isomorphism by the first isomorphism for Lie groups (Fact~\ref{fact: iso theorems Lie}).  In light of the quotient integral formula (\ref{quotientintegralGandH}), we can choose an appropriate Haar measure $\d t$ on $\mathbb{R}$ such that for any Borel subset $A$ of $G$, we have the Fubini-type 
measure formula
\begin{equation}\label{fubinimeasureformula}
    \mu_G (A) = \int_{\Bbb{R}} \mu_H ((\exp(-tZ) A) \cap H) \d t.
\end{equation}
Without loss of generality we assume $\d t$ is the standard Lebesgue measure (otherwise we multiply $\mu_G$ by a constant).

With the preliminary discussions above, we now construct $X$ satisfying the inequality in Theorem \ref{thm: mainsharp}.

Before going into details of the construction, we first describe the intuition behind it. We arrange our $X$ to live very close to $H$ so that $\mu$ and $\nu$ are almost proportional on $X$ and $X^2$. We then realize that it suffices to choose our $X$ to be like a thickened copy of the almost sharp example of Theorem \ref{thm: mainsharp} for (the unimodular group) $H$.

More precisely, let $\tilde{\varepsilon} > 0$ be a small number (depending on $\varepsilon$) to be determined. let $\tilde{X}$ and $\tilde{X}_1$ be the $X$ and $X_1$, respectively,  in Proposition \ref{stableversionof13} where we replace $G$ by $H$. We now take $X = \{\exp(tZ)h: t \in [0, \tilde{\varepsilon}], h \in \tilde{X}\}$ and will show that $X^2$ is reasonably small when $\tilde{\varepsilon}$ is small enough.

By (\ref{fubinimeasureformula}), we have
\begin{equation}\label{measofXYeqn1}
    \mu_G (X^2) = \int_{\Bbb{R}} \mu_H ((\exp(-tZ) X^2) \cap H) \d t.
\end{equation}

Note that an arbitrary element in $X^2$ can be written as
\begin{align*}
&\ \exp(t_1 Z)h_1 \exp(t_2 Z)h_2 \\
=&\  \exp((t_1 + t_2)Z) (\exp(-t_2 Z) h_1 \exp(t_2 Z) h_2) \in \exp((t_1 + t_2)Z) H,
\end{align*}
where $t_1, t_2 \in [0, \tilde{\varepsilon}]$ and $h_1, h_2 \in H$. Hence (\ref{measofXYeqn1}) is reduced to
\begin{equation}\label{measofXYeqn2}
    \mu_G (X^2) = \int_{0}^{2\tilde{\varepsilon}} \mu_H ((\exp(-tZ) X^2) \cap H) \d t
\end{equation}
and moreover for any $0 \leq t_0 \leq 2\tilde{\varepsilon}$, we see from the above discussion that $$(\exp(-t_0 Z) X^2)\cap H  = \bigcup_{0 \leq t_1, t_2 \leq \tilde{\varepsilon}, t_1 + t_2 = t_0} (\exp(-t_1 Z) \tilde{X}\exp(t_1 Z)) \p \tilde{X}.$$

By Lemma \ref{openinclose} and Proposition \ref{stableversionof13}, when $\tilde{\varepsilon}$ is sufficiently small, which we will always assume, we have the above union contained in $\tilde{X}_1 \p \tilde{X}$. Now by (\ref{measofXYeqn2}),
\begin{equation}\label{measofXYeqn3}
    \mu_G (X^2) \leq \int_{0}^{2\tilde{\varepsilon}} \mu_H (\tilde{X}_1 \p \tilde{X}) \d t = 2\tilde{\varepsilon}\mu_H (\tilde{X}_1 \p \tilde{X}).
\end{equation}

On the other hand, by (\ref{fubinimeasureformula})
\begin{equation}\label{measofXeqn}
    \mu_G (X) = \tilde{\varepsilon}\mu_H (\tilde{X}).
\end{equation}

Combining (\ref{measofXYeqn3}) and (\ref{measofXeqn}) and using the measure properties of $\tilde{X}$ and $\tilde{X}_1$ guaranteed by Proposition \ref{stableversionof13}, we have
\begin{equation}\label{murelation}
    \frac{\mu_G(X)}{\mu_G (X^2)} \geq \frac{\mu_H (\tilde{X})}{2\mu_H (\tilde{X}_1 \p \tilde{X})} > \frac{1}{2(2+\tilde{\varepsilon})^{n-1}}.
\end{equation}
 Recall that the support of $X$ with respect to $\mu_G$ consists of $a \in X$ such that $\mu_G(X \cap U) >0 \}$ for all open neighborhoods of $a$. By Fact \ref{fact: modular function}.4, the support of $X$ with respect to $\mu_G$ is the same as that with respect to $\nu_G$, so we can refer to the support of $X$ without ambiguity.
Recall also that $\Delta (\exp(\cdot Z))$ is an isomorphism from $(\Bbb{R},+)$ to $(\Bbb{R}^{>0},\times)$.  Hence there exists a constant $C>0$ only depending on $Z$ such that on the support of $X$ we have $e^{-C\tilde{\varepsilon}} < \Delta < e^{C\tilde{\varepsilon}}$ and on the support of $X^2$ we have $e^{-2C\tilde{\varepsilon}} < \Delta < e^{2C\tilde{\varepsilon}}$. Thus by Fact \ref{fact: modular function}.4, we have $$\frac{\nu_G (X)}{\mu_G (X)} > e^{-C\tilde{\varepsilon}}$$ and $$\frac{\mu_G (X^2)}{\nu_G (X^2)} > e^{-2C\tilde{\varepsilon}}.$$

Combining the above inequalities with (\ref{murelation}), we have
\begin{equation}\label{nurelation}
    \frac{\nu_G(X)}{\nu_G (X^2)} > \frac{e^{-3C\tilde{\varepsilon}}}{2(2+\tilde{\varepsilon})^{n-1}}.
\end{equation}

Hence for the $X$ we constructed,
\begin{equation}\label{munucombinedalmostsharp}
\frac{\nu_G(X)^{\frac{1}{n}-\varepsilon}}{\nu_G(X^2)^{\frac{1}{n}-\varepsilon}}+ \frac{\mu_G(X)^{\frac{1}{n}-\varepsilon}}{\mu_G(X^2)^{\frac{1}{n}-\varepsilon}} >(1+e^{-3C\tilde{\varepsilon}(\frac{1}{n}-\varepsilon)})(2(2+\tilde{\varepsilon})^{n-1})^{-\frac{1}{n}+\varepsilon}.
\end{equation}
It suffices to take $\tilde{\varepsilon}$ small enough such that the right hand side of (\ref{munucombinedalmostsharp}) is $> 1$.
\end{proof}

With the Gleason--Yamabe Theorem and the results developed in Section 2, we are able to pass our Lie group constructions to general locally compact groups. 

\begin{proof}[Proof of Theorem~\ref{thm: mainsharpopen}]
 By Fact~\ref{fact: Gleason}, there is  open subgroup $G'$ of $G$ which is almost-Lie. Since $\mu_G|_{G'}$ is a left Haar measure on $G'$, and same holds $\nu_G|_{G'}$, we may assume without loss of generality that $G$ is almost-Lie.
 
With this assumption, there is 
  a short exact sequence $1 \to H \to G\xrightarrow{\pi} G/H \to 1$ where $H$ is a  compact subgroup, and  $G/H$ is a Lie group. Let $X$ be a subset of $G/H$ such that 
 \begin{equation}\label{eq: Lie case}
 \frac{\nu_{G/H}(X)^{\frac{1}{n}-\varepsilon}}{\nu_{G/H}(X^2)^{\frac{1}{n}-\varepsilon}}+ \frac{\mu_{G/H}(X)^{\frac{1}{n}-\varepsilon}}{\mu_{G/H}(X^2)^{\frac{1}{n}-\varepsilon}} >1, 
 \end{equation}
 where $n$ is the noncompact Lie dimension of $G/H$. Thus by the quotient integral formula, we have 
 \begin{align*}
\mu_G(\pi^{-1}(X))&=\int_{G/H}\mu_H(g^{-1}(\pi^{-1}(X))\cap H)\d\mu_{G/H}(g)\\
&=\int_{G/H}\mathbbm{1}_X(g)\d\mu_{G/H}(g)=\mu_{G/H}(X),
  \end{align*}
 and similarly $\nu_G(\pi^{-1}(X))=\nu_{G/H}(X)$. Observe that $\pi^{-1}(X^2)=\pi^{-1}(X) \cdot \pi^{-1}(X)$. Thus the desired conclusion follows from \eqref{eq: Lie case} and Propositions~\ref{prop: welldefinednomcompactdim}. 
\end{proof}

We now prove Theorem~\ref{thm: mainsharp}.

\begin{proof}[Proof of Theorem~\ref{thm: mainsharp}]

We consider first the case when $n=0$ and $G$ is a Lie group.  Then the identity component $G_0$ of $G$ is compact. Taking $X=G_0$, we have $X =X^2$ has positive left and right measures. In particular, we have $$\mu_G(X)=\mu_G(X^2) \text{ and } \nu_G(X)=\nu_G(X^2),$$  and Theorem \ref{thm: mainsharp} holds in this case. 

Suppose $n=0$ and $G$ is a locally compact group.  By Fact~\ref{fact: Gleason}, there is  open subgroup $G'$ of $G$ which is almost-Lie. Since $\mu_G|_{G'}$ is a left Haar measure on $G'$, and same holds $\nu_G|_{G'}$, we may replace $G$ with $G'$ and assume without loss of generality that $G$ is almost-Lie. Let $H$ be a compact subgroup of $G$ such that $G/H$ is a Lie group, and $\pi: G \to G/H$ is the quotient map. Choose compact $Y \subseteq G/H$ such that $Y^2 =Y$ has positive left and right measure. Choosing $X=\pi^{-1}(Y)$, we have  $X =X^2$ has positive left and right measures, and  Theorem \ref{thm: mainsharp} holds in this case too. 

In the rest of the proof, we assume $n>0$. Using Theorem~\ref{thm: mainsharpopen} we obtain an open subset $Y$ of $G$ such that 
\[
\frac{\nu_G(Y)^{\frac{1}{n}-\varepsilon/2}}{\nu_G(Y^2)^{\frac{1}{n}-\varepsilon/2}}+ \frac{\mu_G(Y)^{\frac{1}{n}-\varepsilon/2}}{\mu_G(Y^2)^{\frac{1}{n}-\varepsilon/2}}>1.
\]
By the inner regularity of $\mu_G$ and $\nu_G$, we can choose $X \subseteq Y$ such that $\mu_G(X) > \mu_G(Y) -\delta$ and $\nu_G(X)> \nu_G(Y)- \delta$ for any given $\delta>0$. On the other hand, for such $X$, we have $\mu_G(X^2) \leq \mu_G(Y^2)$ and $\nu_G(X^2)\leq \nu_G(Y^2)$ as $X^2 \subseteq Y^2$. Hence, by choosing $\delta$ sufficiently small,  we can arrange that $X$ satisfies the desired inequality~\eqref{eq: BM in theorem1.3} .
\end{proof}

\section{Reduction to outer terms of certain short exact sequences}

For $n \in  \ZZ^{\geq0}$ and $(x,y) \in \RR^2$, we set 
$$
\big\Vert (x,y) \big\Vert_{1/n} = 
\begin{cases}
({\color{black}|x|}^{1/n} + {\color{black}|y|}^{1/n})^{n} &\text{ if } n\neq 0,\\
\max\{ {\color{black}|x|}, {\color{black}|y|}\}  &  \text{ if } n=0.
\end{cases} 
$$ 
We say that the group $G$ satisfies  the {\bf Brunn--Minkowski inequality with exponent $n$}, abbreviated as BM($n$), if for all compact $X, Y\subseteq G$,
\[
\Bigg\Vert\Big(\frac{\nu(X)}{\nu(XY)},\frac{\mu(Y)}{\mu(XY)}\Big)\Bigg\Vert_{1/n}\leq1.
\]
When $G$ is unimodular and $n\geq 1$, the above is equivalent to having the inequality $ \mu(XY)^{1/n} \geq  \mu(X)^{1/n} + \mu(Y)^{1/n}.$ Note that $\frac{\nu(X)}{\nu(XY)}\leq 1$ and $\frac{\mu(Y)}{\mu(XY)}\leq 1$. Hence, every locally compact group $G$ satisfies the  Brunn--Minkowski inequality with exponent $n=0$. Moreover, if $n <n'$ and  $G$ satisfies the Brunn--Minkowski inequality with exponent $n'$, then it satisfies the Brunn--Minkowski inequality with exponent $n$.

Given a function $f: \Omega\to \RR$, for every $t\in\RR$, define the \emph{superlevel set} of $f$ 
\[
L^+_{f}(t):=\{x\in \Omega: f(x)\geq t\}.
\]
We will use this notation at various points in the later proofs. We use the following simple  {\color{black}consequence of Fubini's theorem concerning} the superlevel sets~\cite[Theorem 8.16]{Rudin}:
\begin{fact}\label{fact: 4.1}
Let $\mu$ be a positive measure on some $\sigma$-algebra in the set $\Omega$. 
Suppose $f: \Omega\to [0,\infty]$ is a compactly supportted  measurable function. For every $r>0$, 
\[
\int_\Omega f^r(x)\d \mu(x)=\int_{\RR^{\geq0}} rx^{r-1} \mu(L_f^+(x))\d x.
\]
\end{fact}

The next proposition is the main result of this section. The current statement of the proposition is proved by McCrudden as the main result in~\cite{Mccrudden69}. We give a simpler (but essentially the same) proof here for the sake of completeness. 

\begin{proposition}\label{lem: quotient unimodular}
Assume $G$ is a unimodular group, $n_1,n_2\geq0$ are integers, $H$ is a closed normal subgroup of $G$ satisfying $\BM(n_1)$, and the quotient group $G/H$ is  unimodular satisfying $\BM(n_2)$. Then $G$ satisfies $\BM(n_1+n_2)$. 
\end{proposition}
\begin{proof}

Suppose $\Omega$ is a compact subset of $G$. Let the ``fiber length function'' $f_\Omega: G/H\to \RR^{\geq0}$ be a measurable function such that for every $gH\in G/H$, $f_\Omega(gH)=\mu_H(g^{-1}\Omega\cap H)$. The case when both $n_1=n_2=0$ holds trivially.

Now we split the proof into three cases. \medskip

\noindent {\bf Case 1.} \emph{When $n_1\geq1$ and $n_1+n_2\geq2$.}\medskip

By the quotient integral formula (Fact~\ref{fact: Quotient Integral Formula}), we have
\begin{align}
    \mu_G^{1/(n_1+n_2)}(\Omega)&=\left(\int_{G/H}f_{\Omega}(x)\d \mu_{G/H}(x)\right)^{1/(n_1+n_2)} \nonumber \\
    &= \left(\int_{G/H}(f_{\Omega}^{1/{n_1}}(x))^{n_1}\d \mu_{G/H}(x)\right)^{1/(n_1+n_2)}. \nonumber
\end{align}    
Now by Fact~\ref{fact: 4.1} and the fact that $\mu_{G/H}(L^+_{f^{1/{n_1}}_{\Omega}}(t)) = \mu_{G/H}(L^+_{f_{\Omega}}(t^{n_1}))$, we get
    \begin{align}
    \mu_G^{1/(n_1+n_2)}(\Omega)&=\left(\int_{\RR^{>0}} n_1 t^{n_1-1}\mu_{G/H}\big(L^+_{f_{\Omega}}(t^{n_1})\big)\d t\right)^{1/(n_1+n_2)}. \label{eq: mu(AB) when n_1>1 and n_2>1}
\end{align}

Set $\alpha=\frac{n_1-1}{n_1+n_2-1}$, $\beta=\frac{n_2}{n_1+n_2-1}$, $\gamma = n_1+n_2-1$,  and 
\[
F_\Omega(t)=t^\alpha \mu_{G/H}^{\beta/n_2}\left(L_{f_\Omega}^+(t^{n_1})\right),
\]
for compact set $\Omega$ in $G$ and $t>0$ (Note that $F_\Omega$ is well-defined when $n_2=0$). Then (\ref{eq: mu(AB) when n_1>1 and n_2>1}) can be rewritten as 
\begin{equation}
    n_1^{-1/(\gamma+1)}\mu_G^{1/(\gamma+1)}(\Omega)=\left( \int_{\RR^{>0}}F_\Omega^\gamma(t)\d t\right)^{1/(\gamma+1)}. \label{eq: mu(AB) when n_1>1 and n_2>1 rewrite}
\end{equation}
Fix nonempty compact sets $X,Y \subseteq G$. By \eqref{eq: mu(AB) when n_1>1 and n_2>1 rewrite}, we need to show that 
\begin{equation}
    \left( \int_{\RR^{>0}}F_{XY}^\gamma(t)\d t\right)^{1/(\gamma+1)} \geq  \left( \int_{\RR^{>0}}F_X^\gamma(t)\d t\right)^{1/(\gamma+1)} +  \left( \int_{\RR^{>0}}F_Y^\gamma(t)\d t\right)^{1/(\gamma+1)}.  \label{eq: What we want case 1}
\end{equation}
We will do so in two steps. First, we will show the following convexity property 
\begin{equation}\label{eq: F convex}
   F_{XY}(t_1+t_2) \geq F_X(t_1)+F_Y(t_2).
\end{equation}
For every $t_1,t_2\in \RR^{>0}$, since $H$ satisfies $\BM(n_1)$, by definition we have
\begin{equation}\label{eq: containment of f_XY}
L^+_{f_X}(t_1^{n_1})L^+_{f_Y}(t_2^{n_1})\subseteq L^+_{f_{XY}}\left(\big(t_1+t_2\big)^{n_1}\right).
\end{equation}
Also, since $G/H$ satisfies $\BM(n_2)$, we have
\begin{equation}\label{eq: superlevel when n_1>1 and n_2>1}
\mu_{G/H}^{1/n_2}\left(L^+_{f_X}(t_1^{n_1})\right)+\mu_{G/H}^{1/n_2}\left(L^+_{f_Y}(t_2^{n_1})\right)\leq \mu_{G/H}^{1/n_2}\left(L^+_{f_{XY}}\left(\big(t_1+t_2\big)^{n_1}\right)\right). 
\end{equation}
Now by \eqref{eq: superlevel when n_1>1 and n_2>1},
\begin{align*}
&\, (t_1+t_2)^{n_1-1}\left(\mu_{G/H}^{1/n_2}\big(L^+_{f_{XY}}((t_1+t_2)^{n_1})\big)\right)^{n_2} \\
\geq&\,  (t_1+t_2)^{n_1-1}\left(\mu_{G/H}^{1/n_2}\left(L^+_{f_X}(t_1^{n_1})\right)+\mu_{G/H}^{1/n_2}\left(L^+_{f_Y}(t_2^{n_1})\right)\right)^{n_2}\\
=&\, \big\Vert \left(t_1^{\alpha},t_2^{\alpha}\right)  \big\Vert_{1/\alpha}^{\gamma}
\left\Vert \left(\mu_{G/H}^{\beta/n_2}\left(L^+_{f_X}(t_1^{n_1})\right) , \mu_{G/H}^{\beta/n_2}\left(L^+_{f_Y}(t_2^{n_1})\right)\right)  \right\Vert_{1/\beta}^{\gamma},
\end{align*} 
applying H\"older's inequality with exponents $1/\alpha$ and $1/\beta$, and using the fact that $n_1\geq1$ and $n_1+n_2\geq2$ (which would imply $1/\alpha, 1/\beta \in [1,\infty]$), we obtain
\begin{align*}
&\, (t_1+t_2)^{n_1-1}\left(\mu_{G/H}^{1/n_2}\big(L^+_{f_{XY}}((t_1+t_2)^{n_1})\big)\right)^{n_2} \\
\geq&\, \left(t_1^{\alpha}\mu_{G/H}^{\beta/n_2}\left(L^+_{f_X}(t_1^{n_1})\right)+t_2^{\alpha}\mu_{G/H}^{\beta/n_2}\left(L^+_{f_Y}(t_2^{n_1})\right)\right)^{\gamma}.
\end{align*} 
We remark that the above inequalities also make sense when $n_2=0$. In that case $\Vert (a,b)\Vert_{1/n_2}$ is to be understood as $\max\{a,b\}$ for every $a,b\in\RR^{\geq 0}$. The first line of the above inequality is $F_{XY}^\gamma (t_1+t_2)$ and the last line is $(F_X(t_1)+F_Y(t_2))^\gamma$. So we finished the first step.

We now prove \eqref{eq: What we want case 1}. As a consequence of the above convexity property~\eqref{eq: F convex}, when $a,b\in\RR$ and $F_X(a)\geq s_1$, $F_Y(b)\geq s_2$, we have $F_{XY}(a+b)\geq s_1+s_2$. This implies
\[
L^+_{F_X}(s_1)+L^+_{F_Y}(s_2)\subseteq L^+_{F_{XY}}(s_1+s_2).
\]
Thus using Kneser's inequality~\cite{kneser} for $\RR$ (i.e. the Brunn--Minkowski inequality for $\RR$), we have
\begin{equation}\label{eq: convexity of F_AB}
   \mu_\RR\big( L^+_{F_{XY}} ({\color{black}s_1}+{\color{black}s_2})\big) \geq \mu_\RR\big( L^+_{F_{X}} ({\color{black}s_1})\big)+\mu_\RR\big( L^+_{F_{Y}} ({\color{black}s_2})\big).
\end{equation}
Let $M_X=\mathrm{ess}\sup_{x} F_X(x)$, $M_Y=\mathrm{ess}\sup_{x} F_Y(x)$.
By Fact~\ref{fact: 4.1} we have
\begin{align}
    \int_{\RR^{>0}} F^\gamma_{XY}(s)\d s&\geq \int_0^{M_X+M_Y} \gamma {\color{black}s}^{\gamma-1} \mu_{\RR}\big( L^+_{F_{XY}}({\color{black}s})\big)\d {\color{black}s} \nonumber\\
    & = (M_X+M_Y)^{\gamma}\int_0^1 \gamma {\color{black}s}^{\gamma-1}\mu_\RR\big( L^+_{F_{XY}} (M_X{\color{black}s}+M_Y{\color{black}s}) \big)\d {\color{black}s}. \nonumber
\end{align}
By \eqref{eq: convexity of F_AB}, the above quantity is at least
\begin{align}
&(M_X+M_Y)^{\gamma}\left(\int_0^1 \gamma {\color{black}s}^{\gamma-1}\mu_\RR\big( L^+_{F_{X}} (M_X{\color{black}s}) \big)\d {\color{black}s}+\int_0^1 \gamma {\color{black}s}^{\gamma-1}\mu_\RR\big( L^+_{F_{Y}} (M_Y{\color{black}s}) \big)\d s\right) \nonumber\\
=&\, (M_X+M_Y)^{\gamma}\left( \frac{1}{M_X^\gamma}\int_{\RR^{>0}}F^\gamma_X({\color{black}s})\d {\color{black}s}+\frac{1}{M_Y^\gamma}\int_{\RR^{>0}}F^\gamma_Y({\color{black}s})\d {\color{black}s} \right), \label{eq: one dim F_AB}
\end{align}
and the latter is a change of variables. 

Finally, by \eqref{eq: mu(AB) when n_1>1 and n_2>1}{\color{black} and \eqref{eq: one dim F_AB}},
\begin{align*}
    &\,n_1^{-1/(\gamma+1)}\mu_G^{1/(\gamma+1)}(XY)
    =\left( \int_{\RR^{>0}} F^{\gamma}_{XY}(t)\d t\right)^{1/(\gamma+1)}\\
    =&\, \left((M_X+M_Y)^{\gamma}\left( \frac{1}{M_X^\gamma}\int_{\RR^{>0}}F^\gamma_X({\color{black}s})\d {\color{black}s}+\frac{1}{M_Y^\gamma}\int_{\RR^{>0}}F^\gamma_Y({\color{black}s})\d {\color{black}s} \right)\right)^{1/(\gamma+1)}.
\end{align*}    
By the fact $n_1+n_2\geq2$ we see that $\gamma\geq1$. Applying H\"older's inequality with exponents $(\gamma+1)/\gamma$ and $\gamma+1$, the above quantity is 
    \begin{align*}
    &\, \left\Vert \Big(M_X^{\frac{\gamma}{\gamma+1}},M_Y^{\frac{\gamma}{\gamma+1}}\Big)\right\Vert_{\frac{\gamma+1}{\gamma}} \cdot\left\Vert\left( \left(\frac{1}{M_X^\gamma}\int_{\RR^{>0}}F^\gamma_X(t)\d t\right)^{\frac{1}{\gamma+1}},\left(\frac{1}{M_Y^\gamma}\int_{\RR^{>0}}F^\gamma_Y(t)\d t \right)^{\frac{1}{\gamma+1}} \right)\right\Vert_{\gamma+1}\\
    \geq &\, \left( \int_{\RR^{>0}}F_X^\gamma(t)\d t\right)^{1/(\gamma+1)} +\left( \int_{\RR^{>0}}F_Y^\gamma(t)\d t\right)^{1/(\gamma+1)}\\
    =&\, n_1^{-1/(\gamma+1)}\mu_G^{1/(\gamma+1)}(X)+n_1^{-1/(\gamma+1)}\mu_G^{1/(\gamma+1)}(Y),
\end{align*}
this proves the case 1.\medskip

\noindent {\bf Case 2.} \emph{When $n_1= 1$ and $n_2=0$.}\medskip

In this case, the conclusion can be derived from \eqref{eq: mu(AB) when n_1>1 and n_2>1} directly. In particular,
using the fact that $G/H$ satisfies $\BM(0)$ and  $H$ satisfies $\BM(1)$, we have
\[
 \mu_{G/H}\left(L^{+}_{f_{XY}}\left(t_1+t_2\right)\right)\geq \max \left\{ \mu_{G/H}\left( L^+_{f_X}(t_1) \right), \mu_{G/H}\left( L^+_{f_Y}(t_2) \right)\right\}.
\]
Let $N_X=\sup_t f_X(t)$ and $N_Y=\sup_t f_Y(t)$. 
Therefore, by H\"older's inequality,
\begin{align*}
     \mu_G(XY)&= \int_{\RR^{>0}}\mu_{G/H}(L^+_{f_{XY}}(t))\d t \\
    &= \int_{\RR^{>0}} (N_X+N_Y) \mu_{G/H}(L^+_{f_{XY}}((N_X+N_Y)t))\d t \\
    &\geq(N_X+N_Y) \max\left\{  \int_0^1 \mu_{G/H}(L^+_{f_X}(N_Xt)\d t,\int_0^1  \mu_{G/H}(L^+_{f_Y}(N_Yt) \d t\right\}\\
    &\geq N_X\int_0^1 \mu_{G/H}(L^+_{f_X}(N_Xt))\d t  + N_Y\int_0^1 \mu_{G/H}(L^+_{f_Y}(N_Yt))\d t \\
    &= \mu_G(X) + \mu_G(Y).
\end{align*}
Thus $G$ satisfies $\BM(1)$.\medskip

\noindent {\bf Case 3.} \emph{When $n_1=0$ and $n_2\geq 1$.}\medskip

Applying the Brunn--Minkowski inequality with exponent $0$ on $H$, and the fact that $G/H$ satisfies $\BM(n_2)$, we obtain
\begin{equation}\label{eq: G when n_1=0}
 \mu_{G/H}^{1/n_2}\left(L^{+}_{f_{XY}}\left(\max\{t_1,t_2\}\right)\right)\geq   \mu_{G/H}^{1/n_2}\left( L^+_{f_X}(t_1) \right)+ \mu_{G/H}^{1/n_2}\left( L^+_{f_Y}(t_2) \right).
\end{equation}
Given a compact set $\Omega$ in $G$, we define 
\[
E_\Omega(t)=\mu_{G/H}^{1/n_2}(L^+_{f_\Omega}(t)),\quad  t > 0. 
\]
Clearly $E_\Omega$ is a non-increasing function. 
By \eqref{eq: G when n_1=0}, we have 
\begin{equation}\label{eq: property for E}
    E_{XY}(\max\{a_1,a_2\})\geq E_X(a_1)+E_Y(a_2)
\end{equation}
 for all $a_1,a_2$. This can be seen as a ``convexity property'' for $E$, but the maximum operator inside the function $E$ prevents us from using the same argument as used in Case 1 for $F$. On the other hand, using \eqref{eq: property for E}, we see that if $a,b\in\RR$ and $E_X(a)\geq s_1$ and $E_Y(b)\geq s_2$, then $E_{XY}(\max\{a,b\})\geq s_1+s_2$. Hence $L^+_{E_{X}}(s_1) \cup L^+_{E_{Y}}(s_2) \subseteq L^+_{E_{XY}}(s_1+s_2) $, and thus
\begin{equation}\label{relationofmuR}
\mu_\RR(L^+_{E_{XY}}({\color{black}s}_1+{\color{black}s}_2))\geq \max\{\mu_\RR(L^+_{E_{X}}({\color{black}s}_1)),\mu_\RR(L^+_{E_{Y}}({\color{black}s}_2))\}.
\end{equation}
Now we consider $\mu_G(XY)$. We have
\begin{align}\label{eq: the equation below 21}
    \mu_G^{1/n_2}(XY)
    =\left(\int_{\RR^{>0}}E_{XY}^{n_2}(s)\d s \right)^{1/n_2}
    =\left(\int_{\RR^{>0}} n_2{\color{black}s}^{n_2-1}\mu_{\RR}(L_{E_{XY}}^+({\color{black}s}))\d {\color{black}s} \right)^{1/n_2}.
\end{align}    
Let $P_X=\mathrm{ess}\sup_t E_X(t)$ and $P_Y=\mathrm{ess}\sup_t E_Y(t)$.  By \eqref{relationofmuR} and \eqref{eq: the equation below 21} we see
\begin{align*}
&\,n_2^{-1/n_2}\mu_G^{1/n_2}(XY)\\
   \geq &\,\left(\!(P_X+P_Y)^{n_2} \max\left\{  \int_0^1 {\color{black}s}^{n_2-1} \mu_{\RR}(L^+_{E_X}(P_X {\color{black}s}))\d {\color{black}s},\int_0^1 {\color{black}s}^{n_2-1} \mu_{\RR}(L^+_{E_Y}(P_Y {\color{black}s})) \d {\color{black}s}\right\}\!\right)^{1/n_2}\\
    \geq&\, \left(P_X^{n_2}\int_0^1 {\color{black}s}^{n_2-1}\mu_{\RR}(L^+_{E_X}(P_X {\color{black}s}))\d {\color{black}s} \right)^{1/n_2} + \left(P_Y^{n_2}\int_0^1 {\color{black}s}^{n_2-1}\mu_{\RR}(L^+_{E_Y}(P_Y {\color{black}s}))\d {\color{black}s} \right)^{1/n_2} \\
    =&\, n_2^{-1/n_2}\mu_G^{1/n_2}(X) + n_2^{-1/n_2}\mu_G^{1/n_2}(Y).
\end{align*}
This proves the case when $n_1=0$, and hence finishes the proof of the proposition.
\end{proof}

\section{Reduction to unimodular subgroups}
The main result of this section allows us to obtain a Brunn--Minkowski inequality for a nonunimodular group from a certain unimodular normal subgroup. 
We use $\mu_{\RR^\times}$ to denote a Haar measure on the multiplicative group $(\RR^{>0},\times)$. The next lemma concerns the case when the modular function on $X$ and on $Y$ is ``nearly constant''. 

\begin{lemma}\label{lem: similar modular value}
Suppose $G$ is a locally compact group, and the modular function $\Delta_G: G \to (\RR^{>0},\times)$ is a  quotient map of topological groups. Let $X,Y$ be compact subsets of $G$, and parameters $a,b,\varepsilon>0$ and $n\geq0$ an integer, such that for every $x\in X$, $\Delta_G(x)\in[a,a+\varepsilon)$ and for every $y\in Y$, $\Delta_G(y)\in[b,b+\varepsilon)$. Suppose $H=\ker(\Delta_G)$ satisfies $\BM(n)$. Then 
\[
\frac{\nu_G(X)^{1/(n+1)}}{\nu_G(XY)^{1/(n+1)}}+ \frac{\mu_G(Y)^{1/(n+1)}}{\mu_G(XY)^{1/(n+1)}}\leq 1+f(\varepsilon),
\]
where $f(\varepsilon)$ is an explicit function depending only on $a,b, n$ and $\varepsilon$, and $f(\varepsilon)\to 0$ as $\varepsilon\to 0$. Moreover, this convergence is uniform when $n$ is fixed and $a$ and $b$ vary over compact sets.
\end{lemma}
\begin{proof}
We first consider the case when $n\geq1$. Fix Haar measures $\mu_H,\mu_{\RR^\times}$ on $H$ and on $(\RR^{>0},\times)$, and these two measures will also uniquely determine a left Haar measure $\mu_G$ on $G$ and a right Haar measure $\nu_G$ on $G$ via the quotient integral formula.  For every compact subset $\Omega$ of $G$, define two functions $\ell_\Omega,r_\Omega:(\RR^{>0},\times)\to \RR^{\geq0}$ such that
\[
\ell_\Omega(g)=\mu_H(g^{-1}\Omega\cap H), \text{ and } r_\Omega(g)=\mu_H(\Omega g^{-1}\cap H).
\]
Note that given $g_1,g_2$ in $G$, $(X\cap Hg_1)(Y\cap g_2H)$ lies in 
\begin{equation}\label{eq: relation in lem 5.1}
Hg_1g_2H=(g_1g_2)(g_1g_2)^{-1}H(g_1g_2)H = H(g_1g_2)H(g_1g_2)^{-1}(g_1g_2)
\end{equation}
since $H$ is normal.

For every compact sets $X_1,X_2$ in $H$, and $g_1,g_2$ in $G$, by \eqref{eq: relation in lem 5.1}, $X_1g_1g_2X_2\subseteq g_1g_2H$. By Fact~\ref{fact: surjectivemodularonconnectedcomp}.2 and the fact that $H$ satisfies $\BM(n)$, we have
\begin{align}\label{eq: n>0 nonunimodular}
\mu^{1/n}_H((g_1g_2)^{-1}X_1g_1g_2X_2)&\geq \mu^{1/n}_H((g_1g_2)^{-1}X_1g_1g_2)+\mu^{1/n}_H(X_2) \nonumber \\
&=(\Delta_G(g_1)\Delta_G(g_2))^{-1/n}\mu^{1/n}_H(X_1)+\mu^{1/n}_H(X_2).
\end{align}
In light of this, for every $t_1,t_2\geq 0$ we have
\[
 L^+_{\ell_{XY}}\left(\Big(\inf_{x\in X, y\in Y}(\Delta_G(x)\Delta_G(y))^{-1/n}t_1+t_2\Big)^{n}\right)\supseteq L^+_{r_X}(t_1^n)+L^+_{\ell_Y}(t_2^n).
\]
Applying the Brunn--Minkowski inequality on $(\RR^{>0},\times)$,  we get
\begin{align}\label{eq: first inequality in lemma 5.1}
    &\mu_{\RR^\times}\left( L^+_{\ell_{XY}}\left(\Big(\inf_{x\in X, y\in Y}(\Delta_G(x)\Delta_G(y))^{-1/n}t_1+t_2\Big)^{n}\right) \right)\geq \mu_{\RR^\times}(L^+_{r_X}(t_1^n))+\mu_{\RR^\times}(L^+_{\ell_Y}(t_2^n)).
    \end{align}
    Similarly, for the right Haar measure on $H$, we have
\begin{align}    \label{eq: second inequality in lemma 5.1}
    &\mu_{\RR^\times}\left( L^+_{r_{XY}}\left(\Big(t_1+\inf_{x\in X,y\in Y}(\Delta_G(x)\Delta_G(y))^{1/n}t_2\Big)^{n}\right) \right)\geq \mu_{\RR^\times}(L^+_{r_X}(t_1^{n}))+\mu_{\RR^\times}(L^+_{\ell_Y}(t_2^{n})).
\end{align}
Let $M_X=\sup_x \mu_{\RR^\times}(L^+_{r_X}(x))$ and $M_Y=\sup_y \mu_{\RR^\times}(L^+_{\ell_Y}(y))$. By a change of variables and Fact~\ref{fact: 4.1}, \begin{align*}
    \mu_G(XY)=\int_{\RR^\times}\mu_H(g^{-1}XY\cap H)\d\mu_{\RR^\times}(g)=\int_{\RR^{>0}}nt^{n-1}\mu_{\RR^\times}(L^+_{\ell_{XY}}(t^{n}))\d t.
\end{align*}
By a change of variables the above quantity can be rewritten as 
\begin{align*}
    &\Big\Vert \Big(\frac{1}{(a+\varepsilon)(b+\varepsilon)}M_X,M_Y\Big)\Big\Vert_{1/n}\\
    &\quad \cdot\int_0^1 nt^{n-1}\mu_{\RR^\times}\left(L_{\ell_{XY}}^+\left(\Big(\Big(\frac{1}{(a+\varepsilon)(b+\varepsilon)}M_X\Big)^{1/n}t+M_Y^{1/n}t\Big)^n\right)\right)\d t.
\end{align*}
Using~\eqref{eq: first inequality in lemma 5.1}, the above quantity is at least 
\begin{align*}
   & \Big\Vert \Big(\frac{1}{(a+\varepsilon)(b+\varepsilon)}M_X,M_Y\Big)\Big\Vert_{1/n}\\
    &\quad \cdot\left(\int_0^1 nt^{n-1}\mu_{\RR^\times}\left(L_{r_{X}}^+(M_X t^n)\right)\d t + \int_0^1 nt^{n-1}\mu_{\RR^\times}\left(L_{\ell_{Y}}^+(M_Y t^n)\right)\d t \right)\\
    =&\, \Big\Vert \Big(\frac{1}{(a+\varepsilon)(b+\varepsilon)}M_X,M_Y\Big)\Big\Vert_{1/n}\left(\frac{1}{M_X}\nu_G(X)+\frac{1}{M_Y}\mu_G(Y)\right).
\end{align*}
Thus applying H\"older's inequality with exponents $(n+1)/n$ and $n+1$, we get
\begin{align}
   &\, \mu_G^{1/(n+1)}(XY) \nonumber\\
    \geq&\, \left\Vert \left(\Big(\frac{1}{(a+\varepsilon)(b+\varepsilon)}M_X\Big)^{\frac{1}{n+1}},M_Y^{\frac{1}{n+1}}\right)\right\Vert_{\frac{n+1}{n}}\left\Vert\left(\Big(\frac{1}{M_X}\nu_G(X)\Big)^{\frac{1}{n+1}}, \Big(\frac{1}{M_Y}\mu_G(Y)\Big)^{\frac{1}{n+1}}\right)\right\Vert_{n+1} \nonumber\\
    \geq&\, \left(\frac{1}{(a+\varepsilon)(b+\varepsilon)}\nu_G(X)\right)^{1/(n+1)}+\mu^{1/(n+1)}_G(Y).\label{eq: bounded modular muXY}
\end{align}

Similarly, for $\nu_G(XY)$ by~\eqref{eq: second inequality in lemma 5.1} and a change of variables we get
\begin{align*}
    \nu_G(XY)&=\int_{\RR^\times}\mu_H(XYg^{-1}\cap H)\d\mu_{\RR^\times}(g)\\
    &=\int_{\RR^{>0}}nt^{n-1}\mu_{\RR^\times}(L^+_{r_{XY}}(t^{n}))\d t\\
    &\geq \Big\Vert \Big(M_X,abM_Y\Big)\Big\Vert_{1/n}\left(\frac{1}{M_X}\nu_G(X)+\frac{1}{M_Y}\mu_G(Y)\right),
\end{align*}
 and using H\"older's inequality in a similar way as in~\eqref{eq: bounded modular muXY}, we obtain 
\begin{align}\label{eq: bounded modular nuXY}
    \nu_G^{1/(n+1)}(XY)\geq \nu_G^{1/(n+1)}(X)+\big(ab\mu_G(Y)\big)^{1/(n+1)}.
\end{align}

Therefore, combining \eqref{eq: bounded modular muXY} and \eqref{eq: bounded modular nuXY}, we conclude
\begin{align*}
    &\, \frac{\nu_G^{1/(n+1)}(X)}{\nu_G^{1/(n+1)}(XY)}+\frac{\mu_G^{1/(n+1)}(Y)}{\mu_G^{1/(n+1)}(XY)}\\
    \leq&\, \frac{1}{1+(Cab)^{1/(n+1)}}+\frac{1}{1+\big(\frac{1}{C(a+\varepsilon)(b+\varepsilon)}\big)^{1/(n+1)}}\\
    \leq&\, 1+\frac{(C(ab+\varepsilon(a+b+\varepsilon)))^{1/(n+1)}-(Cab)^{1/(n+1)}}{(1+(Cab)^{1/(n+1)})(1+(C(a+\varepsilon)(b+\varepsilon))^{1/(n+1)})}.
\end{align*}
where $C=\mu_G(Y)/\nu_G(X)$.

Hence
\begin{align*}
    \frac{\nu_G^{1/(n+1)}(X)}{\nu_G^{1/(n+1)}(XY)}+\frac{\mu_G^{1/(n+1)}(Y)}{\mu_G^{1/(n+1)}(XY)}\leq 1+f(\varepsilon)
\end{align*}
where $$f(\varepsilon) = \sup_{r>0} \frac{(r(ab+\varepsilon(a+b+\varepsilon)))^{1/(n+1)}-(rab)^{1/(n+1)}}{(1+(rab)^{1/(n+1)})(1+(r(a+\varepsilon)(b+\varepsilon))^{1/(n+1)})}$$
depends only on $a, b, n$ and $\varepsilon$ and we see $\lim_{\varepsilon \rightarrow 0} f(\varepsilon) = 0$ uniformly when $a,b$ takes values in a compact set by an elementary computation.\medskip

The remaining case is when $n=0$. Note that in this case, inequality \eqref{eq: n>0 nonunimodular} becomes
\begin{align*}
    \mu_H((g_1g_2)^{-1}X_1g_1g_2X_2)\geq \max\{(\Delta_G(g_1)\Delta_G(g_2))^{-1}\mu_H(X_1),\mu_H(X_2)\}.
\end{align*}
This implies for every $t_1,t_2$,
\begin{align*}
        \mu_{\RR^\times}\left( L^+_{\ell_{XY}}\max\left\{\inf_{x\in X, y\in Y}(\Delta_G(x)\Delta_G(y))^{-1}t_1,t_2\right\} \right)\geq \mu_{\RR^\times}(L^+_{r_X}(t_1))+\mu_{\RR^\times}(L^+_{\ell_Y}(t_2)).
\end{align*}
For any compact set $\Omega$ in $G$, define two functions $\Phi_\Omega,\Psi_\Omega: \RR\to\RR$, by
\[
\Phi_\Omega(t)=\mu_{\RR^\times}(L^+_{\ell_\Omega}(t)), \quad\text{and}\quad \Psi_\Omega(t)=\mu_{\RR^\times}(L^+_{r_\Omega}(t)). 
\]
Thus we have
\[
\mu_\RR(L^+_{\Phi_{XY}}(t_1+t_2))\geq \max\left\{\inf_{x\in X, y\in Y}(\Delta_G(x)\Delta_G(y))^{-1} \mu_\RR(L^+_{\Psi_X}(t_1)),\mu_\RR(L^+_{\Phi_Y}(t_2))\right\}.
\]
Let $N_X=\sup_x \mu_{\RR}(L^+_{\Psi_X}(x))$ and $N_Y=\sup_y \mu_{\RR}(L^+_{\Phi_Y}(y))$. By a change of variable, for $\mu_G(XY)$ we have
\begin{align}
    \mu_G(XY)&=\int_{\RR^{>0}}\mu_\RR(L^+_{\Phi_{XY}}(t)) \d t \nonumber\\
    &\geq (N_X+N_Y)\max\left\{\frac{1}{(a+\varepsilon)(b+\varepsilon)}\frac{\nu_G(X)}{N_X},\frac{\mu_G(Y)}{N_Y} \right\} \nonumber\\
    &\geq \frac{1}{(a+\varepsilon)(b+\varepsilon)}\nu_G(X)+\mu_G(Y).  \label{eq: n=0 nonunimodular left case}
\end{align}

Similarly, for every $t_1,t_2$ we also have
\begin{align*}
        \mu_{\RR^\times}\left( L^+_{r_{XY}}\max\left\{t_1,\inf_{x\in X, y\in Y}\Delta_G(x)\Delta_G(y)t_2\right\} \right)\geq \mu_{\RR^\times}(L^+_{r_X}(t_1))+\mu_{\RR^\times}(L^+_{\ell_Y}(t_2)),
\end{align*}
which implies
\[
\mu_\RR(L^+_{\Psi_{XY}}(t_1+t_2))\geq \max\left\{ \mu_\RR(L^+_{\Psi_X}(t_1)),\inf_{x\in X, y\in Y}\Delta_G(x)\Delta_G(y)\mu_\RR(L^+_{\Phi_Y}(t_2))\right\}.
\]
Therefore, for $\nu_G(XY)$ we get
\[
\nu_G(XY)\geq \nu_G(X)+ab\mu_G(Y).
\]
Together with \eqref{eq: n=0 nonunimodular left case}, as in the case when $n\geq1$, we get
\begin{align*}
    \frac{\nu_G(X)}{\nu_G(XY)}+\frac{\mu_G(Y)}{\mu_G(XY)}&\leq \frac{1}{1+Cab}+\frac{1}{1+\frac{1}{C(a+\varepsilon)(b+\varepsilon)}}\\
    &\leq 1+\varepsilon\frac{ C(a+b+\varepsilon)}{(1+Cab)(1+C(a+\varepsilon)(b+\varepsilon))},
\end{align*}
where $C=\mu_G(Y)/\nu_G(X)$. The conclusion follows by taking
\[
f(\varepsilon)=\sup_{r>0}\frac{\varepsilon r(a+b+\varepsilon)}{(1+rab)(1+r(a+\varepsilon)(b+\varepsilon))},
\]
and we can see that $f(\varepsilon)\to 0$ as $\varepsilon\to 0$ uniformly when $a,b$ taken values in a compact set by elementary computations.
\end{proof}

The next proposition is the main result of the section. As mentioned in the introduction, the proof uses a discretized ``spillover'' method. We remark that one can always make the proof continuous as we did in Section 4, but we give a discrete proof here since we believe this reflects our idea in a clearer way.

\begin{proposition}\label{prop: reduce to unimodular}
Suppose  $G$ is a locally compact group with $H=\ker(\Delta_G)$ satisfying $\BM(n)$. Suppose the map $\Delta_G: G \to (\RR^{>0},\times)$ is a  quotient map of topological groups, then $G$ satisfies $\BM(n+1)$.
\end{proposition}
\begin{proof}
Since $X$ and $Y$ are compact, there are $a_1,a_2,b_1$ and $b_2>0$, such that 
\[
a_1=\inf_{x\in X}\Delta_G(x),\ a_2=\sup_{x\in X}\Delta_G(x),\ b_1=\inf_{y\in Y}\Delta_G(y),\ b_2=\sup_{y\in Y}\Delta_G(y). 
\]
We fix $\mu_G$ and $\nu_G$ as in the proof of Lemma~\ref{lem: similar modular value}, and let $\varepsilon>0$ be a sufficiently small number (depending on $a_1,a_2,b_1$ and $b_2$). Then by Fact \ref{fact: Quotient Integral Formula} and familiar properties of integrable functions on $\RR$, there is an $N>0$, such that we can partition $[a_1,a_2]$ and $[b_1,b_2]$ into $N$ subintervals, that is
\[
[a_1,a_2]=\bigcup_{i=1}^N A_i, \quad [b_1,b_2]=\bigcup_{i=1}^N B_i,
\]
such that each subinterval has length at most $\varepsilon$, and the intersection of $X$ with $\Delta_G^{-1}(A_i)$ has $\nu_G$-measure $\nu_G(X)/N$, the intersection of $Y$ with $\Delta_G^{-1}(B_i)$ has $\mu_G$-measure $\mu_G(Y)/N$, for every $1\leq i\leq N$. 

Let $X_i=X\cap \Delta_G^{-1}(A_i)$ and let $Y_i=Y\cap \Delta_G^{-1}(B_i)$. Then  $\nu_G(X)=\sum_{i=1}^N \nu_G(X_i)$ and $\mu_G(Y)=\sum_{i=1}^N \mu_G(Y_i)$. In particular, we have $\mu_G(XY)\geq \sum_{i=1}^N \mu_G(X_iY_i)$ and $\nu_G(XY)\geq \sum_{i=1}^N \nu_G(X_iY_i)$. Observe that given $1\leq i,j\leq N$ and $i\neq j$, $X_iY_i$ and $X_jY_j$ are disjoint. Indeed, the modulus of every element in $X_iY_i$ lies in $A_iB_i$ and the modulus of every element in $X_jY_j$ lies in $A_jB_j$. But $A_iB_i$ and $A_jB_j$ are disjoint subsets of $\Bbb{R}^{>0}$ when $i \neq j$. 

By Lemma~\ref{lem: similar modular value}, for every $1\leq i\leq N$, there is a function $f_i(\varepsilon)$, such that $f_i(\varepsilon) \rightarrow 0$ when $\varepsilon \rightarrow 0$  uniformly, and
\[
\frac{\nu_G^{1/(n+1)}(X_i)}{\nu_G^{1/(n+1)}(X_iY_i)}+\frac{\mu_G^{1/(n+1)}(Y_i)}{\mu_G^{1/(n+1)}(X_iY_i)}\leq 1+f_i(\varepsilon). 
\]
Take $\tilde{f}(\varepsilon)=\sup_i f_i(\varepsilon)$,  hence $\tilde{f}(\varepsilon)\to 0$ as $\varepsilon\to 0$. Therefore, for every $1\leq t\leq N$,
\begin{align}\label{eq: inter step}
    \frac{\nu_G^{1/(n+1)}(X)}{\nu_G^{1/(n+1)}(XY)}+\frac{\mu_G^{1/(n+1)}(Y)}{\mu_G^{1/(n+1)}(XY)}
    \leq  \left(\frac{N\nu_G(X_t)}{\sum_{i=1}^N\nu_G(X_iY_i)}\right)^{\frac{1}{n+1}}+\left(\frac{N\mu_G(Y_t)}{\sum_{i=1}^N\mu_G(X_iY_i)}\right)^{\frac{1}{n+1}}. 
\end{align}    
Also by H\"older's inequality with exponents $\frac{n+2}{n+1}$ and $n+2$, we observe that for every $t$ we have
\begin{align*}
   &\,\left( \sum_{i=1}^N\left( \frac{\nu_G(X_i)}{\nu_G(X_iY_i)} \right)^{\frac{1}{n+2}\cdot\frac{n+2}{n+1}}\right)^{\frac{n+1}{n+2}}\left(\sum_{i=1}^N \nu_G(X_iY_i)\right)^{\frac{1}{n+2}} \nonumber\\
   \geq&\,\sum_{i=1}^N \left( \frac{\nu_G(X_i)}{\nu_G(X_iY_i)} \right)^{\frac{1}{n+2}}\nu_G(X_iY_i)^{\frac{1}{n+2}} = N\nu_G^{\frac{1}{n+2}}(X_t). \label{placeholder}
\end{align*}
After simplification, the above inequality can be rewritten as
\begin{equation}\label{placeholder}
   N(N\nu_G(X_t))^{\frac{1}{n+1}}\leq \Big(\sum_{i=1}^N \Big(\frac{\nu_G(X_i)}{\nu_G(X_iY_i)}\Big)^{\frac{1}{n+1}}\Big) \Big(\sum_{i=1}^N\nu_G(X_iY_i)\Big)^{\frac{1}{n+1}} .
\end{equation}
Averaging \eqref{eq: inter step} over all $t$ and using inequality \eqref{placeholder}, we have
\begin{align*}
   &\, \frac{\nu_G^{1/(n+1)}(X)}{\nu_G^{1/(n+1)}(XY)}+\frac{\mu_G^{1/(n+1)}(Y)}{\mu_G^{1/(n+1)}(XY)}\\
    \leq &\, \frac{1}{N}\sum_{i=1}^N\left(\frac{\nu_G(X_i)}{\nu_G(X_iY_i)}\right)^{1/(n+1)}+\frac{1}{N}\sum_{i=1}^N\left(\frac{\mu_G(X_i)}{\mu_G(X_iY_i)}\right)^{1/(n+1)}\leq 1+\tilde{f}(\varepsilon).
\end{align*}
The desired conclusion follows by taking $\varepsilon\to 0$. 
\end{proof}

\section{Reduction to cocompact and codiscrete subgroups}

The main results in this section will help us to reduce the problem to cocompact subgroups or open normal subgroups. We make use of the following integral formula, see~\cite[Proposition 5.26, Consequence 1]{Knapp}.

\begin{fact}\label{fact: integral formula nonunimodular}
Let $G$ be a connected unimodular Lie group. Suppose $S,T$ are closed subgroups of $G$, such that $G=ST$, and the intersection $S \cap T$ is compact. Then there is a left Haar measure $\mu_S$ on $S$ and a right Haar measure $\nu_T$ on $T$, such that
\[
\int_G f(x) \d\mu_G(x) = \int_{S\times T} f(st) \d\mu_S(s)\d\nu_T(t),
\]
for every $f\in C_{c}(G)$. 
\end{fact}

The next proposition allows us to reduce the problem to closed cocompact subgroups with the same noncomapct Lie dimension. 

\begin{proposition} \label{Prop: Cocompact reduction}
Suppose  $G$ is a connected unimodular Lie group, $H$ is a connected closed subgroup of $G$ satisfying $\BM(n)$, $K$ is a connected unimodular subgroup of $G$, such that $G=KH$ and $K\cap H$ is compact. Then $G$ satisfies $\BM(n)$. 
\end{proposition}
\begin{proof}
We assume $n\geq1$, otherwise the result is trivial. Note that both $G$ and $K$ are unimodular. In light of this we will not be using $\nu_G$, $\nu_K$, etc. and only use $\mu_G = \nu_G$ and $\mu_K = \nu_K$ below.

We fix a  Haar measure $\mu_K$ on $K$, and a Haar measure $\mu_G$ on $G$. These measures will also uniquely determine a left Haar measure $\mu_H$ and a right Haar measure $\nu_H$ on $H$ such that we have the integral formula in Fact~\ref{fact: integral formula nonunimodular} and another similar formula involving $\d \mu_H(h)\d\mu_K(k)$. 
For a compact subset $\Omega$ of $G$, we define two functions $r_\Omega,\ell_\Omega: K\to \RR^{\geq0}$, such that 
\[
r_{\Omega}(k):=\nu_H(k\Omega\cap H),\quad \ell_\Omega(k):=\mu_H(\Omega k\cap H),
\]
for every $k\in K$. We also define two bivariate functions $R_\Omega,L_\Omega:K\times K\to\RR^{\geq0}$ that for every $k_1,k_2$ in $K$,
\[
R_{\Omega}(k_1,k_2):=\nu_H(k_1\Omega k_2\cap H),\quad L_\Omega(k_1,k_2):=\mu_H(k_1\Omega k_2\cap H).
\]

Thus Fact~\ref{fact: integral formula nonunimodular} gives us
\[
\mu_G(\Omega)=\int_K\nu_H(k^{-1}\Omega\cap H)\d\mu_K(k)=\int_K\mu_H(\Omega k^{-1}\cap H)\d\mu_K(k). 
\]
We define two probability measures $\pr_X$ and $\pr_Y$ on $K$ in the following way:
\[
\d\pr_X = \frac{r_X \d\mu_K}{\mu_G(X)}, \quad \d\pr_Y = \frac{ \ell_Y \d\mu_K}{\mu_G(Y)}. 
\]
Now, we choose a left coset $k_1H$ of $H$ in $G$ randomly with respect to the probability measure $\pr_X$, and choose a right coset $Hk_2$ of $H$ in $G$  randomly with respect to the probability measure $\pr_Y$. By the fact that $H$ satisfies $\BM(n)$, we get
\[
\left( \frac{r_X(k_1)}{R_{XY}(k_1,k_2)} \right)^{1/n}+\left( \frac{\ell_Y(k_2)}{L_{XY}(k_1,k_2)} \right)^{1/n}\leq 1. 
\]
This implies
\begin{equation}\label{eq: key}
\mathbb E_{\pr_X(k_1)}\mathbb E_{\pr_Y(k_2)}\left[ \left( \frac{r_X(k_1)}{R_{XY}(k_1,k_2)} \right)^{1/n}+\left( \frac{\ell_Y(k_2)}{L_{XY}(k_1,k_2)} \right)^{1/n} \right]\leq1. 
\end{equation}

On the other hand, by the definition of $\pr_X$,
\begin{equation}\label{eq: expections}
\mathbb E_{\pr_X(k_1)}\left( \frac{r_X(k_1)}{R_{XY}(k_1,k_2)} \right)^{\frac{1}{n}}= \frac{1}{\mu_G(X)}\int_K r_X^{\frac{n+1}{n}}(k_1)R^{-\frac{1}{n}}_{XY}(k_1,k_2) \d\mu_K(k_1).
\end{equation}
Applying H\"older's inequality with exponents $(n+1)/n$ and $n+1$, we get
\begin{align*}
&\,\int_K r_X(k_1)\d\mu_K(k_1) \\
\leq&\, \left(\int_K \Big|r_X(k_1)R^{-\frac{1}{n+1}}_{XY}(k_1,k_2)\Big|^{\frac{n+1}{n}} \d\mu_K(k_1)\right)^{\frac{n}{n+1}}\left(\int_K \Big|R_{XY}^{\frac{1}{n+1}}(k_1,k_2)\Big|^{n+1}\d\mu_K(k_1)\right)^{\frac{1}{n+1}}.
\end{align*}
Thus \eqref{eq: expections} gives us
\begin{align*}
    &\,\mathbb E_{\pr_X(k_1)}\left( \frac{r_X(k_1)}{R_{XY}(k_1,k_2)} \right)^{\frac{1}{n}}\\
\geq&\,\frac{1}{\mu_G(X)}\left( \int_K r_X(k_1)\d\mu_K(k_1) \cdot \left(\int_K R_{XY}(k_1,k_2)\d\mu_K(k_1)\right)^{-\frac{1}{n+1}} \right)^{\frac{n+1}{n}}.
\end{align*}
Finally, using Fact~\ref{fact: integral formula nonunimodular} and the fact that $G$ is unimodular, we conclude that
\[
\mathbb E_{\pr_X(k_1)}\left( \frac{r_X(k_1)}{R_{XY}(k_1,k_2)} \right)^{\frac{1}{n}}\geq \left( \frac{\mu_G(X)}{\mu_G(XY)}\right)^{\frac{1}{n}}.
\]

We have a similar inequality concerning $\mathbb E_{\pr_Y(k_2)}\left( \frac{\ell_Y(k_2)}{L_{XY}(k_1,k_2)} \right)^{\frac{1}{n}}$. Combining both inequalities with \eqref{eq: key}, we get
\begin{align*}
  \left( \frac{\mu_G(X)}{\mu_G(XY)}\right)^{\frac{1}{n}}+\left( \frac{\mu_G(Y)}{\mu_G(XY)}\right)^{\frac{1}{n}}\leq 1,
\end{align*}
and hence $G$ satisfies $\BM(n)$. 
\end{proof}

Using the proportionated averaging trick in a similar fashion, the next result allows us to reduce the problem to certain open subgroups. 

\begin{proposition} \label{prop: Reduceinequalitytoopensubgroups}
Let $G$ be a locally compact group, and let $G'$ be an open normal unimodular subgroup of $G$. Suppose $G'$ satisfies $\BM(n)$ for some integer $n\geq1$. Then $G$ satisfies $\BM(n)$. 
\end{proposition}
\begin{proof}
Let $\mu_{G'}$ be a left (and hence right) Haar measure on $G'$. By Fact~\ref{fact: surjectivemodularonconnectedcomp}.2, there is a left Haar measure $\mu_G$ and a right Haar measure $\nu_G$ on $G$, such that for every compact set $\Omega$ in $G$ we have
\[
\nu_G(\Omega)=\sum_{\alpha\in G/G'} \Delta_G(g_\alpha^{-1})\mu_{G'}(g_\alpha^{-1}\Omega\cap G'),\quad \mu_G(\Omega)=\sum_{\beta\in G'\backslash G} \Delta_G(g_\beta)\mu_{G'}(\Omega g_\beta^{-1}\cap G'),
\]
with  $g_\alpha$ and $g_\beta$ are representative elements of $\alpha$ and $\beta$ respectively. 

Now we fix two compact sets $X,Y$ in $G$. For every $\alpha\in G/G'$, let $X_\alpha=g_\alpha^{-1}X\cap G'$, and we similarly define $Y_\beta=Yg_\beta^{-1}\cap G'$ for every $\beta\in G'\backslash G$. Since $G'$ satisfies $\BM(n)$, we have that
\begin{equation}\label{eq: codiscrete induction}
 \left(\frac{\mu_{G'}(X_\alpha)}{\mu_{G'}(X_\alpha Y_\beta)}\right)^{1/n}+ \left(\frac{\mu_{G'}(Y_\beta)}{\mu_{G'}(X_\alpha Y_\beta)}\right)^{1/n}\leq 1.
\end{equation}
Now we choose $\alpha$ from $G/G'$ randomly with probability $\pr_X(\alpha)=\frac{\Delta_G(g_\alpha^{-1})\mu_{G'}(X_\alpha)}{\nu_G(X)}$. Therefore by H\"older's inequality,
\begin{align*}
\mathbb{E}_{\pr_X(\alpha)}\left(\frac{\mu_{G'}(X_\alpha)}{\mu_{G'}(X_\alpha Y_\beta)}\right)^{\frac1n}&=\frac{1}{\nu_G(X)}\sum_{\alpha\in G/G'}\frac{\left(\mu_{G'}(X_\alpha)\Delta_G(g_\alpha^{-1})\right)^{\frac{n+1}{n}}}{\left(\mu_{G'}(X_\alpha Y_\beta)\Delta_G(g_\alpha^{-1})\right)^{\frac{1}{n}}}\\
&\geq \left(\frac{\nu_G(X)}{\nu_G(XYg_\beta^{-1})}\right)^{\frac1n} = \left(\frac{\nu_G(X)}{\nu_G(XY)}\right)^{\frac1n}. 
\end{align*}
Similarly, we choose $\beta$ from $G'\backslash G$ randomly with probability $\pr_Y(\beta)=\frac{\Delta_G(g_\beta)\mu_{G'}(Y_\beta)}{\mu_G(Y)}$.
Again using H\"older's inequality, we conclude that
\[
\mathbb{E}_{\pr_Y(\beta)}\left(\frac{\mu_{G'}(Y_\beta)}{\mu_{G'}(X_\alpha Y_\beta)}\right)^{\frac1n}\geq \left(\frac{\mu_G(Y)}{\mu_G(XY)}\right)^{\frac1n}.
\]
Hence by \eqref{eq: codiscrete induction}, 
\begin{align*}
&\,\left(\frac{\nu_G(X)}{\nu_G(XY)}\right)^{\frac1n}+\left(\frac{\mu_G(Y)}{\mu_G(XY)}\right)^{\frac1n}\\
\leq&\, \mathbb{E}_{\pr_X(\alpha)}\mathbb{E}_{\pr_Y(\beta)}\left[\left(\frac{\mu_{G'}(X_\alpha)}{\mu_{G'}(X_\alpha Y_\beta)}\right)^{1/n}+ \left(\frac{\mu_{G'}(Y_\beta)}{\mu_{G'}(X_\alpha Y_\beta)}\right)^{1/n} \right]\leq1,
\end{align*}
and thus $G$ also satisfies $\BM(n)$. 
\end{proof}

The normality assumption in Proposition~\ref{prop: Reduceinequalitytoopensubgroups} is used to control the modular function via Fact~\ref{fact: surjectivemodularonconnectedcomp}.2. Thus when the ambient group $G$ is unimodular, the same proof of Proposition~\ref{prop: Reduceinequalitytoopensubgroups} will allow us to reduce the problem to a not necessarily normal open subgroup. We include the proof here for the sake of completeness. 

\begin{proposition} \label{prop: reducetoopensubgroup2}
Let $G$ be a unimodular group, and let $G'$ be an open subgroup of $G$. Suppose $G'$ satisfies $\BM(n)$ for some integer $n\geq 0$, then $G$ satisfies $\BM(n)$. 
\end{proposition}

\begin{proof}
When $n=0$, the conclusion follows from $\mu(XY) \geq \mu (Y)$. In the remainder of the proof we assume $n \geq 1$.

Let $\mu_G$ be a Haar measure on $G$, and let $\mu_{G'}$ be the restricted Haar measure of $\mu_G$ on $G'$. As $G'$ is open, it is also unimodular. By  Fact \ref{fact: homomorphism facts}.1,  for every compact set $\Omega$ in $G$ we have
\[
\mu_G(\Omega)=\sum_{g\Omega\in G/G'}\mu_{G'}(g\Omega\cap G')=\sum_{g\Omega\in G\backslash G'}\mu_{G'}(\Omega g\cap G'). 
\]

Now we fix two compact sets $X,Y$ in $G$. For every $\alpha\in G/G'$, let $X_\alpha=g_\alpha^{-1}X\cap G'$, and we similarly define $Y_\beta=Yg_\beta^{-1}\cap G'$ for every $\beta\in G'\backslash G$. Since $G'$ satisfies $\BM(n)$, we have that
\begin{equation}\label{eq: codiscrete induction}
 \left(\frac{\mu_{G'}(X_\alpha)}{\mu_{G'}(X_\alpha Y_\beta)}\right)^{1/n}+ \left(\frac{\mu_{G'}(Y_\beta)}{\mu_{G'}(X_\alpha Y_\beta)}\right)^{1/n}\leq 1.
\end{equation}
Now we choose $\alpha$ from $G/G'$ randomly with probability $\pr_X(\alpha)=\frac{\mu_{G'}(X_\alpha)}{\mu_G(X)}$ and $\beta$ from $G'\backslash G$ randomly with probability $\pr_Y(\beta)=\frac{\mu_{G'}(Y_\beta)}{\mu_G(Y)}$.
Then H\"older's inequality gives us,
\begin{align*}
\mathbb{E}_{\pr_X(\alpha)}\left(\frac{\mu_{G'}(X_\alpha)}{\mu_{G'}(X_\alpha Y_\beta)}\right)^{\frac1n}
\geq \left(\frac{\mu_G(X)}{\mu_G(XYg_\beta^{-1})}\right)^{\frac1n} = \left(\frac{\mu_G(X)}{\mu_G(XY)}\right)^{\frac1n},
\end{align*}
and
\[
\mathbb{E}_{\pr_Y(\beta)}\left(\frac{\mu_{G'}(Y_\beta)}{\mu_{G'}(X_\alpha Y_\beta)}\right)^{\frac1n}\geq \left(\frac{\mu_G(X)}{\mu_G(g_\alpha^{-1}XY)}\right)^{\frac1n} = \left(\frac{\mu_G(Y)}{\mu_G(XY)}\right)^{\frac1n}.
\]
Hence by \eqref{eq: codiscrete induction}, 
\begin{align*}
&\,\left(\frac{\nu_G(X)}{\nu_G(XY)}\right)^{\frac1n}+\left(\frac{\mu_G(Y)}{\mu_G(XY)}\right)^{\frac1n}\\
\leq&\, \mathbb{E}_{\pr_X(\alpha)}\mathbb{E}_{\pr_Y(\beta)}\left[\left(\frac{\mu_{G'}(X_\alpha)}{\mu_{G'}(X_\alpha Y_\beta)}\right)^{1/n}+ \left(\frac{\mu_{G'}(Y_\beta)}{\mu_{G'}(X_\alpha Y_\beta)}\right)^{1/n} \right]\leq1,
\end{align*}
and thus $G$ also satisfies $\BM(n)$. 
\end{proof}

\section{Proof of Theorems~\ref{thm: main Lie}, \ref{thm: main}, and~\ref{thm: mainreduction} and Corollary~\ref{cor: 1.6BG}}
\subsection{A dichotomy lemma} In this subsection, we prove a dichotomy result for the kernel of a continuous homomorphism to $(\RR^{>0},\times)$.

The following lemma records a fact on open maps between locally compact groups.

\begin{lemma} \label{lem: dichotomypreparation}
Suppose $G,H$ are  locally compact groups, $\phi: G \to H$ is a continuous and surjective group homomorphism, and there is an open subgroup $G'$ of $G$ such that $\phi|_{G'}$ is open. Then $\phi: G \to H$ is a quotient map of locally compact groups.
\end{lemma}
\begin{proof}
By the first isomorphism theorem (Fact~\ref{fact: homomorphism facts}.1), it suffices to check that $\phi$ is open. Suppose $U$ is an open subset of $G$. Then $U = \bigcup_{a \in G} U \cap aG'$. For each $a \in G$, we have 
$$\phi(U \cap aG') = \phi(a) \phi|_{G'}( a^{-1}U \cap G'   ). $$
As $\phi|_{G'}$ is open, $\phi(U \cap aG')$ is open for each $a \in G$.
Hence, $\phi(U) = \bigcup_{a \in G} \phi(U \cap aG')$ is open in $H$, which is the desired conclusion.
\end{proof}

In the next lemma, we present our main dichotomy result.

\begin{lemma} \label{Lem: dichotomy2}
Suppose $G$ is a locally compact group, and $\pi: G \to (\RR^{>0}, \times)$ is a  continuous group homomorphism. Then exactly one of the following holds:
\begin{enumerate}
       \item we have the short exact sequence of locally compact groups 
    $$1  \to \ker\pi \to G \overset{\pi\ }{\to} (\RR^{>0}, \times) \to 1;$$ 
     \item $\ker \pi$ is an open subgroup of $G$.
     
\end{enumerate}
\end{lemma}

\begin{proof}
It is easy to see that (1) and (2) are mutually exclusive, so we need to prove that we are always either in (1) or (2). Consider first the case when $G$ is a Lie group. Let $G_0$ be the identity component of $G$. Then $G_0$ is open by  Fact~\ref{factLieid}. Hence $\pi(G_0)$ is a connected subgroup of $(\RR^{>0},\times)$. As the only connected subsets of $(\RR^{>0},\times)$ are points and intervals, we deduce that $\pi(G_0)$ can only be $\{ 1\}$ or $(\RR^{>0},\times)$. In the former case,  $\ker \pi$ is open as a union of translates of $G_0$. Now suppose $\pi(G_0) = (\RR^{>0} ,\times)$. Since $G_0$ is a connected Lie group. Using   the first isomorphism theorem for Lie groups (Fact~\ref{fact: iso theorems Lie}.1), we get $\pi|_{G_0}$ is open. Applying Lemma~\ref{lem: dichotomypreparation}, we get that $\pi$ is a quotient map as desired.

We now deal with the general situation  where $G$ is locally compact. Using the Gleason--Yamabe Theorem (Fact~\ref{fact: Gleason}.1),  we obtain an almost-Lie open subgroup $G'$ of $G$. Since $G'$ is open, the natural embedding of $i: G'\to G$ induces a continuous homomorphism $\pi |_{G'}: G'\to(\RR^{>0},\times)$. Note that there is a compact normal subgroup $H$ of $G'$ such that $G'/H$ is a Lie group. Then $H\leq \ker(\pi |_{G'})$ since $\pi |_{G'}(H)$ is a compact subgroup of $(\RR^{>0},\times)$. Let $\phi: G' \to G'/H$ be the quotient map. Hence the homomorphisms induce a continuous group homomorphism $\psi$ from $G'/H$ to $(\RR^{>0},\times)$.
\begin{center}
\begin{tikzcd}
G' \arrow[r, "\phi"] \arrow[rd, "\pi |_{G'}"] \arrow[d, "i"]
& G'/H \arrow[d, dashed, "\psi"] \\
G \arrow[r, "\pi"] 
& (\RR^{>0},\times)
\end{tikzcd}
\end{center}
Note that the above diagram commutes. 
By the proven special case for Lie groups, we then either have the exact sequence
$$  1 \to \ker \psi \to  G'/H \to (\RR^{>0}, \times) \to 1 $$
or  $\ker \psi$ is open in $G'/H$. In the former case, $\pi |_{G'}$ is open  as a composition of open maps.  By Lemma~\ref{lem: dichotomypreparation}, we conclude that $\pi$ is a quotient map in this case.
In the latter case, $\ker (\pi |_{G'})$ is open in $G'$. Therefore $\ker \pi$  is open in $G$ because $\ker \pi$ is a union of translations of $\ker (\pi |_{G'})$.
\end{proof}

The modular function $\Delta_G: G \to (\RR^{>0}, \times)$ is a continuous group homomorphism by Fact~\ref{fact: modular function}.2, but generally not a quotient map. It is easy to construct examples where $G/(\ker \Delta_G)$ is discrete. The above proposition claims that these are the only two possibilities, which will be used in the later proofs.

\subsection{Proofs of the main theorems}
In this subsection, we prove Theorems~\ref{thm: main Lie} and \ref{thm: main}.  For the reader's convenience, Proposition~\ref{prop: availableinduction} gathers together all the induction steps we can do using the earlier results with the exception of Proposition~\ref{Prop: Cocompact reduction}, which will be used in the proof of Theorem~\ref{thm: main Lie} directly.

\begin{proposition} \label{prop: availableinduction}
Let $G$ be a locally compact group with noncompact Lie dimension $n$ and helix dimension $h$.  Then $G$ satisfies $\BM(n-h)$ if one of the following assumptions holds:
\begin{enumerate}
    \item Let $\Delta_G: G \to (\RR^{>0}, \times)$ be the modular function of $G$. With $n'$ and $h'$ the noncompact Lie dimension and the helix dimension of the locally compact group $\ker \Delta_G$ respectively, the locally compact group $\ker \Delta_G$ satisfies $\BM(n'-h')$.
    \item $G$ is unimodular, $G'$ is an open subgroup of $G$, $n'$ and $h'$ are the noncompact Lie dimension and the helix dimension of  $G'$ respectively, and the locally compact group $G'$ satisfies $\BM(n'-h')$.
    \item $G$ is  unimodular, $H$ is a compact normal subgroup of $G$, $n'$ and $h'$ are the noncompact Lie dimension and the helix dimension of  $G/H$ respectively, and the locally compact group $G/H$ satisfies $\BM(n'-h')$.
    \item There is an exact sequence of connected semisimple Lie groups 
    $$
    1 \to H \to G \to G/H \to 1,
    $$
   $n_1$ and $h_1$ are the noncompact Lie dimension and the helix dimension of  $H$ respectively, $n_2$ and $h_2$ are the noncompact Lie dimension and the helix dimension of  $G/H$ respectively, and the locally compact groups $H$ and $G/H$ satisfy $\BM(n_1-h_1)$ and $\BM(n_2-h_2)$ respectively. 
    \item There is an exact sequence of connected unimodular Lie groups 
    $$1 \to H \to G \to G/H \to 1$$
     $n_1$ and $h_1$ are the noncompact Lie dimension and the helix dimension of  $H$ respectively, $n_2$ and $h_2$ are the noncompact Lie dimension and the helix dimension of  $G/H$ respectively, $h_1=0$, $h_2=h$, and the locally compact groups $H$ and $G/H$ satisfy $\BM(n_1-h_1)$ and $\BM(n_2-h_2)$ respectively. 
\end{enumerate}
\end{proposition}

\begin{proof}

We first prove (1).  Note that by Fact \ref{fact: modular function}.1, $\ker \Delta_G$ is unimodular.
By Lemma~\ref{Lem: dichotomy2}, we either have the exact sequence of locally compact groups 
\[
1 \to \ker \Delta_G \to G \to (\RR^{>0}, \times) \to 1
\]
or $\ker \Delta_G$ is open in $G$. In the former case, by Proposition~\ref{prop: additivityofdimension2}, we have $n=n'+1$ and $h=h'$. Hence, in this case $G$ satisfies $\BM(n-h)$ by Proposition~\ref{prop: reduce to unimodular}. In the latter case, by Corollary~\ref{cor: prop 2 (1)}, $n=n'$ and $h=h'$. Here, we have that $G$ satisfies $\BM(n-h)$ by Proposition~\ref{prop: Reduceinequalitytoopensubgroups}. 

Next we prove (2). By Corollary~\ref{cor: prop 2 (1)}, we have $n=n'$ and $h=h'$. The desired conclusion then follows from  Proposition~\ref{prop: reducetoopensubgroup2}.

We now prove (3). By Corollary~\ref{cor: prop 2 (2)}, we have $n=n'$ and $h=h'$. Also by Corollary~\ref{cor: prop 2 (2)}, the compact group $H$ has noncompact Lie dimension and helix dimension $0$. Hence, using Proposition~\ref{lem: quotient unimodular}, we obtain the conclusion that we want.

 We prove (4). By Proposition~\ref{prop: additivitydim1}.1 and Proposition~\ref{prop: additivitydim1}.2 respectively, we have $n=n_1+n_2$ and $h=h_1+h_2$. Recall that semisimple groups are unimodular. Using Proposition~\ref{lem: quotient unimodular}, we learn that $G$ satisfies $\BM(n-h)$. 
 
 Finally, we prove (5). By Proposition~\ref{prop: additivitydim1}.1, we have $n=n_1+n_2$. Since the helix dimension of $H$ is $0$, and the helix dimension of $G/H$ is $h$, by Proposition~\ref{lem: quotient unimodular}, $G$ satisfies $\BM(n-h)$. 
\end{proof}

The following corollary says that when $G$ is a Lie group, we can further reduce the problem to connected unimodular groups.

\begin{corollary}\label{cor: lie -> unimodular & connected}
Let $G$ be a Lie group with noncompact Lie dimension $n$ and helix dimension $h$. Let $\Delta_G: G \to (\RR^{>0}, \times)$ be the modular function of $G$.  Let $G'=(\ker \Delta_G)_0$ be the identity component of $\ker \Delta_G$ with noncompact Lie dimension $n'$ and helix dimension $h'$. Then $G'$ is connected and unimodular, and if $G'$ satisfies $\BM(n'-h')$, $G$ satisfies $\BM(n-h)$.
\end{corollary}
\begin{proof}
Note that $(\ker \Delta_G)_0$ is open in $\ker \Delta_G$ by  Fact~\ref{factLieid}. The desired conclusion is then a consequence of Proposition~\ref{prop: availableinduction}.1 and Proposition~\ref{prop: availableinduction}.2.
\end{proof}

Now we are able to prove the main inequality \eqref{eq: BM for nonunimodular} for Lie groups. As mentioned earlier, the main strategy is induction on dimension.

\begin{proof}[Proof of Theorem~\ref{thm: main Lie}]

Consider first the case where $G$ is a solvable Lie group. Using Corollary~\ref{cor: lie -> unimodular & connected}, we can also assume that $G$ is connected and unimodular. Recall that $d$ is the topological dimension of $G$. The case when $d=0$ or $1$ is trivial, as every group satisfies $\BM(0)$, and the one dimensional solvable Lie group is either $\TT$ or $\RR$ by Fact~\ref{fact: classification}.1.  If $G$ is abelian, then it is isomorphic to $\TT^m \times \RR^{d-m} $. We get the desired conclusion by applying Proposition~\ref{prop: availableinduction}.5 repeatedly. Otherwise, from the solvability of $G$ we get the exact sequence 
\[
1 \to \overline{[G,G]} \to G \to G/\overline{[G,G]} \to 1
\]
with both $\overline{[G,G]}$ and $G/\overline{[G,G]}$ connected, solvable and having smaller dimensions than $G$. Note that $G/\overline{[G,G]}$ is abelian, and hence unimodular. Applying Proposition~\ref{prop: availableinduction}.5, and the statement of the theorem for abelian Lie groups, we get the desired conclusion for this case.

Consider next the case where $G$ is connected and semisimple. We may further assume that $G$ is a connected simple Lie group, otherwise by Fact~\ref{fact: simple}, we can always find a connected group $H\vartriangleleft G$ such that both $H$ and $G/H$ are connected semisimple Lie groups with lower dimension; by  Proposition~\ref{prop: availableinduction}.4, the Brunn--Minkowski inequality on $G$ can be obtained from the Brunn--Minkowski inequalities on $H$ and $G/H$. Now we write $G = KAN$ as in Fact~\ref{fact: Lie group decomp Iwasawa}. We first consider the case when $G$ has a finite center, and then $K$ is compact. 
Let $n$ be the noncompact Lie dimension of $G$. Hence, $n$ is the dimension of the solvable Lie group $Q=AN$. Note that $A$ and $N$ are simply connected by Fact~\ref{fact: Lie group decomp Iwasawa}. Hence their noncompact Lie dimensions are the same as their dimensions by Fact~\ref{fact: classification}.2. By Proposition~\ref{prop: additivitydim1}.1 and Fact~\ref{fact: Lie group decomp Iwasawa}, the noncompact Lie dimension of $Q$ is $n$, and hence $Q$ satisfies $\BM(n)$ from the solvable Lie case. We obtain the desired conclusion for $G$ by applying  Proposition~\ref{Prop: Cocompact reduction}.

Suppose the connected simple Lie group $G$ has a center of rank $h\geq1$. We again apply Proposition~\ref{Prop: Cocompact reduction} and obtain inequality \eqref{eq: BM for nonunimodular} for $G$ with exponent $\dim(AN)$. By Proposition~\ref{prop: helixandnoncompact}, we have $\dim(AN)=n-h$.  The desired conclusion for the connected semisimple Lie groups follows similarly from Fact~\ref{fact: simple} and Proposition~\ref{prop: availableinduction}.4.

Finally, we prove the statement for an arbitrary Lie group $G$. Using Corollary~\ref{cor: lie -> unimodular & connected} again, we can assume that $G$ is connected and unimodular.
Then by Fact~\ref{fact: Lie group decomp Levi} we obtain an exact sequence 
\[
1 \to Q \to G \to S \to 1,
\]
where $Q$ is a connected unimodular solvable group and $S$ is a connected semisimple Lie group. We then apply Proposition~\ref{prop: availableinduction}.5 and the earlier two cases to get the desired conclusion.
\end{proof}

Finally, we prove the inequality \eqref{eq: BM for nonunimodular} for all locally compact groups. 

\begin{proof}[Proof of Theorem~\ref{thm: main}] 
By Proposition~\ref{prop: availableinduction}.1, replacing $G$ by $\ker\Delta_G$ if necessary, we can assume that $G$ is unimodular. By the Gleason--Yamabe Theorem (Fact~\ref{fact: Gleason}.1), $G$ has an almost-Lie open subgroup. Now using Proposition~\ref{prop: availableinduction}.2, we can  further assume that $G$ is a unimodular almost-Lie group. Then we can choose a compact subgroup $K$ of $G$ such that $G/K$ is a unimodular Lie group. The desired conclusion then follows from Theorem~\ref{thm: main Lie} and Proposition~\ref{prop: availableinduction}.3.
\end{proof}

We briefly discuss Theorem~\ref{thm: mainreduction}, which is a consequence of the proof of Theorem~\ref{thm: main}.

\begin{proof}[Proof of Theorem~\ref{thm: mainreduction}] 
Repeating the arguments in the proofs of Proposition~\ref{prop: availableinduction}, Corollary~\ref{cor: lie -> unimodular & connected}, Theorem~\ref{thm: main Lie}, Theorem~\ref{thm: main}, and Fact~\ref{fact: simple} while ignoring the helix dimension, it suffices to prove the theorem when $G$ is a simple Lie group. 

From the hypothesis, we already have the desired conclusion under the further assumption that our simple Lie group $G$ is also simply connected. We now consider the general case. If $G$ has finite center, the result is a special case of Theorem~\ref{thm: main Lie}. So suppose the center $Z(G)$ of  $G$ is infinite. Let $\widetilde{G}$ be the universal cover of $G$, $Z(\widetilde{G})$ its center, and $\rho: \widetilde{G} \to G$ the covering map. Then $\ker \rho$ is a subgroup of  $Z(\widetilde{G})$ by Fact~\ref{fact: centerless2}. 
Using Fact~\ref{fact: centersimpleLie}, the center $Z(\widetilde{G})$ has rank at most $1$. By the earlier assumption, the center $Z(G)$ also has rank at least $1$.  Hence, by Fact~\ref{fact: centerless2}, both $Z(\widetilde{G})$ and $Z(G)$ must have rank $1$, and $\ker \rho$ is finite. Therefore, the desired conclusion for $G$ can be reduced to that of $\widetilde{G}$ by taking the inverse image under $\rho$, which we already know from the hypothesis. 
\end{proof}

\begin{proof}[Proof of Corollary~\ref{cor: 1.6BG}]
Let $G$ be a noncompact semisimple Lie group. 
By Proposition~\ref{prop: additivitydim1}.1, Fact~\ref{fact: Lie group decomp Iwasawa}.3 and Fact~\ref{fact: Lie group decomp Iwasawa}.4, and Fact~\ref{fact: classification2}, the noncompact Lie dimension of $G$ is at least $2$. The conclusion then follows from Theorem~\ref{thm: main Lie}.
\end{proof}

\appendix

\section{Some results about topological groups}

This section gathers some facts about topological groups which are needed in the proof. We begin with the three isomorphism theorems of topological groups. Note that the third isomorphism theorem is almost the same as the familiar result for groups, whereas the first two isomorphism theorems require extra assumptions; see ~\cite[Proposition III.2.24]{BourbakiTopology},~\cite[Proposition III.4.1]{BourbakiTopology}, and~\cite[Proposition III.2.22]{BourbakiTopology} for details. For Fact~\ref{fact: homomorphism facts}, we do not need to assume that $G$ is locally compact. The quotient $G/H$ is equipped with the quotient topology (i.e., $X \subseteq G/H$ is open if and only if it inverse image under the quotient map is open).

\begin{fact} \label{fact: homomorphism facts}
Suppose $H$ is a closed normal subgroup of $G$. Then we have the following:
\begin{enumerate}
    \item \emph{(First isomorphism theorem)} Suppose $\phi: G \to Q$ is a continuous surjective group homomorphism with $\ker \phi = H$.  Then the exact sequence of groups
        $$  1 \to H \to G \to Q \to 1 $$ 
    is an exact sequence of topological groups if and only if  $\phi$ is open; the former condition is equivalent to saying that  $Q$ is canonically isomorphic to $G/H$ as topological groups. 
    \item \emph{(Second isomorphism theorem)} Suppose  $S$ is a closed subgroup of $G$ and $H$ is compact. Then $S/(S \cap H)$ is canonically isomorphic to the image of $SH/H$ as topological groups. This is equivalent to saying that we have the exact sequence of topological groups
    $$   1 \to H \to SH \to S/(S \cap H)\to 1. $$
    \item \emph{(Third isomorphism theorem)}  Suppose $S \leq  G$ is closed, and $H \leq S$. Then $S/H$ is a closed subgroup of $G/H$. If $S\vartriangleleft G$ is normal, then $S/H$ is a normal subgroup of $G/H$, and we have the exact sequence of topological groups
    $$  1 \to S/H \to G/H \to G/S \to 1; $$
    this is the same as saying that $(G/H)/(S/H)$ is canonically isomorphic to $G/S$ as topological groups. 
\end{enumerate}
\end{fact}

Recall that a topological space $X$ is {\bf completely regular} if points can be separated from closed sets via continuous real-valued functions. More precisely,  for any closed set $C \subseteq X$ and any point $a \in X\setminus A$,  there exists a real-valued continuous function $f:X\to \RR$ such that $f(a)=1$ and $f(c) = 0$ for all $c \in C$.

\begin{fact}
The underlying topological space of a topological group is completely regular.
\end{fact}

We also need the following simple property of locally compact groups~\cite[Theorem~6.7]{Folland}.
\begin{fact}
Closed subgroups and quotient groups of locally compact groups are locally compact. 
\end{fact}

The following lemma holds for all topological groups.

\begin{lemma} \label{lemma:productofclosedandcompact}
Suppose $X, Y \subseteq G$, $X$ is compact and $Y$ is closed. Then $XY$ is closed. 
\end{lemma}

\begin{proof}
Let $a$ be in $G \setminus XY$. Then $X^{-1}a$ is compact and $X^{-1}a\cap Y=\emptyset$. For each point $x \in  X^{-1}a$, we choose an open neighborhood of identity $U_x$ such that $xU^2_x \cap Y =\emptyset$. Then $(xU_x)_{x \in X^{-1}a}$ is an open cover of $X^{-1}a$. Using the fact that $X^{-1}a$ is compact, we get a subcover $(U_i)_{i=1}^k$. Set $U= \bigcap_{i=1}^k U_i$. It is easy to check that $X^{-1}aU \cap Y = \emptyset$. Then $aU\cap XY =\emptyset$, which implies that $XY$ is closed as $a$ can be chosen arbitrarily.  
\end{proof}

The next lemma records a simple fact of compact  subgroups.

\begin{lemma}\label{lem: inverse image of compact sets}
If $H$ is a compact subgroup of $G$, then the quotient map $\pi: G \to G/H$ is a proper map (i.e., the inverse images of compact subsets are compact). 
\end{lemma}
\begin{proof}
Let $\Omega$ be a compact subset of $G/H$. In particular $\Omega$ is closed. Hence, $\pi^{-1}(\Omega)$ is closed, so it suffices to find a compact set containing $\pi^{-1}(\Omega)$. Since $G$ is locally compact, we can find an open covering $(U_i)_{i \in I}$ of $\pi^{-1}(\Omega)$ such that $U_i$ has compact closure $\overline{U_i}$ for each $i \in I$. Then $(\pi U_i)_{i \in I}$ is an open cover of $\Omega$ as $\pi$ is open. Using the assumption that $\Omega$ is compact, we get a finite $I' \subseteq I$ such that $(\pi(U_i))_{i \in I'}$ is an open cover of $\Omega$. Then $\bigcup_{i \in I'} \overline{U_i}H$ is a compact set containing $\pi^{-1}(\Omega)$. 
\end{proof}

\section{Measures and the modular function}
 
We say that a measure $\mu$ on the collection of Borel subsets of $G$ is a {\em left Haar measure}~\cite[Section 2.2]{Folland} if it satisfies the following properties: 
\begin{enumerate}
     \item (nonzero) $\mu(X)>0$ for all open $X\subseteq G$;
     \item (left-translation-invariant) $\mu(X) =\mu(aX)$ for all $a\in G$ and all measurable sets $X \subseteq G$;
    \item (inner regular for open sets) when $X$ is open, $\mu(X) =\sup \mu(K)$ with $K$ ranging over compact subsets of $X$;
    \item (outer regular for Borel sets) when $X$ is Borel, $\mu(X) =\inf \mu(U)$ and $U$ ranging over open subsets of $G$ containing $X$;
    \item (compactly finite) $\mu$ takes finite measure on compact subsets of $G$.
\end{enumerate}
The notion of a {\em right Haar measure} is obtained by making the obvious modifications to the above definition. The following classical result by Haar
guarantees the existence of such a measure on every locally compact group.
\begin{fact}\cite[Theorem 2.20]{Folland}\label{fact: haar measure}
Up to multiplication by a positive constant, there is a unique left Haar measure of $G$. A similar statement holds for right Haar measure.
\end{fact}

Given a locally compact group $G$, and $\mu$ is a left Haar measure on $G$. For every $x\in G$, recall that $\Delta_G: x\mapsto \mu_x/\mu$ is the \emph{modular function} of $G$, where $\mu_x$ is a left Haar measure on $G$ defined by $\mu_x(A)=\mu(Ax)$, for every measurable set $A$. When the image of $\Delta_G$ is $\{1\}$, we say $G$ is \emph{unimodular}. In general, $\Delta_G(x)$ takes values in $\RR^{>0}$. We use $(\RR^{>0}, \times)$ to denote the multiplicative group of positive real numbers together with the usual Euclidean topology. The next fact records some basic properties of the modular function; see~\cite[Section 2.4]{Folland}.
\begin{fact}\label{fact: modular function}
Let $G$ be a locally compact group. Assume $\mu$ is a left Haar measure and $\nu$ is a right Haar measure. 
\begin{enumerate}
    \item Suppose $H$ is a normal closed subgroup of $G$, then $\Delta_H=\Delta_G|_H$. In particular, if $H=\ker\Delta_G$, then $H$ is unimodular.
    \item The function $\Delta_G: G\to (\RR^{>0}, \times)$ is a continuous homomorphism. 
    \item For every $x\in G$ and every measurable set $A$, we have $\mu(Ax)=\Delta_G(x)\mu(A)$, and $\nu(xA)=\Delta_G^{-1}(x)\nu(A)$.
    \item There is a constant $c$ such that $\int_G f\d \mu = c \int_G f\Delta_G\d \nu$ for every $f\in C_c(G)$. 
\end{enumerate}
\end{fact}

We use the following integral formula ~\cite[Theorem 2.49]{Folland} in our proofs.
\begin{fact}[Quotient integral formula]\label{fact: Quotient Integral Formula}
Let $G$ be a locally compact group, and let $H$ be a closed normal subgroup of $G$. Let $\mu_G$ and $\mu_H$ be left Haar measures on $G$ and on $H$, respectively. Then  there is a unique left Haar measure $\mu_{G/H}$ on $G/H$, such that
for every $f\in C_c(G)$, 
\[
\int_G f(x)\d \mu_G(x)=\int_{G/H}\int_H f(xh) \d \mu_H(h)\d \mu_{G/H}(xH).
\]
\end{fact}

The following fact is a consequence of a result of Haar measure on closed subgroups and quotients~\cite[Proposition VII. 2.7.10]{Bourbaki}.

\begin{fact}\label{fact: surjectivemodularonconnectedcomp}
Suppose $G$ is nonunimodular, and $\Delta_G: G \to (\RR^{>0}, \times)$ is the modular function of $G$, then we have the following:
\begin{enumerate}
    \item If $K\vartriangleleft G$ is a compact normal subgroup of $G$, $\Delta_{G/K}$ is the modular function of $G/K$, and $\pi: G \to G/K$ is the quotient map, then we have $\Delta_{G} = \Delta_{G/K} \circ \pi $.
    \item If $H\vartriangleleft G$ is a closed unimodular group, and $\mu_H$ is a Haar measure on $H$.  Suppose $G/H$ is unimodular, and $X$ is a compact subset of $H$. Then for every $g\in G$, $\mu_H(gXg^{-1})=\Delta_G(g)\mu_H(X)$. 
\end{enumerate}
\end{fact}

\section{Almost-Lie groups and the Gleason--Yamabe Theorem}

In our proof we need the solution of Hilbert's 5th problem, which is known as the Gleason--Yamabe Theorem~\cite{Gleason,Yamabe}, to reduce the problem into Lie groups. For convenience, we introduce the following   terminology. A locally compact group  $G$ is an {\bf almost-Lie group} if every open neighborhood $U$ of the identity in $G$ contains a compact   $H\vartriangleleft G$ such that $G/H$ is a Lie group.

\begin{lemma} \label{lem: almostliesubgroupandquotient}
Suppose $G$ is an almost-Lie group. Then every open subgroup of $G$ and every quotient of $G$ by a closed normal subgroup is an almost-Lie group.
\end{lemma}
\begin{proof}
We first show that every open subgroup of $G$ is almost-Lie. Let $S$ be an open subgroup of $G$, and $U$ be an open neighborhood of identity in $S$. We need to find a compact subgroup $K$ of $S$ such that  $K \subseteq U$ and $S/K$ is a Lie group. Since $U$ is also a neighborhood of identity in $G$,  $U$ contains a compact normal subgroup $K$ of $G$ such that $G/K$ is a Lie group.  Note that $K\vartriangleleft S$. As $S$ is open, $S/K$ is open in $G/K$ and hence a Lie group as desired.

Next, suppose $H$ is a closed normal subgroup of $G$, and $\pi: G \to G/H$ is the quotient map. If $U$ is an open neighborhood of the identity in $G/H$, then $\pi^{-1}(U)$ is an open neighborhood of identity in $G$. Hence, we can get a normal compact subgroup $K$ of $G$ such that $K \subseteq \pi^{-1}(U)$ and that $G/K$ is a Lie group. Then $\pi(K)$ is a compact subgroup of $U$. With $S =\pi^{-1}( \pi(K))$, we have $\pi(K) = S/H$. Since $K$ is normal in $G$ we have $\pi(K)$ is normal in $G/H$ and thus $S$ is normal in $G$. Whence by the third isomorphism theorem (Fact~\ref{fact: homomorphism facts}.3), we conclude that $(G/H)/\pi(K) \cong G/S$. By the third isomorphism theorem again, we have $G/S\cong (G/K)/(S/K)$, thus $G/S$ is a Lie group.
\end{proof}

We use the following strong version of the Gleason--Yamabe Theorem. 

\begin{fact}\label{fact: Gleason}
We have the following:
\begin{enumerate}
    \item \emph{(Gleason--Yamabe Theorem)} Suppose  $G$ is a locally compact group. Then there is an open subgroup of $G$ which is an almost-Lie group.
    \item An almost-Lie group $G$ is a Lie group if and only if there is an open neighborhood $U$ of the identity in $G$ that contains no nontrivial compact subgroup of $G$. 
\end{enumerate}
\end{fact}

 Fact~\ref{fact: Gleason}.2 is not officially part of the Gleason--Yamabe Theorem. However, the forward direction is an easy fact about the no small subgroup property of Lie groups, and the backward direction is a direct consequence of Fact~\ref{fact: Gleason}.1.

\section{Some results about Lie groups}

In this section, we gather some facts and lemmas about Lie groups and Lie algebras. Throughout the paper, all  Lie groups are finite dimensional  second countable  \emph{real} Lie groups.

\begin{fact}
Closed subgroups and quotient groups of Lie groups are Lie groups.
\end{fact}

The {\em identity component} of a topological group $G$ is the connected component containing the identity element. The identity component of a topological group $G$ might not be open even if $G$ is locally compact. For instance, there are nondiscrete totally disconnected locally compact groups. For these groups, the identity component only consists of the identity element, and it is not open because the topology is not discrete. Nevertheless, the following holds for Lie groups~\cite[Proposition 9.1.15]{hilgert2011structure}.
 
\begin{fact}\label{factLieid}
If $G$ is a Lie group, then the identity component of $G$ is open and is contained in every open subgroup of $G$. 
\end{fact}

In Fact~\ref{fact: homomorphism facts}, we introduce the three isomorphism theorems of topological groups. When $G$ is a Lie group, we can weaken the assumption required for the first two isomorphism theorems; see \cite[Proposition 3.11.2, Proposition 3.31]{BourbakiLie}.

\begin{fact}\label{fact: iso theorems Lie}
Suppose $G$ is a Lie group, and $H$ is a closed normal subgroup of $G$. Then we have the following:
\begin{enumerate}
    \item \emph{(First isomorphism theorem for Lie groups)} Suppose $Q$ is a  Lie group, $\phi: G \to Q$ is a surjective and continuous group homomorphism, and $G$ has countably many connected components. Then $Q$ is isomorphic as a topological group to $G/H$.  
    \item \emph{(Second isomorphism theorem for Lie groups)} Suppose $G$ is a finite dimensional Lie group, $S$ is a closed subgroup of $G$, and $SH$ is a closed subgroup of $G$. Then $S/(S \cap H)$ is canonically isomorphic to the image of $SH/H$ as Lie groups. This is also equivalent to saying that we have the exact sequence of Lie groups
    $$   1 \to H \to SH \to S/(S \cap H)\to 1. $$
\end{enumerate} 
\end{fact}

 We also need the following fact about maximal compact subgroups consisting of Theorem 14.1.3 (iii) and Theorem 14.3.13 (i) (a) of \cite{hilgert2011structure}:

\begin{fact} \label{fact: maximal compact} Suppose $G$ is a Lie group with finitely many connected components. Then we have the following:
\begin{enumerate}
    \item All maximal compact subgroups of $G$ are conjugate.
    \item If $0 \to H \to G \overset{\pi\ }{\to} G/H \to 0$ is an exact sequence of connected Lie groups, and $K$ is a maximal compact subgroup of $G$, then $K \cap H$ is a maximal compact subgroup of $H$, and $\pi(K)$ is a maximal compact subgroup of $G/H$.
\end{enumerate}
\end{fact}

We also use the following classification results for Lie groups. Fact~\ref{fact: classification} can be found on  \cite[p. 212]{Encylopedia}.  Fact~\ref{fact: classification}(2) is from Corollary~3 of \cite[Theorem~2.3.1]{Encylopedia} and \cite[Proposition~4.1.2]{Encylopedia}.

 \begin{fact}  \label{fact: classification} 
 Let $G$ be a connected Lie group. 
 \begin{enumerate}
     \item If $G$ has dimension $1$, then it is  isomorphic to either $\RR$ or $\TT$ as a topological group.
     \item Suppose $G$ is a solvable group with dimension $d$, and  the maximal compact subgroups of $G$ have dimension $m$. Then $G$ is diffeomorphic to $\TT^m \times \RR^{d-m}$. Moreover, if $G$ is compact, then $G\cong \TT^d$. 
 \end{enumerate}
\end{fact}

We say that a topological group $G$ (not necessarily simply connected) is a covering group of a topological group $G'$ with covering homomorphism $\rho$ if $\rho: G\to G'$ is a topological group homomorphism which is also a covering map. 
The following is a consequence of \cite[Theorem 9.5.4]{hilgert2011structure}:

\begin{fact} \label{fact: covering group} 
Suppose that $G$ and $G'$ are connected Lie groups and that $G$ is a covering group of $G'$ with covering homomorphism $\rho$. Then $\ker \rho$ is a closed normal subgroup of the center $Z(G)$ of $G$.
\end{fact}

We end this section with a lemma about conjugate actions on compact sets in Lie groups.

\begin{lemma}\label{openinclose}
For a Lie group $G$ and a closed normal subgroup $H$, if a precompact $A \subseteq H$ such that the closure of $A$ is contained in $B$ and $B$ is a relative open subset in $H$, then the following holds: when $g \in G$ is sufficiently close to $\id_G$, we have $gAg^{-1} \subseteq B$.
\end{lemma}
\begin{proof}
We prove the lemma by contradiction. Assuming there exist sequences $g_n \rightarrow \id$ and $\{h_n\} \subseteq A$ such that $g_n h_n g_n^{-1} \notin B$. Since $A$ is precompact we may assume $h_n \rightarrow h \in \overline{A}$. But then $g_n h_n g_n^{-1} \rightarrow h \in \overline{A}$. This contradicts the fact that each $g_n h_n g_n^{-1}$ is in the closed set $H \setminus B$ that does not meet $\overline{A}$. Hence the assumption is false and the conclusion holds.
\end{proof}

\section{Solvable and Semisimple Lie groups}

From ~\cite[Section 9.1]{hilgert2011structure}, there is a functor $\bm{\mathrm{L}}$ from the category of Lie groups to the category of Lie algebras that assigns to each Lie group $G$ its Lie algebra $\bm{\mathrm{L}}(G)$ and to a Lie group morphism $\phi: G \to H $ its tangent morphism $\bm{\mathrm{L}}(\phi): \bm{\mathrm{L}}(G) \to \bm{\mathrm{L}}(H)$ of Lie algebras. We will adopt a more colloquial language in this paper, invoking this functor implicitly.

\begin{fact} \label{fact: CorrespondenceLiegroupandalgebra}
Suppose $G$ and $H$ are Lie groups, and $\mathfrak{g}$ and $\mathfrak{h}$ are their Lie algebras. If $H$ is a subgroup of $G$, then $\mathfrak{h}$ is a subalgebra of $\mathfrak{g}$. If $H$ is a normal subgroup of $G$, then $\mathfrak{h}$ is an ideal in $\mathfrak{g}$, and $\mathfrak{g}/\mathfrak{h}$ is canonically isomorphic to the Lie algebra of  $G/H$.
\end{fact}

Suppose $\mathfrak{g}$ is the Lie algebra of $G$. The exponential function $\exp: \mathfrak{g}\to G$ is defined as in~\cite[Section 9.2]{hilgert2011structure}. We will use the functoriality of the exponential function~\cite[Proposition 9.2.10]{hilgert2011structure} 

\begin{fact} \label{fact:functoriality of exponential function}
Suppose $G$ and $H$ are Lie groups, $\phi: G \to H$ is a homomorphism of Lie groups,  $\mathfrak{g}$ and $\mathfrak{h}$ are the Lie algebras of $G$ and $H$, $\alpha:  \mathfrak{g} \to \mathfrak{h}$ is the tangent morphism of $\phi$, and $\exp_G: \mathfrak{g} \to G$ and $\exp_H: \mathfrak{h} \to H$ are the exponential maps. Then $\exp_H \circ \alpha = \exp_G \circ \phi$. In other words, the following diagram commutes:
\begin{center}
\begin{tikzcd}
G \arrow[r, "\phi"]  
& H  \\
\mathfrak{g} \arrow[r, "\alpha"] \arrow[u, "\exp_G"]
& \mathfrak{h} \arrow[u,"\exp_H"]
\end{tikzcd}
\end{center}
\end{fact}

Suppose $\mathfrak{g}$ is a Lie algebra. The {\it derived Lie algebra} $[\mathfrak{g}, \mathfrak{g}]$ of $\mathfrak{g}$ is the subalgebra of $\mathfrak{g}$ generated by the Lie brackets of the pairs of elements of $\mathfrak{g}$. We say that $\mathfrak{g}$ is {\it solvable} if the derived sequence
$$  \mathfrak{g} \geq [\mathfrak{g}, \mathfrak{g}] \geq [[\mathfrak{g}, \mathfrak{g}],  [\mathfrak{g}, \mathfrak{g}]] \geq \ldots $$
eventually arrives at the trivial Lie algebra $\{0\}$. A Lie group is {\it solvable} if its Lie algebra is solvable. The following is a consequence of~\cite[Proposition 5.4.3]{hilgert2011structure}: 

\begin{fact} \label{fact: preservationsolvable}
Every subalgebra and quotient algebra of a solvable Lie algebra is solvable. Hence, every closed subgroup and quotient group of a solvable Lie group is solvable.
\end{fact}

The following is another consequence of~\cite[Proposition 5.4.3, Theorem 5.6.6]{hilgert2011structure}:

\begin{fact} \label{fact: radical}
Suppose $\mathfrak{g}$ is a Lie algebra. Then $\mathfrak{g}$ has a largest solvable ideal $\mathfrak{q}$.
If $G$ is a Lie group with Lie algebra $\mathfrak{g}$ and $\exp: \mathfrak{g} \to G$ is the exponential map, then $Q =\langle\exp(\mathfrak{q})\rangle$ is the largest closed connected solvable normal subgroup of $G$. Hence, $Q$ is a characteristic subgroup of $G$.
\end{fact}

The subalgebra $\mathfrak{q}$ as in Fact~\ref{fact: radical} is called the {\it radical} of $\mathfrak{g}$, and the subgroup $Q$ as in Fact~\ref{fact: radical} is called the {\it radical} of $G$. A Lie algebra is {\it semisimple} if it has a trivial radical. A lie group is {\it semisimple} if its Lie algebra is semisimple, or equivalently, if it has trivial radical. The following results follows from ~\cite[Theorem 5.6.6, Corollary 5.6.14]{hilgert2011structure}: 

\begin{fact}\label{fact: Lie group decomp Levi} Let $G$ be a  connected Lie group and let $Q$ be its radical. Then $S= G/Q$ is a semisimple Lie group. 
\end{fact} 

A Lie algebra is {\it simple} if it is semisimple and contains no ideals other than itself and the trivial ideal $\{0\}$. A Lie group is {\it simple} if its Lie algebra is simple. Note that a simple Lie group is not necessarily simple as a group. We use the following fact for simple Lie groups.

\begin{fact}\label{fact: simple}
A connected Lie group $G$ is a simple Lie group if and only if all its normal proper subgroups are discrete, and  contained in $Z(G)$. 
\end{fact}

Suppose $\mathfrak{g}$ is a finite dimensional Lie algebra. For $x \in \mathfrak{g}$, let $\text{ad}x: \mathfrak{g} \to \mathfrak{g}, y \mapsto [x,y]$. Then $\text{ad}$ is an endomorphism of $\mathfrak{g}$. The {\em Cartan--Killing form} of $\kappa_\mathfrak{g}: \mathfrak{g} \times \mathfrak{g} \to \RR $ is given by $$\kappa_\mathfrak{g}(x,y)=\text{tr}(\text{ad} x\  \text{ad} y).$$
The Cartan--Killing form is invariant under an automorphism of $\mathfrak{g}$; this follows from a direct computation and the fact that if $\rho$ is an automorphism of $\mathfrak{g}$, then $\mathrm{ad}\rho(x)=\rho\circ\mathrm{ad}x\circ\rho^{-1}$.
The following fact is from \cite[Lemma 5.5.8]{hilgert2011structure}.

\begin{fact} \label{fact: orthogonaldecomposition}
Suppose $\mathfrak{g}$ is a Lie algebra, $\kappa_\mathfrak{g}$ is the Cartan--Killing form of $\mathfrak{g}$, and $\mathfrak{h}$ is an ideal of $\mathfrak{g}$. 
Then the orthogonal space $\mathfrak{h}^\perp$ of $\mathfrak{h}$ with respect to $\kappa_\mathfrak{g}$ is also an ideal. 
If $\mathfrak{g}$ is semisimple, then $\mathfrak{g} = \mathfrak{h}\oplus \mathfrak{h}^\perp$ and $\kappa_\mathfrak{g}= \kappa_\mathfrak{h}\oplus \kappa_{\mathfrak{h}^\perp}$ where $\kappa_\mathfrak{h}$ and $\kappa_{\mathfrak{h}^\perp}$ are the Cartan--Killing form of $\mathfrak{h}$ and $\mathfrak{h}^\perp$.
\end{fact}

The following fact follows from \cite[Lemma 5.5.13]{hilgert2011structure}. It is also a consequence of Fact~\ref{fact: orthogonaldecomposition} and the alternative characterization of semisimple Lie algebras as those whose Cartan--Killing form is nondegenerate.

\begin{fact} \label{fact: preservationsemisimple}
Every ideal and quotient algebra of a semisimple Lie algebra is semisimple. Hence, every closed normal subgroup and quotient group of a semisimple Lie group is semisimple.
\end{fact}

The first and second assertions in the following fact are immediate consequences of Facts~\ref{fact: radical}, ~\ref{fact: preservationsolvable}, \ref{fact: preservationsemisimple}   

\begin{fact} \label{fact: centerless1}
If $G$ is a connected semisimple  Lie group, then its center $Z(G)$ is a finitely generated discrete group, the quotient map $\rho: G \to G/Z(G)$ is a covering map. 
\end{fact}

The following fact is a consequence of~\cite[Proposition 9.5.2 and Theorem 9.5.4]{hilgert2011structure}.

\begin{fact} \label{fact: centerless2}
If $G$ and $G'$ are connected   Lie groups, $\rho: G \to G'$ is covering map, $Z(G)$ and $Z(G')$ are the centers of $G$ and $G'$. Then we have $\ker \rho \leq Z(G)$ and $Z(G') = Z(G)/\ker \rho$.
\end{fact}

The first assertion in the following fact is known as Weyl's theorem on Lie groups with semisimple compact Lie algebra~\cite[Theorem 12.1.17]{hilgert2011structure}. 
\begin{fact}\label{fact: semisimple compact}
If $G$ is a connected semisimple Lie group with compact Lie algebra, then $G$ is compact and $Z(G)$ is finite.
\end{fact}

The following fact is a consequence of Fact~\ref{fact: semisimple compact} and the result in~\cite{Sirota}. This can also be proven directly using~\cite[Proposition 13.1.10 (ii)]{hilgert2011structure}; we thank Jinpeng An for pointing this out to us. 

\begin{fact}\label{fact: centersimpleLie}
 If $G$ is a simply connected simple Lie group, then the center $Z(G)$ of $G$ has rank at most $1$.
\end{fact}

Suppose $\mathfrak{g}$ is a finite dimensional Lie algebra with Cartan--Killing form $\kappa_\mathfrak{g}$. A Lie algebra automorphism $\tau$ of $\mathfrak{g}$  is a {\it Cartan involution} if $\tau^2=\id_{\mathfrak{g}}$ and $(x, y) \mapsto -\kappa_{\mathfrak{g}}(x, \tau(y))$ is a positive definite bilinear form. The following fact is~\cite[Theorem 13.2.10]{hilgert2011structure}
\begin{fact} \label{fact: ExistenceCartaninvolution}
Let $\mathfrak{g}$ be a semisimple Lie algebra. Then $\mathfrak{g}$ has a Cartan involution $\tau$.
\end{fact}

We refer the reader to~\cite[Section 6.4]{knapp2013lie} for the full definition of Iwasawa decomposition; we will need the following fact which is a consequence of~\cite[Theorem 6.31, Theorem 6.46]{knapp2013lie} and~\cite[Theorem 13.3.8]{hilgert2011structure}. 
\begin{fact}[Iwasawa decomposition] \label{fact: Lie group decomp Iwasawa} Suppose $G$ is a connected semisimple Lie group with Lie algebra  $\mathfrak{g}$, $\tau$ is a Cartan involution of $\mathfrak{g}$, $\mathfrak{k}$ the subalgebra of $\mathfrak{g}$ fixed by $\tau$,  and $\exp: \mathfrak{g} \to G$ is the exponential map. Then there is an Iwasawa decomposition $G=KAN$ such that the following holds:
\begin{enumerate}
    \item the multiplication map $$\Phi: K \times A \times N \rightarrow G: (k, a, n) \mapsto kan$$ is a diffeomorphism.
    \item $K= \exp(\mathfrak{k)}$ is a connected closed subgroup of $G$, $Z(G) \subseteq K$, and $K$ is a maximal compact subgroup of $G$ if $Z(G)$ is finite. 
    \item $A$ is an abelian closed subgroup of $G$, $N$ is a nilpotent closed subgroup of $G$, and both $A$ and $N$ are simply connected.
    \item  $Q=AN$, we have that $Q$ is a solvable closed subgroup of $G$, and $N\vartriangleleft Q$.
\end{enumerate}
\end{fact} 

The following fact is a consequence of the definition of Iwasawa decomposition in~\cite[Section 6.4]{knapp2013lie}.

\begin{fact} \label{fact: classification2}
If $G$ is a noncompact semisimple Lie group with Iwasawa decomposition $G=KAN$, then $AN$ has dimension at least $2$.
\end{fact}

\section*{Acknowledgements}
The authors would like to thank Ehud Hrushovski for introducing the authors to this problem, J\'ozsef Balogh, Anand Pillay, and Lei Zhang for helpful discussions, Takashi Satomi for pointing out several typos and errors, Richard Gardner, Vitali Milman, and Rolf Schneider for several historical remarks, and Guoxian Song for drawing the figure. Special thanks to Jinpeng An for suggesting many useful references, and for carefully reading the paper and pointing out an important error in the first version of the manuscript. Finally, we would like to thank the anonymous referee for a very detailed report with many useful suggestions.

\bibliographystyle{amsplain}
\bibliography{ref}

\end{document}